\documentclass[11pt]{amsart}
\usepackage{graphicx}

\usepackage{amsmath}
\usepackage{amsthm}
\usepackage{amsfonts}
\usepackage{amssymb}
\usepackage[usenames,dvipsnames]{color}
\usepackage{verbatim}
\usepackage{amscd}
\usepackage[latin1]{inputenc}
\usepackage{latexsym}

\usepackage{mdframed}
\usepackage{enumerate}
\usepackage{array}
\usepackage{commath}
\usepackage{mathtools}
\usepackage{hyperref}
\usepackage{multirow}
\usepackage{parcolumns}
\usepackage{multicol}
\usepackage{mathrsfs}
\usepackage{xfrac}
\usepackage{tabularx}
\usepackage[all]{xy}
\usepackage{tabu}
\usepackage{mathrsfs}
\usepackage{enumitem}
\usepackage{bbm}
\usepackage{appendix}

\usepackage[normalem]{ulem}

\usepackage{thmtools}
\usepackage{thm-restate}
\usepackage{cleveref}

\usepackage{bm}
\usepackage[many]{tcolorbox}
\usepackage[left=1.2in,right=1.2in,top=1.2in,bottom=1.2in]{geometry}

\usepackage{pgf,tikz}
\usetikzlibrary {positioning}
\usetikzlibrary{arrows}
\usetikzlibrary{calc}
\usepackage{capt-of}
\newcommand{\dist}{3}

\usepackage{color}
\usepackage{graphicx}
\usepackage{subfigure}
\usepackage{pdfsync}
\usepackage{centernot}

\usepackage{algorithm}
\usepackage{algpseudocode}

\numberwithin{equation}{section}

\DeclareMathOperator{\spp}{supp}

\DeclareMathOperator{\Id}{Id}

\DeclareMathOperator{\argmin}{argmin}
\DeclareMathOperator*{\argminunder}{argmin}
\DeclareMathOperator{\WL}{WL}

\DeclareMathOperator{\OTC}{OTC}

\DeclareMathOperator{\OGJ}{OGJ}

\DeclareMathOperator{\inte}{int}

\DeclareMathOperator{\dst}{dist}

\DeclareMathOperator{\pr}{pr}

\def\R{\mathbb{R}}

\newtheorem{theorem}{Theorem}
\numberwithin{theorem}{section}

\newtheorem{lemma}[theorem]{Lemma}

\newtheorem{remark}[theorem]{Remark}

\newtheorem*{theorem*}{Theorem}

\theoremstyle{definition}
\newtheorem{definition}[theorem]{Definition}
\newtheorem{example}[theorem]{Example}

\newtheorem*{assumptions*}{Standing Assumptions}

\usepackage{fancyhdr}
\definecolor{CharlotteGreen}{RGB}{0,80,53}
\definecolor{NinerGold}{RGB}{164,150,101}
\definecolor{PineGreen}{RGB}{137,144,100}

\newcommand\firstauthor{\thanks{*First author designation.}}

\newcommand{\phuong}[1]{{\color{cyan}#1}}
\newcommand{\km}[1]{{\color{blue}#1}}

\newcommand{\abn}[1]{{\color{red}#1}}

\begin{document}
\title{Optimal Graph Joining with Applications to \\ Isomorphism Detection and Identification}

\author{Phuong N. Ho\`ang* \and Kevin McGoff \and \\ Andrew B. Nobel \and Yang Xiang \and Bongsoo Yi}

\thanks{PNH and KM gratefully acknowledges support from the National Science Foundation grants DMS-2113676 and DMS-2413929. KM also gratefully acknowledges support from the National Science Foundation grant DMS-1847144. ABN, YX, and BY gratefully acknowledge support from the National Science Foundation grants DMS-2113676 and DMS-2413928.}
\firstauthor

\maketitle

\begin{abstract}
We introduce an optimal transport based approach for comparing undirected graphs with
non-negative edge weights and general vertex labels, and we
study connections between the resulting linear program and 
the graph isomorphism problem.
Our approach is based on the notion of a joining of two graphs $G$ and $H$, 
which is a product graph that preserves their marginal structure.
Given $G$ and $H$ and a vertex-based cost function $c$, the optimal graph joining (OGJ) 
problem finds a joining of $G$ and $H$ minimizing degree weighted cost. The OGJ problem
can be written as a linear program with a convex polyhedral solution set.
We establish several basic properties of the OGJ problem, and present theoretical 
results connecting the OGJ problem to the graph isomorphism problem. 
In particular, we examine a variety of conditions on graph families that are sufficient
to ensure that for every pair of graphs $G$ and $H$ in the family
(i) $G$ and $H$ are isomorphic if and only if their optimal joining 
cost is zero, and (ii) if $G$ and $H$ are isomorphic, the the extreme points
of the solution set of the OGJ problem are deterministic joinings corresponding
to the isomorphisms from $G$ to $H$.
\end{abstract}


\section{Introduction} 
\label{Sect:Intro}

Optimal transport provides a flexible and powerful approach to comparing probability measures that has found application in a variety of disciplines.
This paper introduces an optimal transport-based framework 
for comparing graphs and investigates
its application to the problem of detecting and identifying graph
isomorphisms.  We focus on undirected graphs with
non-negative edge weights, including graphs that possess 
vertex labels.
Our investigation is based on the notion of a graph joining, 
which is analogous to the notion of a coupling in the standard optimal transport problem.
Loosely speaking, a joining of graphs $G$ and $H$ 
with vertex sets $U$ and $V$ is a graph $K$ on $U \times V$ from which 
$G$ and $H$ can be recovered via coordinate projections. 
A precise definition
of graph joining is given in Section \ref{Sect:weightjoinings}.
As discussed in Section \ref{Sect:MCs}, the definition of graph joining is
based on probabilistic ideas, specifically reversible Markovian couplings
of random walks on the given graphs $G$ and $H$.

The family of joinings of two graphs $G$ and $H$ provides a means of 
quantifying and studying their potential interactions.  As in the
setting of optimal transport, one may discriminate among graph joinings 
of $G$ and $H$ on the basis of a cost function 
$c: U \times V \to [0,\infty)$ relating their vertices.
The optimal graph joining (OGJ) problem minimizes the degree-weighted average 
of $c(u,v)$ over all possible graph joinings of $G$ and $H$, where the degree of $(u,v)$ is the sum of the graph joining weights on the edges incident to $(u,v)$.
When graphs and their joinings are represented in terms of edge weight functions, the OGJ problem takes the form of a linear program whose solution set is a (non-empty) convex polytope. 

While the OGJ problem is based on probabilistic principles, we show
that it has close connections with the detection and identification 
of graph isomorphisms. Our analysis builds on the elementary fact that 
any isomorphism $f: U \to V$ between graphs $G$ and $H$ corresponds to a 
``bijective'' graph joining whose degree function is supported on 
pairs $(u,f(u))$, that is, the graph of $f$.
In more detail, let $\mathcal{G}$ be a family of graphs together 
with a scheme that assigns labels to the vertices of each graph 
in $\mathcal{G}$. For $G$ and $H$ in $\mathcal{G}$, let $\rho(G,H)$ be the 
minimizing value of the OGJ problem for the pair $(G,H)$ under the $0$-$1$ label-agreement cost 
$c(u,v) = \mathbb{I}(\mbox{label}(u) \neq \mbox{label}(v))$.
We say that OGJ \textit{detects} isomorphism for $\mathcal{G}$ if 
for all $G$ and $H$ in $\mathcal{G}$, we have $\rho(G,H) = 0$ if and only if $G$ and $H$ are isomorphic. 
Assuming that OGJ detects isomorphism for $\mathcal{G}$,  
we say that OGJ \textit{identifies} isomorphism if for all isomorphic graphs $G$ and $H$ in $\mathcal{G}$, 
the extreme points of the OGJ solution set are bijective graph joinings corresponding to the isomorphisms from $G$
to $H$.
Identification is similar to the property of \textit{convex exactness}, which arises in convex relaxations of the graph isomorphism problem 
(see Section \ref{Sect:relatedworks} for further discussion).

The primary theoretical results of the paper provide sufficient 
conditions on graph families $\mathcal{G}$ and their labeling schemes
under which OGJ detects and identifies isomorphism.
A brief summary of these results follows; more detailed  
statements can be found in Section \ref{Sect:theoresults}.
Initially we establish that OGJ detects and identifies isomorphism 
if for each graph in $\mathcal{G}$ its vertex labels 
are invariant under isomorphism, the labels assigned to the neighbors
of every vertex are distinct, and at least one vertex has a unique label.
We then show that positivity of the OGJ cost between two labeled graphs is
unaffected if for both graphs the label at each vertex
is replaced by the label process of the standard random walk 
beginning at that vertex.  
As a corollary, OGJ can detect 
and identify isomorphism for families of graphs with suitable vertex landmarks. 

The sufficient conditions of the detection and identification results above imply that each graph 
in $\mathcal{G}$ is asymmetric (i.e., has trivial automorphism group).
Asymmetry of graphs is assumed in many detection results 
in the literature; see Section \ref{Sect:relatedworks}.
We next consider structured graph families whose elements may be symmetric (i.e., have nontrivial automorphism groups).  
In particular we show that OGJ can detect and identify isomorphism 
for the family of weighted trees with vertex labels. 
We then establish an extension principle showing that 
if OGJ can detect and identify isomorphism for a family $\mathcal{G}$, then it can detect and identify isomorphisms for 
families obtained by gluing together graphs in $\mathcal{G}$ in an 
appropriate manner. 
As an application of the extension principle, we show that OGJ can detect and identify isomorphism for forests and certain graphs with loops. 
Finally, we show that OGJ has at least as much discriminating power as the classical Weisfeiler-Leman test. 

\vskip.1in

The notion of graph joining and the optimal graph joining problem introduced here extend ideas from earlier work concerning optimal transport for Markov chains and directed graphs \cite{o2022optimal, yi2024alignment}. 
Existing approaches to detection and identification of graph isomorphism 
may be roughly categorized 
into three groups: those employing algebraic techniques from group theory, those employing combinatorial techniques, and those based on convex relaxations of graph isomorphism (see Section \ref{Sect:relatedworks} for further discussion of related work).  
By contrast, our approach to graph isomorphism through optimal graph joinings is based on ideas from optimal transport and stochastic processes.  
The OGJ approach automatically integrates local information captured
by the vertex-based cost function (the objective function of the 
OGJ problem) with global information captured by the family of graph joinings 
(the feasible set of the OGJ problem).
The OGJ approach is conceptually straightforward and flexible.  It applies without modification to graphs with non-negative edge weights, and to graphs with vertex labels. 
Our theoretical results guarantee that OGJ can detect and identify isomorphism 
for a variety of graph families, including families with highly symmetric graphs,
for which existing detection and identification methods may not have theoretical guarantees.

\vskip.1in

\noindent
{\bf Organization of the paper.}
The next section introduces graph joinings, including several examples 
and basic properties, and the optimal graph joining problem.
Section \ref{Sect:IsomorphismIntro} is devoted to background and preliminary results concerning graph isomorphism.
The main theoretical results of the paper, connecting the OGJ problem and graph isomorphisms, are presented in Section \ref{Sect:theoresults}.
Section \ref{Sect:relatedworks} provides a discussion of related work. 
Additional properties of the OGJ problem, as well as the proofs of the main 
theoretical results, are given in Sections \ref{Sect:mainweight} - \ref{Sect:graphisomorphisms}.
\section{Graph joinings and the optimal graph joining problem} 
\label{Sect:generaltheory}

This section introduces graph joinings and the optimal graph joining (OGJ) problem. We begin with preliminary definitions and notation.

\subsection{Notation and graphs of interest} 
\label{Sect:background}

Throughout this paper we study finite undirected graphs whose edges may be weighted, and whose vertices may be labeled.   
Graphs may contain self-loops, but multi-edges are excluded.  In what follows the support of a function 
$f: S \to [0,\infty)$ is defined by $\spp(f) = \{s \in S : f(s) > 0 \}$. 

\begin{definition} \label{Def:weightfunctions}
Let $U$ be a finite set.  A function $\alpha: U \times U \to \mathbb{R}$ is a {\em weight function} on $U$ if it satisfies the 
following three properties:
\begin{enumerate}
\item (Non-negativity) $\alpha(u, u') \ge 0$, for all $u, u' \in U$;
\vskip.05in
\item (Symmetry) $\alpha(u, u') = \alpha(u', u)$, for all $u, u' \in U$;
\vskip.05in
\item (Normalization) $\sum_{u, u' \in U} \alpha(u, u') = 1$.
\end{enumerate}
\end{definition}

\vskip.05in

\begin{definition}
\label{Def:weighteddegree}
A weight function $\alpha$ on $U$ has {\it marginal function} $p: U \to \mathbb{R}$ defined by $p(u) = \sum_{u' \in U} \alpha(u, u')$.
\end{definition} 

Note that any marginal function $p$ is non-negative and satisfies $\sum_{u \in U} p(u) = 1$ and $p(u) = \sum_{u' \in U} \alpha(u', u)$.  In particular, $p$ is a probability
mass function on $U$.

\vskip.05in

\begin{definition}
\label{Def:uwlg}
Let $U$ be a finite vertex set and $\mathcal{L}$ be a (possibly infinite) label set.  A weighted undirected labeled graph is
defined by a triple $G = (U,\alpha, \phi_G)$, where $\alpha$ is a weight function on $U$
and $\phi_G: U \to \mathcal{L}$ is referred to as a vertex label function. 
\end{definition} 

Under the definition, the edges in $G$ are described by ordered pairs:  the pair $(u,u')$ represents an edge from $u$ to $u'$,
which is assigned weight $\alpha(u,u')$.  Symmetry of $\alpha$
ensures that the weight assigned to $(u,u')$ is the same as the weight assigned to $(u',u)$, 
and in this sense the graph $G$ is undirected.  The edge set $E(G) := \mbox{supp}(\alpha)$ of $G$ consists of 
all edges with positive weight. For a vertex $u \in U$, the neighborhood of $u$ is the set $N_G(u) = \{u' : \alpha(u,u') >0\}$, and any element of $N_G(u)$ is referred to as a neighbor of $u$.  The marginal function $p$ of $\alpha$ coincides with the weighted
degree function of $G$; in particular $p(u)$ is the sum of the weights of all edges attached to the vertex $u$.  
Here and throughout we say that a graph $G$ is \textit{fully supported} if its marginal function is strictly positive. 
We also recall that a graph $G = (U,\alpha,\phi_G)$ is \textit{connected} if for each pair of distinct vertices $u$ and $u'$ in $U$, there exists $u_0,\dots,u_n \in U$ such that $u_0 = u$, $u_n = u'$, and $\alpha(u_k,u_{k+1}) >0$ for each $k \in \{0,\dots,n-1\}$.

The map $\phi_G$ appearing in the specification of $G$ will be called the {\em primary label function} of $G$. 
Whereas the weight function $\alpha$ specifies the connectivity between vertices, the primary label function may capture 
features of the vertices. 
Unlabeled graphs are a special case of labeled graphs in which the label space $\mathcal{L}$ may be taken to be a singleton. 
An unweighted graph 
can be represented as a graph with constant edge weights.
The set $\mathcal{L}$ will be referred to as the label space of $G$.  In the sequel, when considering families
of graphs, we will assume that they have a common label space. See Figure \ref{Fig:primarylabels} for two examples.

    \begin{figure}[h!]
    \centering
    
    \begin{tikzpicture}

    \draw[line width=0.5mm] (0,0) -- (1,0) -- (2,0) -- (3,0);
    
    \foreach \y in {0} {
      \foreach \x in {0, 1, 2, 3} {
        \node at (\x,\y) [circle,fill=black,inner sep=0pt,minimum size=6pt, color=black] {};
      }
    }

    \node at (0,-0.4) {$0$};
    \node at (1,-0.4) {$1$};
    \node at (2,-0.4) {$2$};
    \node at (3,-0.4) {$2$};
    
    \node at (0.5,0.3) {$\sfrac{1}{3}$};
    \node at (1.5,0.3) {$\sfrac{1}{12}$};
    \node at (2.5,0.3) {$\sfrac{1}{12}$};
    
    \node at (1.5,-1) {\text{(a)}};
    
    \draw[line width=0.5mm] (7,0) -- (9,0) -- (8,1.8) -- (7,0);
    
    \foreach \y in {0} {
      \foreach \x in {7, 9} {
        \node at (\x,\y) [circle,fill=black,inner sep=0pt,minimum size=6pt, color=black] {};
      }
    }
    
    \node at (8,1.8) [circle,fill=black,inner sep=0pt,minimum size=6pt, color=black] {};
    
    \node at (7,-0.4) {}; 
    \node at (9,-0.4) {}; 
    \node at (8,2.2) {}; 
    
    \node at (8,-0.3) {$\sfrac{1}{8}$};
    \node at (7.2,1) {$\sfrac{1}{8}$};
    \node at (8.8,1) {$\sfrac{1}{4}$};
    
    \node at (8,-1) {\text{(b)}};

    \end{tikzpicture}

    \caption{(a) A labeled graph with vertex labels in $\{0,1,2\}$ and edge weights in $\{\tfrac{1}{3}, \tfrac{1}{12}\}$, 
    and (b) an unlabeled graph with edge weights in $\{\tfrac{1}{4}, \tfrac{1}{8}\}$. Note that the edge weights drawn here are equal to the weight function value $\alpha(u,u')$ for each ordered pair $(u,u')$ corresponding to an edge.}
    \label{Fig:primarylabels}
\end{figure}
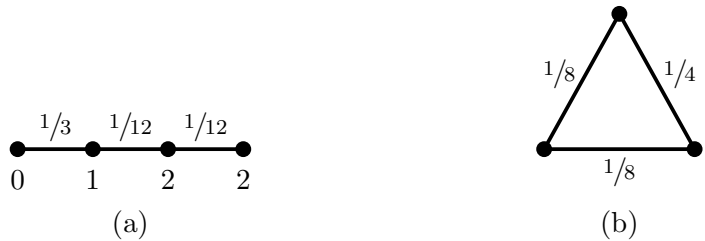

\subsection{Weight joinings and graph joinings} 
\label{Sect:weightjoinings}
In this section we introduce graph joinings and state some of their basic properties. 

\begin{definition}
\label{Def:weightjoinings}
Let $\alpha$ be a weight function on a finite set $U$ with marginal $p$, and let $\beta$ be a weight function on a finite set $V$ with marginal $q$. 
A weight function $\gamma$ on $U \times V$ with marginal function $r_{\gamma} : U \times V \to [0,1]$ defined by
$$r_\gamma (u,v) = \sum_{u' \in U, v' \in V} \gamma((u,v),(u',v'))$$
is a {\it weight joining} of $\alpha$ and $\beta$ if it satisfies the following conditions.
    \begin{enumerate}
    \vskip.07in
        \item (Marginal coupling condition) The marginal function $r_\gamma$ is a coupling of $p$ and $q$, i.e.,
            \begin{align*} 
                \sum_{v \in V} r_\gamma(u,v) & = p(u), \mbox{ for all $u \in U$}, \; \text{and} \\
                \sum_{u \in U} r_\gamma(u,v) & = q(v),  \mbox{ for all $v \in V$}.
            \end{align*}
                \vskip.04in
        \item (Transition coupling condition) For all $u, u' \in U$ and $v, v' \in V$, 
            \begin{align*} 
                p(u) \displaystyle \sum_{\tilde{v} \in V} \gamma((u,v),(u',\tilde{v})) & = \alpha(u,u') \, r_\gamma(u,v), \; \text{and} \\
                q(v) \displaystyle \sum_{\tilde{u} \in U} \gamma((u,v),(\tilde{u},v')) & = \beta(v,v') \, r_\gamma(u,v).
            \end{align*}
    \end{enumerate}
        \vskip.04in
In what follows we will denote the set of all weight joinings of $\alpha$ and $\beta$ by $\mathcal{J}(\alpha, \beta)$. 
\end{definition}

The marginal coupling and transition coupling conditions ensure that the weight functions $\alpha$ and $\beta$ are, in an appropriate sense, marginals of $\gamma$ 
(see Proposition \ref{Prop:realcoupling}). 
The definition of weight joining is based on reversible Markovian couplings of random walks on $G$ and $H$; see Section \ref{Sect:MCs} below for further discussion.
The family $\mathcal{J}(\alpha,\beta)$ of weight joinings is non-empty, as it contains the product weight joining $\gamma = \alpha \otimes \beta$, 
defined for all $u, u' \in U$ and $v, v' \in V$ by 
	\begin{align*} 
		\gamma((u,v),(u',v')) = \alpha(u,u') \beta(v,v').
	\end{align*}

If $|U| = m$ and $|V| = n$, then a weight joining $\gamma$ can be written as an $mn \times mn$ real matrix. 
The following result, whose proof is in Section \ref{Sect:mainweight}, is straightforward. 
Recall that a convex polyhedron is a subset of a Euclidean space that can be written as an intersection of a finite number of closed half-spaces. The proof of this proposition appears in Section \ref{Sect:mainweight}.

\begin{restatable}[]{proposition}{basicproperties}
\label{Prop:basicproperties}
The set $\mathcal{J}(\alpha, \beta)$ is a nonempty, compact, convex polyhedron in $\mathbb{R}^{mn \times mn}$. 
\end{restatable}

\vskip.1in

We now turn to the notion of graph joinings. 

\begin{definition}\label{Def:graphjoinings}
A  {\em graph joining}  of graphs $G = (U, \alpha, \phi_G)$ and $H = (V,\beta,\phi_H)$ is any graph of the form
$K = (U \times V, \gamma, \phi_G \otimes \phi_H)$,
where $\gamma \in \mathcal{J}(\alpha,\beta)$ and 
$(\phi_G \otimes \phi_H)(u,v) = (\phi_G(u), \phi_H(v))$.
\end{definition} 

Informally, a graph joining of $G$ and $H$ is a graph on the product of their vertex sets that admits $G$ and $H$ as marginal factors (see Proposition \ref{Prop:realcoupling} for more details).  
Graph joinings preserve edges in the sense that if $((u,v),(u',v'))$ 
is an edge in $K$, then $(u,u')$ is an edge in $G$ and $(v,v')$ 
is an edge in $H$.
See Figure \ref{Fig:exagraphjoining} for examples.

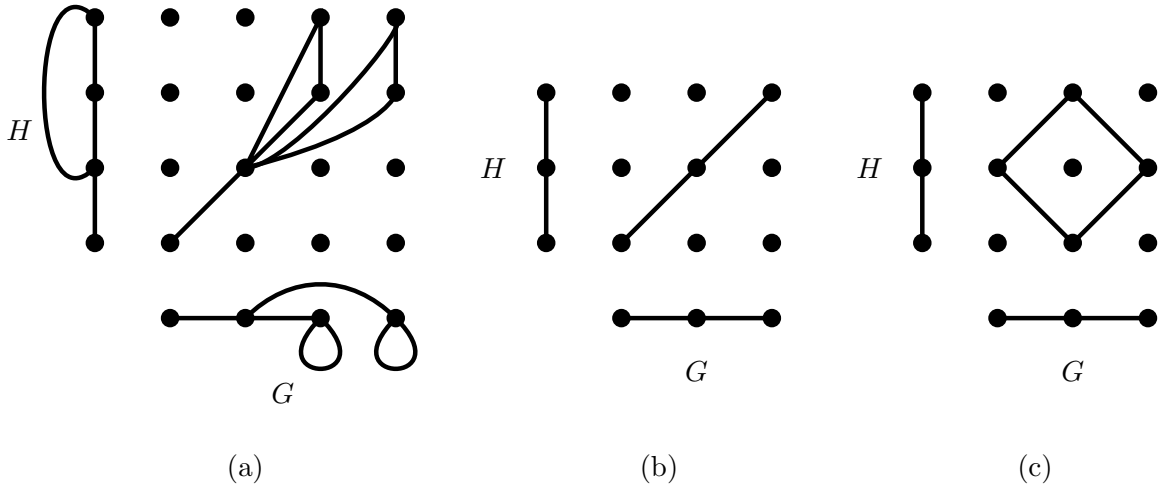
\begin{figure}[h!]
    \centering
    \begin{tikzpicture}
        \draw[line width=0.6mm] (0,0) -- (1,0) -- (2,0);    
        \draw[line width=0.6mm] (3,0) .. controls ++ (0.3*\dist,-0.3*\dist) and ++ (-0.3*\dist,-0.3*\dist) .. (3,0);
        \draw[line width=0.6mm] (2,0) .. controls ++ (0.3*\dist,-0.3*\dist) and ++ (-0.3*\dist,-0.3*\dist) .. (2,0);
        \draw[line width=0.6mm] (1,0) .. controls ++ (0.2*\dist,0.2*\dist) and ++ (-0.2*\dist,0.2*\dist) .. (3,0);

        \draw[line width=0.6mm] (-1,4) -- (-1,3) -- (-1,2) -- (-1,1);
        \draw[line width=0.6mm] (-1,2) .. controls ++ (-0.3*\dist,-0.3*\dist) and ++ (-0.3*\dist,0.3*\dist) .. (-1,4);

        \draw[line width=0.6mm,color=black] (0,1) -- (1,2) -- (2,3) -- (2,4) -- (1,2);
        \draw[line width=0.6mm,color=black] (3,3) -- (3,4);
        \draw[line width=0.6mm,color=black] (1,2) .. controls ++ (0.8*\dist,0.2*\dist) and ++ (0*\dist,0.1*\dist) .. (3,3);
        \draw[line width=0.6mm,color=black] (1,2) .. controls ++ (0.2*\dist,0*\dist) and ++ (0.1*\dist,0.0*\dist) .. (3,4);

        \foreach \x in {0,1,2,3} {
        \foreach \y in {0} {
            \node at (\x,\y) [circle,fill=black,inner sep=0pt,minimum size=7pt, color=black] {};
        }
        }

        \foreach \x in {0, 1, 2, 3} {
        \foreach \y in {1, 2, 3, 4} {
            \node at (\x,\y) [circle,fill=black,inner sep=0pt,minimum size=7pt] {};
        }
        }

        \foreach \y in {1, 2, 3, 4} {
        \foreach \x in {-1} {
            \node at (\x,\y) [circle,fill=black,inner sep=0pt,minimum size=7pt,color=black] {};
        }
        }
 
        \node at (1.5,-1) {$G$};
        \node at (-2,2.5) {$H$}; 
        \node at (1,-2) {\text{(a)}};


        \draw[line width=0.6mm] (6,0) -- (7,0) -- (8,0);
        \draw[line width=0.6mm] (5,1) -- (5,2) -- (5,3);
        \draw[line width=0.6mm, color=black] (6,1) -- (7,2) -- (8,3);

        \foreach \x in {6, 7, 8} {
            \foreach \y in {0} {
            \node at (\x,\y) [circle,fill=black,inner sep=0pt,minimum size=7pt, color=black] {};
            }
        }

        \foreach \x in {6, 7, 8} {
            \foreach \y in {1, 2, 3} {
            \node at (\x,\y) [circle,fill=black,inner sep=0pt,minimum size=7pt] {};
            }
        }

        \foreach \x in {5} {
            \foreach \y in {1, 2, 3} {
            \node at (\x,\y) [circle,fill=black,inner sep=0pt,minimum size=7pt] {};
            }
        }

        \node at (7,-0.7) {$G$};
        \node at (4.3,2) {$H$};


        \draw[line width=0.6mm] (11,0) -- (12,0) -- (13,0);
        \draw[line width=0.6mm] (10,1) -- (10,2) -- (10,3);
        \draw[line width=0.6mm, color=black] (11,2) -- (12,3) -- (13,2) -- (12,1) -- (11,2);

        \foreach \x in {11, 12, 13} {
            \foreach \y in {0} {
            \node at (\x,\y) [circle,fill=black,inner sep=0pt,minimum size=7pt, color=black] {};
            }
        }

        \foreach \x in {10} {
            \foreach \y in {1, 2, 3} {
            \node at (\x,\y) [circle,fill=black,inner sep=0pt,minimum size=7pt] {};
            }
        }

        \foreach \x in {11, 12, 13} {
            \foreach \y in {1, 2, 3} {
            \node at (\x,\y) [circle,fill=black,inner sep=0pt,minimum size=7pt] {};
            }
        }

        \node at (12,-0.7) {$G$};
        \node at (9.3,2) {$H$};
        \node at (6.5,-2) {\text{(b)}};
        
         \node at (11.5,-2) {\text{(c)}};
    \end{tikzpicture}
    \caption{Illustrations of three graph joinings, with the vertex sets of $G$ and $H$ drawn on the horizontal and vertical axes, respectively, and a graph joining drawn on the grid representing the product of their vertex sets. Edge weights and vertex labels are not drawn. Note that (b) and (c) depict two distinct graph joinings of the same pair of marginal graphs.} 
    \label{Fig:exagraphjoining}
\end{figure}


\subsubsection{Connections with Markov chains and ergodic theory} 
\label{Sect:MCs}


Weighted undirected graphs are intimately connected with 
reversible Markov chains.
Let $\alpha$ and $\beta$ be weight functions on finite sets $U$ and $V$ with marginals $p$ and $q$, respectively.
Together $\alpha$ and $p$ determine a unique 
stationary Markov chain $X = X_0, X_1, \ldots$ with state space $U$, transition probability matrix
\begin{align*}
P(u' \mid u) = \left\{ \begin{array}{lll}
         \dfrac{\alpha(u,u')}{p(u)} &\text{ if } p(u) > 0, \\[.1in]
         0 &\text{ otherwise,}
    \end{array} \right.
\end{align*}
and stationary distribution $p$.  The Markov chain $X$ describes 
the standard random walk on $G$, which moves from each vertex $u$ to a neighboring vertex $u'$ with probability proportional
to $\alpha(u,u')$.
The symmetry of $\alpha$ ensures that the process $X$ is reversible.
Let $Y = Y_0, Y_1, \ldots$ be the reversible Markov 
chain on $V$ determined by $\beta$ and $q$,
with transition probability matrix $Q(v' \mid v)$.
In the same fashion, a joining $\gamma \in \mathcal{J}(\alpha,\beta)$ with marginal $r_\gamma$ determines a reversible Markov chain 
$(\widetilde{X},\widetilde{Y}) = (\widetilde{X}_0,\widetilde{Y}_0), (\widetilde{X}_1,\widetilde{Y}_1), \ldots$ taking values in $U \times V$ with transition probability matrix
\begin{align*}
R((u',v') \mid (u,v)) 
= \left\{ \begin{array}{lll}
         \dfrac{\gamma((u,v), (u',v'))}{r_\gamma(u,v)} &\text{ if } r_\gamma(u,v) > 0, \\[.1in]
         0 &\text{ otherwise,}
    \end{array} \right.
\end{align*}
and stationary distribution $r_\gamma$. 
The marginal coupling condition ensures that the stationary distribution of $(\widetilde{X},\widetilde{Y})$ is a coupling of the stationary distributions of $X$ and $Y$. 
Furthermore, the transition coupling condition ensures that if 
$r_{\gamma}(u,v) >0$, then the transition distribution 
$R(\cdot \mid (u,v))$ is a coupling of $P(\cdot \mid u)$ and 
$Q(\cdot \mid v)$. 
In the language of Markov chains \cite{levin2017markov},
$(\widetilde{X},\widetilde{Y})$ is a reversible Markovian coupling of $X$ and $Y$.  In graph terms, the processes $X$, $Y$, and $(\widetilde{X},\widetilde{Y})$ are
the standard random walks on the graphs $(U,\alpha)$, $(V,\beta)$,  and $(U \times V, \gamma)$, respectively (and we note that vertex labels do not play a role).

The term joining was introduced by Furstenberg \cite{furstenberg1967disjointness} to refer to a stationary coupling of stationary processes.  Joinings have subsequently played an important 
role in ergodic theory; we refer the interested reader to 
\cite{de2005introduction, glasner2003ergodic} for more details.
While the random walk associated with a graph joining $K$ of graphs $G$ and $H$ is a stationary coupling of the random walks on $G$ and $H$, our definition of graph joining imposes additional constraints.  In particular not every stationary coupling of the random walks on $G$ and $H$ will correspond to
a graph joining in the sense defined here (since it need not be a reversible Markovian coupling).


\subsection{Optimal graph joining} 
\label{Sect:optimaljoining}

Given a vertex-based cost function, one may seek to minimize the expected value of the cost over the space of graph joinings.  
Given two functions $f, g: A \to \mathbb{R}$ defined on a finite set $A$, let $\langle f,g \rangle = \sum_{x \in A} f(x)g(x)$. 

\vskip.1in

\begin{definition} 
\label{Def:OGJproblem}
Let $G = (U, \alpha, \phi_G)$ and $H = (V, \beta, \phi_H)$ be graphs, and let
$c : U \times V \to \mathbb{R}_{\geq 0}$ be a cost function relating the
vertices of $G$ and $H$. 
The {\it optimal graph joining (OGJ) problem} for $G$ and $H$ is 
\begin{align*}
    &\text{minimize }  \langle c, r_{\gamma} \rangle \\[.04in]
    &\text{subject to } \gamma \in \mathcal{J}(\alpha, \beta),
\end{align*} 
where $r_\gamma$ denotes the marginal function of $\gamma$.
Denote the minimum transport cost by  
\vspace*{.08in}
\[
\rho(G,H) = \min_{\gamma \in \mathcal{J}(\alpha,\beta)} \langle c, r_{\gamma} \rangle,
\]
and the set of optimal solutions by
\vspace*{.04in}
\begin{align*}
    \mathcal{J}^*(\alpha, \beta) = \argminunder_{\gamma \in \mathcal{J}(\alpha,\beta)} \langle c, r_{\gamma} \rangle.
\end{align*}
\end{definition}

The minimum transport cost $\rho(G,H)$ may be viewed as a measure of the difference between $G$ and $H$ with respect to the cost function $c$. 
Any optimal weight joining $\gamma \in \mathcal{J}^*(\alpha, \beta)$ provides a transport plan from $G$ to $H$ that  minimizes the
expected cost over all such plans.  
As the objective function $\gamma \mapsto \langle c, r_\gamma \rangle$ is linear, 
Proposition \ref{Prop:basicproperties} gives the following result. The proof of this proposition can be found in Section \ref{Sect:propertiesOGJ}. 

\begin{restatable}[]{proposition}{linearOGJ}
\label{Prop:linearOGJ}
The OGJ problem is a linear program and always has at least one solution.  As such, a solution of the OGJ problem can be found in polynomial time. 
\end{restatable}
As graph joinings can be interpreted in terms of random walks (as in Section \ref{Sect:MCs}), the OGJ problem can be viewed as a constrained optimal transport problem for stochastic processes. 
Indeed, if $X$ and $Y$ are the reversible Markov chains associated
with the graphs $G$ and $H$, respectively, then 
the OGJ problem is to minimize $\mathbb{E} c(\widetilde{X}_0,\widetilde{Y}_0)$ over 
all reversible Markovian couplings $(\widetilde{X},\widetilde{Y})$ of
$X$ and $Y$. In the special case that the cost function takes the form $c(u,v) = \mathbb{I}(\psi_G(u) \neq \psi_H(v))$ for some labeling functions $\psi_G$ and $\psi_H$, the expected cost $\mathbb{E} c (\widetilde{X}_0,\widetilde{Y}_0)$ can be interepreted as the long-run average amount of time that the coupled chains spend in states with different labels.  

We note that the OGJ problem described above differs from the problem of finding an optimal transition coupling of {\em directed} graphs, as studied in \cite{o2022optimal, yi2024alignment}. 
Indeed, the symmetry requirement on weight joinings places nontrivial constraints on the set of feasible couplings, with a substantial impact on the associated optimization problem. 
\section{Connections with graph isomorphism: Background and preliminary results} \label{Sect:IsomorphismIntro}

In this section we provide some background and preliminary results connecting the OGJ problem with the problems of detecting and identifying graph isomorphism. 
To begin, we recall the definition of isomorphism for labeled graphs. 

\begin{definition} \label{Def:graphisomorphism}
Labeled graphs $G = (U,\alpha, \phi_G)$ and $H = (V,\beta, \phi_H)$ are {\it isomorphic},
written $G \cong H$, if there is function $f : U \to V$, called a graph isomorphism, such that
    \begin{enumerate}
        \item $f$ is bijective,
        \item for each pair $(u,u') \in U \times U$, we have $\alpha(u,u') = \beta(f(u),f(u'))$, and
        \item for each $u \in U$, we have $\phi_G(u) = \phi_H(f(u))$.
    \end{enumerate}
In words, $G$ and $H$ are isomorphic if there is a bijective correspondence between their vertices that preserves 
both edge weights and  vertex labels. 
\end{definition}

\subsection{Labeling schemes} 
\label{Sect:labelingsch}
Let $\mathcal{G}$ be a family of labeled graphs. 
A {\it labeling scheme} for $\mathcal{G}$ is an indexed family of functions
$\Upsilon = \{ \eta_G : G \in \mathcal{G} \}$ such that $\eta_G: U \to \mathcal{A}$, where $U$ is the vertex set of $G$,
and $\mathcal{A}$ is a label set that is the same for every $G$.
The {\it primary labeling scheme} of $\mathcal{G}$ is the family $\Phi = \{ \phi_G : G \in \mathcal{G} \}$ 
of primary label functions associated with graphs in $\mathcal{G}$, with label space $\mathcal{L}$. 
In studying connections between optimal graph joinings and graph isomorphisms, we will also allow a secondary label 
function $\phi_G': U \to \mathcal{L}'$ for each graph 
$G \in \mathcal{G}$, giving rise to a secondary labeling scheme 
$\Phi' = \{ \phi_G' : G \in \mathcal{G} \}$ with label space $\mathcal{L}'$.
In general, $\mathcal{L}'$ may be different from $\mathcal{L}$. 

\begin{example} \label{Exa:lblscheme}
Let $\mathcal{G}$ be a family of graphs. 
Below are some examples of labeling schemes on $\mathcal{G}$. 
\begin{enumerate}
      
\vskip.1in
       
\item (Identity labeling scheme) Suppose that for each $G = (U, \alpha, \phi_G) \in \mathcal{G}$, the vertex set $U$ is a subset of $\mathbb{N}$. Then the restriction of the identity function to $U$, denoted $\Id|_{U} : U \to \mathbb{N}$, is a label function. The corresponding labeling scheme 
is referred to as the {\it identity labeling scheme}. 
  
\vskip.1in
     
\item (Discrete degree labeling scheme) For a graph $G$ with vertex set $U$, for each $u \in U$, let $\deg(u)$ denote the discrete degree of $u$, i.e., the number of edges incident to $u$. For each $G = (U, \alpha, \phi_G) \in \mathcal{G}$, define a label function $\phi_G' : U \to \mathbb{N}$ by setting $\phi_G'(u) = \deg(u)$. Then $\Phi' = \{ \phi_G' : G \in \mathcal{G} \}$ is a labeling scheme, called the {\it discrete degree labeling scheme}. 

\vskip.1in
 
\item (Multiweight labeling scheme) For $G = (U, \alpha, \phi_G) \in \mathcal{G}$ define a label function $\phi_G'$ by 
        \begin{align*} 
            \phi_G'(u) = \{ \! \{ \alpha(u,u') : u' \in N_G(u) \} \! \},
        \end{align*}
where $N_G(u) = \{u' \in U : \alpha(u,u')  > 0 \}$ denotes the neighbors of $u$, and the double brackets denote a multiset. 
The labeling scheme
$\Phi'$ is called the multiweight labeling scheme.
    \end{enumerate} 
    See Figure \ref{Fig:secondarylabel1} for an illustration. 
\end{example}

\vskip.1in

\begin{figure}[h!]
    \centering
    
    \begin{tikzpicture}

    \draw[line width=0.5mm] (0,0) -- (1,0) -- (2,0) -- (3,0);
    
    \foreach \y in {0} {
      \foreach \x in {0, 1, 2, 3} {
        \node at (\x,\y) [circle,fill=black,inner sep=0pt,minimum size=6pt, color=black] {};
      }
    }

    \node at (0,-0.4) {$1$};
    \node at (1,-0.4) {$2$};
    \node at (2,-0.4) {$3$};
    \node at (3,-0.4) {$4$};
    
    \node at (0.5,0.3) {$\sfrac{1}{3}$};
    \node at (1.5,0.3) {$\sfrac{1}{12}$};
    \node at (2.5,0.3) {$\sfrac{1}{12}$};

    \end{tikzpicture}

    \caption{A labeled graph with primary labels $1,2,3,4$. The discrete degree label function takes values $1, 2, 2, 1$, and the multiweight label function takes values: $\{ \! \{ \tfrac{1}{3} \} \! \}$, $\{  \! \{ \tfrac{1}{3},\tfrac{1}{12}\} \!  \}$, $\{ \!  \{ \tfrac{1}{12}, \tfrac{1}{12} \} \!  \}$, and $\{ \!  \{ \tfrac{1}{12} \} \!  \}$.}
    \label{Fig:secondarylabel1}
\end{figure}
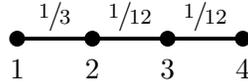

Combining the primary and secondary label functions of a graph $G \in \mathcal{G}$ yields 
an \textit{augmented label function} $\psi_G(u) = (\phi_G(u), \phi_G'(u))$ with values in $\mathcal{L} \times \mathcal{L}'$.
Let 
\[
\Psi =  \{ \psi_G = (\phi_G, \phi_G') : G \in \mathcal{G} \}
\]
be the resulting augmented labeling scheme, and let $\hat{G} = (U, \alpha, \psi_G)$ be the augmented version of $G$.  
It follows from the definitions above that if the augmented graphs $\hat{G}$ and $\hat{H}$ are isomorphic then so too are the original graphs $G$ and $H$ (with only the primary labels). However, the reverse implication may not hold:
there may be isomorphic graphs $G$ and $H$ such that $\hat{G}$ is \textit{not} isomorphic to $\hat{H}$,  
as an isomorphism from $G$ to $H$ may not preserve augmented labels. To address this situation, we make the following definition.

\begin{definition} \label{Def:invariantiso}
A labeling scheme $\{ \eta_G : G \in \mathcal{G} \}$ for $\mathcal{G}$ is {\it isomorphism-invariant} if for all $G$ and $H$ in $\mathcal{G}$ and any isomorphism $f$ from $G$ to $H$ we have $\eta_H \circ f = \eta_G$. 
\end{definition}

\begin{example} \label{Exa:invariantscheme}
The discrete degree and multiweight labeling schemes in Example \ref{Exa:lblscheme} are always isomorphism-invariant, but 
the identity labeling scheme is not isomorphism-invariant in general. 
\end{example}

The definition of isomorphism ensures that the primary labeling scheme $\Phi$ is isomorphism-invariant.  
The next lemma follows readily from the definitions above. The proof this lemma appears in Section \ref{Sect:graphisomorphisms}. 

\begin{restatable}[]{lemma}{secondaug}
\label{Lem:secondaug}
Let $\mathcal{G}$ be a family of graphs with primary, secondary, and augmented labeling schemes as above. Then the following are equivalent.
\begin{enumerate}
    \item The secondary labeling scheme $\Phi'$ is isomorphism-invariant.
    \item The augmented labeling scheme $\Psi$ is isomorphism-invariant.
    \item For all $G = (U,\alpha,\phi_G)$ and $H = (V,\beta,\phi_H)$ in $\mathcal{G}$, a map $f : U \to V$ is an isomorphism from $G$ to $H$ if and only if $f$ is an isomorphism from $\hat{G}$ to $\hat{H}$. 
\end{enumerate}
\end{restatable}

\subsection{Detection and identification of graph isomorphisms}

Here we define what is meant by detection and identification of isomorphism for a graph family $\mathcal{G}$ 
with augmented labeling scheme $\Psi$. We begin by describing the cost function induced by the augmented labeling scheme. 
For graphs $G = (U, \alpha, \phi_G)$ and $H = (V, \beta, \phi_H)$ in $\mathcal{G}$  with augmented label functions $\psi_G$ and $\psi_H$, 
we define the \textit{label-based binary cost function} by
\begin{equation*}
c_\Psi(u,v) = \mathbb{I}( \psi_G(u) \neq \psi_H(v)) = \left\{ \begin{array}{ll}
1, & \text{ if } \,  \psi_G(u) \neq \psi_H(v) \\
0, & \text{ otherwise}.
\end{array}
\right.
\end{equation*}
In particular, $c_\Psi(u,v) = 0$ if and only if
$u$ and $v$ have identical primary and secondary labels.
Let $\rho_\Psi (G,H)$ denote the OGJ transport cost between $G$ and $H$ with cost function $c_\Psi$, 
and let $\mathcal{J}_\Psi^*(\alpha,\beta)$ be the associated family of optimal weight joinings. The proof of this proposition can be found in Section \ref{Sect:graphisomorphisms}. 

\begin{restatable}[]{proposition}{necessary}
\label{Prop:necessary}
Let $\mathcal{G}$ be a graph family with augmented labeling scheme $\Psi$.  If $\Psi$ is 
isomorphism-invariant, then for any $G$ and $H$ in $\mathcal{G}$, 
\[
G \cong H \Longrightarrow \rho_{\Psi}(G,H) = 0.
\]
\end{restatable}

In general, the reverse implication does not hold. 
For example, consider two unlabeled, unweighted graphs that are
not isomorphic but have the property that the vertices of both
graphs all have the same
discrete degree. If the discrete degree labeling function is used as the secondary labeling scheme, then 
the cost function $c_{\Psi}$ will be identically zero, and the 
OGJ transport cost will also be zero, despite the fact that the graphs are not isomorphic.

\begin{definition} \label{Def:detection} 
Let $\mathcal{G}$ be a graph family with augmented labeling scheme $\Psi$. We say that OGJ with cost $c_\Psi$ {\em detects isomorphisms} for $\mathcal{G}$ if for all $G, H \in \mathcal{G}$, 
    \begin{equation*}
    \rho_\Psi (G,H) = 0 \quad \iff \quad  G \cong H.
    \end{equation*}
\end{definition}

Our study of graph isomorphism begins with the study of bijective weight joinings.

\begin{definition} \label{Def:deterbijective}
Let $\alpha$ and $\beta$ be weight functions on $U$ and $V$ respectively. 
A joining $\gamma \in \mathcal{J}(\alpha,\beta)$ with marginal $r_{\gamma}$ is {\em bijective} if (i) for 
each $u \in U$, there is a unique $v \in V$ such that $r_{\gamma}(u,v)>0$, 
and (ii) for each $v \in V$, there is a unique $u \in U$ such that $r_{\gamma}(u,v)>0$. 
\end{definition}

If $\gamma \in \mathcal{J}(\alpha,\beta)$ is bijective, then one may define a bijection $f_{\gamma} : U \to V$ by letting $f_{\gamma}(u)$ 
be the unique $v \in V$ such that $r_{\gamma}(u,v) > 0$. 
We refer to $f_{\gamma}$ as the map induced by $\gamma$. Note that $f_{\gamma}$ preserves the weights of all edges.
Furthermore, if $\gamma$ has transport cost $\sum c_{\Phi}(u,v) r_{\gamma}(u,v) = 0$, 
then $f_{\gamma}$ necessarily also preserves the labels of all nodes. This is summarized in the following proposition. The proof of this proposition appears in Section \ref{Sect:deterministic}. 

\begin{restatable}[]{proposition}{bijectiveiso}
\label{Prop:bijectiveiso}
Let $G = (U,\alpha,\phi_G)$ and $H = (V,\beta,\phi_H)$ be two fully supported graphs. 
If $\rho_{\Phi}(G,H)=0$ and $\gamma \in \mathcal{J}^*_{\Phi}(\alpha,\beta)$ is a bijective weight joining of $\alpha$ and $\beta$, then the induced map $f_{\gamma}$ is a graph isomorphism between $G$ and $H$.
\end{restatable}

Conversely, if $f : U \to V$ is a graph isomorphism from $G = (U, \alpha,\phi_G)$ to $H = (V,\beta,\phi_H)$, 
then there is a unique bijective weight joining $\gamma \in \mathcal{J}(\alpha,\beta)$ such that $f = f_{\gamma}$, and $\gamma$ has zero cost under $c_\Phi$; see Section \ref{Sect:deterministic}. 
Thus zero-cost bijective weight joinings are in one-to-one correspondence with graph isomorphisms.

 
\begin{definition} \label{Def:ident}
Let $\mathcal{G}$ be a graph family with augmented labeling scheme $\Psi$.  We say that OGJ with cost $c_\Psi$ \textit{identifies isomorphism} for $\mathcal{G}$ if for every pair of isomorphic graphs $G = (U,\alpha,\phi_G)$ and $H = (V,\beta,\phi_H)$ in $\mathcal{G}$, the extreme points of the OGJ solution set $\mathcal{J}^*_{\Psi}(\alpha,\beta)$ coincide with the zero-cost bijective weight joinings of $\alpha$ and $\beta$. 
\end{definition}

This property is closely related to the notion of \textit{convex exactness} in the study of convex relations of the graph isomorphism problem \cite{dym2018exact}; see Section \ref{Sect:relatedworks} for further discussion.

\section{Connections with graph isomorphism: Main results} \label{Sect:theoresults}
In this section we present our main results connecting the OGJ problem and the problems of detecting and identifying graph isomorphisms. The proofs of all results in this section appear in Section \ref{Sect:graphisomorphisms}. 

\subsection{Detection and identification with informative labeling schemes}
We say that a labeling scheme $\Psi = \{\psi_G : G \in \mathcal{G}\}$ is \textit{injective} if for every $G \in \mathcal{G}$ the label function $\psi_G$ is injective, that is, $\psi_G$ distinguishes the vertices of $G$. 

\begin{restatable}[]{proposition}{injectivesuff}
\label{Prop:injectivesuff}
Let $\mathcal{G}$ be a family of connected graphs with an isomorphism-invariant augmented labeling scheme $\Psi$. If $\Psi$ is injective, then OGJ with cost $c_\Psi$ detects and identifies isomorphism for $\mathcal{G}$. 
\end{restatable} 

We next investigate weaker conditions on $\Psi$ under which the same conclusions hold.   

\begin{definition} \label{Def:localmagic}
Let $G$ be graph with vertex set $U$ and an augmented label function $\psi_G$.
\begin{enumerate}

\item $\psi_G$ is {\it locally injective} if for each $u \in U$, the restriction of $\psi_G$ to the neighborhood $N_G(u)$ is injective. 

\item $\psi_G$ has a {\it magic symbol} if there is a vertex $u \in U$ with a unique label not shared by any other 
vertex, that is, $|\psi_G^{-1}(\psi_G(u))| = 1$.  In this case, $\psi_G(u)$ is called a {\it magic symbol}.

\end{enumerate}
An augmented labeling scheme $\Psi$ is locally injective for a family $\mathcal{G}$ if $\psi_G$ is locally injective for every 
$G \in \mathcal{G}$.  An augmented labeling scheme $\Psi$ has magic 
symbols if $\psi_G$ has a magic symbol for every 
$G \in \mathcal{G}$. 
\end{definition}

We note that the term \textit{magic symbol} appears in a related context in the field of symbolic dynamics; see \cite{lind2021introduction}. 
Now we are ready to state our next result.

\begin{restatable}[]{theorem}{weakinjectivesuff}
\label{Thm:weakinjectivesuff}
Let $\mathcal{G}$ be a family of connected graphs with an isomorphism-invariant augmented labeling scheme $\Psi$. 
If $\Psi$ is locally injective and has magic symbols, then OGJ with cost $c_\Psi$ 
detects and identifies isomorphism for $\mathcal{G}$. 
\end{restatable}

\subsection{Results with process-level labeling schemes} \label{Sect:ProcessLevel}

The results of the previous section rely on the power of the augmented labeling scheme to discriminate between vertices in some fashion.
In this section we change our perspective and consider the OGJ problem from the point of view of  random walks on graphs and Markovian couplings (as in Seccion \ref{Sect:MCs}). 
Using this perspective, we show that OGJ inherently takes into account information about the random walks at all times, despite the fact that the OGJ problem only explicitly refers to an expectation at time $0$. Combining the results of this section and the previous section, we obtain that OGJ can detect and identify isomorphisms for graphs with landmarks.

We begin with some definitions.
Let $G = (U,\alpha,\phi_G)$ be a graph with marginal function $p$ and any 
label function $\psi_G$. 
Let $P(u' | u) = \alpha(u,u') / p(u)$ be the transition matrix 
of the standard random walk on $G$ (see Section \ref{Sect:MCs}). 
For each vertex $u \in U$, let $X^u = X^u_0,X^u_1,\dots$ be the (non-stationary) Markov chain with initial state $X^u_0 =u$ and 
transition matrix $P$.  Finally, let $\psi_G^*$ be the function that 
assigns to each node $u \in U$ the process-level label
\[
\psi^*_G(u) \ = \ \text{Law}(\psi_G(X^u_0), \psi_G(X^u_1), \dots).
\]
In other words, $\psi^*_G(u)$ is the distribution of the process
of labels of $X^u$.
If $\Psi$ is a labeling scheme for $\mathcal{G}$,  then we let
$\Psi^*$ denote the corresponding process-level labeling scheme. 

It is clear that the process-level scheme $\Psi^*$ contains 
at least as much information as the original scheme $\Psi$, 
as the time-zero label of $\psi_G^*(u)$ is equal to $\psi_G(u)$. 
As Figure \ref{Fig:processlevel} illustrates, the process-level label function may contain more information.
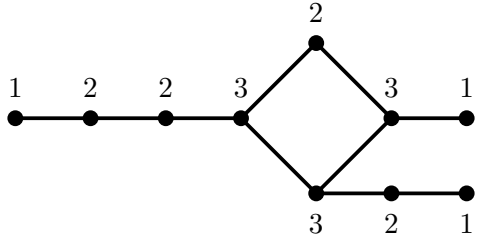
\begin{figure}[h!]
    \centering
    
    \begin{tikzpicture}

    \draw[line width=0.5mm] (-3,0) -- (-2,0) -- (-1,0) -- (0,0) -- (1,1) -- (2,0) -- (1,-1) -- (0,0);
    \draw[line width=0.5mm] (2,0) -- (3,0);
    \draw[line width=0.5mm] (1,-1) -- (2,-1) -- (3,-1);

    \node at (0,0) [circle,fill=black,inner sep=0pt,minimum size=6pt, color=black] {};
    \node at (2,0) [circle,fill=black,inner sep=0pt,minimum size=6pt, color=black] {};
    \node at (1,1) [circle,fill=black,inner sep=0pt,minimum size=6pt, color=black] {};
    \node at (1,-1) [circle,fill=black,inner sep=0pt,minimum size=6pt, color=black] {};
    
    \node at (-3,0.4) {1};
    \node at (-2,0.4) {2};
    \node at (-1,0.4) {2};
    \node at (0,0.4) {3};
    \node at (2,0.4) {3};
    \node at (3,0.4) {1};
    \node at (1,1.4) {2};
    \node at (1,-1.4) {3};
    \node at (2,-1.4) {2};
    \node at (3,-1.4) {1};

    \foreach \x in {-1, -2, -3, 3} {
      \foreach \y in {0} {
        \node at (\x,\y) [circle,fill=black,inner sep=0pt,minimum size=6pt, color=black] {};
      }
    }

    \foreach \x in {2, 3} {
      \foreach \y in {-1} {
        \node at (\x,\y) [circle,fill=black,inner sep=0pt,minimum size=6pt, color=black] {};
      }
    }
    
    \end{tikzpicture}

    \caption{A graph with vertices labeled by the discrete degree label function. For this example the discrete degree label function is not injective, while direct calculation shows that the process-level label function is injective.}
    \label{Fig:processlevel}
\end{figure}

The following result indicates that OGJ with the labeling scheme $\Psi$ contains the same information (at least regarding isomorphisms) as OGJ with the process-level labeling scheme $\Psi^*$. 
\begin{restatable}[]{theorem}{equivcost}
\label{Thm:equivcost}
If $G$ and $H$ are fully supported graphs with augmented labeling scheme $\Psi$, then $\rho_{\Psi}(G,H) > 0$ if and only if $\rho_{\Psi^*}(G,H) > 0$.
In addition, OGJ with cost $c_\Psi$ detects and identifies isomorphism for $\mathcal{G}$ if and only if OGJ with cost $c_{\Psi^*}$ detects and identifies isomorphism for $\mathcal{G}$.
\end{restatable}

Thus, OGJ with cost $c_\Psi$ distinguishes two graphs if and only 
if OGJ with cost $c_{\Psi^*}$ does. 
Although the OGJ problem with cost $c_{\Psi}$ makes reference only to the base-level labeling 
$\Psi$, it automatically incorporates the information in the process-level labeling $\Psi^*$. 

Theorem \ref{Thm:equivcost} enables us to extend the detection and identification results of the previous section to other graph families. 
As an example, we consider the family of graphs with landmarks \cite{khuller1996landmarks}, which have been used in applications of game theory, source localization, and embedding biological sequence data; see \cite{tillquist2023getting}. 

\begin{definition}
Let $G = (U,\alpha,\phi_G)$ be a graph with label space $\mathcal{L}$. Define $\dst_G : U \times \mathcal{L} \to [0,\infty]$ by setting
\begin{equation*}
\dst_G(u,a) = \inf \left\{ n \geq 0 : \exists \text{ a path of length $n$ in $G$ from $u$ to $\phi^{-1}_G(a)$} \right\}.
\end{equation*}
\end{definition}

\begin{definition} \label{landmarks}
Let $\mathcal{G}$ be a graph family with label space $\mathcal{L}$.  A subset $\mathcal{L}^* \subset \mathcal{L}$ is called 
a {\it set of landmarks} for $\mathcal{G}$ if $|\phi^{-1}_G(a)| \leq 1$ for all $a \in \mathcal{L}^*$ and all $G \in \mathcal{G}$. 
A set of landmarks $\mathcal{L}^*$ is {\it complete} for $\mathcal{G}$ if for each $G \in \mathcal{G}$, for each $u \neq u'$ in $U$, there exists $a \in \mathcal{L}^*$ such that $\dst_G(u,a) \neq \dst_G(u',a)$. 
\end{definition}

If $\mathcal{L}^*$ is a set of landmarks for $\mathcal{G}$, then the labels in $\mathcal{L}^*$ allow one to uniquely identify the vertices in $\phi_G^{-1}(\mathcal{L}^*)$ for each $G \in \mathcal{G}$. Moreover,  if $\mathcal{L}^*$ is a complete set of landmarks for $\mathcal{G}$, then all vertices in each $G \in \mathcal{G}$ can be uniquely identified (within $G$) by their distances to the vertices in $\phi_G^{-1}(\mathcal{L}^*)$. 
The following result is a consequence of Proposition \ref{Prop:injectivesuff} and Theorem \ref{Thm:equivcost}. 

\begin{restatable}[]{corollary}{landmarks}
\label{Cor:landmarks}
 Let $\mathcal{G}$ be a family of connected graphs with primary labeling scheme $\Phi$. If each graph in $\mathcal{G}$ has a complete set of landmarks, then OGJ with cost $c_\Phi$ detects and identifies isomorphism for $\mathcal{G}$.
\end{restatable}

\vskip.1in

\subsection{Isomorphism detection and identification for trees}
\noindent 
Recall that an automorphism of a graph $G$ is an isomorphism from $G$ to itself. Furthermore, the set of automorphisms together with the composition operation form the automorphism group of $G$. A graph is said to be asymmetric if its automorphism group is trivial, i.e., the group contains only one element (the identity automorphism). When the automorphism group is nontrivial, the graph is said to be symmetric. Asymmetry is a common assumption in existing work on isomorphism detection and identification; see Section \ref{Sect:relatedworks} for details.  
In this vein, we note that the hypotheses of Proposition \ref{Prop:injectivesuff} and Theorem \ref{Thm:weakinjectivesuff} imply that the graphs under consideration are asymmetric. However, our optimal transport-based approach is flexible enough to address certain families of symmetric graphs. As a first example, we consider the family of weighted trees.

\begin{restatable}[]{theorem}{trees}
\label{Thm:trees}
Let $\mathcal{G}$ be the family of finite, labeled, weighted trees with at least two 
vertices.  Let $\Psi$ be the augmented labeling scheme obtained from the primary 
labels in conjunction with the multiweight labeling. Then OGJ with cost $c_\Psi$ detects and identifies isomorphism for $\mathcal{G}$.
\end{restatable}

We note that there are several algorithms that have been shown to detect isomorphism among the family of (unweighted) trees, including the classical Weisfeiler-Leman test \cite{immerman1990describing} and the AHU algorithm \cite{aho1974design}, which can be modified to handle the weighted case. However, we are not aware of prior work addressing the identification of isomorphism (in the sense of Definition \ref{Def:ident}) or convex exactness for the family of trees.

The proof of Theorem \ref{Thm:trees}, which appears in Section \ref{Sect:graphisomorphisms}, proceeds by induction on the size of the trees involved. In the crucial inductive step, we show that if $G$ and $H$ are trees with $\rho_{\Psi}(G,H) = 0$, then any optimal graph joining must transport the leaves of $G$ to those of $H$ (and vice versa). Then we argue that we can remove the leaves from both graphs and apply our induction hypothesis to the resulting smaller trees and get an isomorphism of these subgraphs. Finally, we rely on Birkhoff's Theorem \cite{birkhoff1946tres} to extend the isomorphism of the subgraphs to an isomorphism of the leaves. 

\vskip.1in

\subsection{An extension principle via magic decompositions} \label{Sect:mgdecom}
Suppose that OGJ detects and identifies isomorphism for a graph family $\mathcal{G}$. 
Here we establish an extension principle showing that if the graphs in a family $\mathcal{F}$ can be created by ``gluing together" graphs from $\mathcal{G}$ in an
appropriate fashion, then OGJ detects and identifies isomorphism for $\mathcal{F}$ (see Theorem \ref{Thm:mgdecomposition} below). Prior to making this statement, we require several technical definitions that give a precise description of how the graphs in $\mathcal{G}$ may be glued together. The main idea is that one may take a disjoint union of graphs in $\mathcal{G}$ and then glue some of their leaves together to create a new graph in $\mathcal{F}$. However, the leaves cannot be glued together arbitrarily, as any graph could be obtained by gluing together finitely many graphs with only one edge. Thus, we also require that the labels of the new graph in $\mathcal{F}$ contain enough information that one may recover the graphs in $\mathcal{G}$ and their gluing data. Let us now state these conditions precisely.

We call a vertex $u$ of a graph $G = (U,\alpha,\phi_G)$ a \textit{leaf} if there is exactly one vertex $u' \in U$ such that $\alpha(u,u') > 0$, and moreover $u' \neq u$.  In this case we refer to $u'$ as the \textit{base} of the leaf $u$ and denote it by $b(u)$.  Let $L(G)$ denote the leaves of $G$.

Now let $G_i = (U_i, \alpha_i, \phi_{G_i})$, $i = 1,\ldots,k$, be a finite collection of connected graphs and let $M$ be a finite set. Let us first describe a gluing of the graphs $G_1,\dots,G_k$ along the set $M$.
We suppose that $M$ and $\{U_i\}_{i=1}^k$ are all pairwise disjoint. 
Further, suppose that for each $i = 1,\ldots,k$ we have identified a (possibly empty) subset $L_i \subset L(G_i)$ of the leaves of $G_i$ such that $U_i \setminus L_i \neq \varnothing$, and additionally we have a map $f_i : L_i \to M$ such that if $u \neq v \in L_i$ and $b(u) = b(v)$, then $f_i(u) \neq f_i(v)$; that is, 
distinct leaves with the same base must map to distinct elements of $M$. 
Let $U'_i = U_i \setminus L_i$ (which we assume is nonempty), and let $U = M \cup U'_1 \cup \dots \cup U'_k$. 
The set $U$ will serve as the vertex set of the new graph.
For each $i = 1,\ldots,k$ define the modified weight function $\alpha'_i : U \times U \to [0,\infty)$ by setting
\begin{equation*}
\alpha'_i(u,v) = \left\{ \begin{array}{ll}
                         \alpha_i(u,v), & \text{ if } u,v \in U'_i \\
                         \alpha_i(u,v'), & \text{ if } u = b(v') \text{ and } f_i(v') = v \\
                         \alpha_i(u',v), & \text{ if } v = b(u') \text{ and } f_i(u') = u \\
                         0, & \text{ otherwise}. 
                          \end{array}
                          \right.
\end{equation*}
Note that $\alpha'_i$ is well-defined by our restriction on the map $f_i$. Informally, the weight function $\alpha'_i$ is just the extension of $\alpha_i$ to the larger vertex set $U$, where we have replaced each leaf $u'$ in $L_i$ with its image $f_i(u')$ in $M$.

A graph $G = (U,\alpha,\phi_G)$ is a \textit{gluing} of the graphs $G_1,\ldots,G_k$ along the set $M$ if the following conditions hold.
\begin{enumerate}
\item There are sets $L_i \subset L(G_i)$ and functions $f_i : L_i \to M$ satisfying the assumptions given above.
\vskip.05in
\item The vertex set of $G$ is $U = M \cup U'_1 \cup \dots \cup U'_k$.
\vskip.05in
\item There are constants $\mu_1, \ldots, \mu_k  > 0$ and a function $\tau: U \times U \to [0,\infty)$ such that $\spp(\tau) \subset M \times M$ and
$\alpha = \tau + \sum_{i=1}^k \mu_i \alpha'_i$.
\end{enumerate}
In short, to construct the gluing $G$, we take a weighted disjoint union of $G_1,\dots,G_k$, replace the leaves $u'$ in $L_i$ by their images $f_i(u')$ in $M$ (which may amount to gluing some of these leaves together), and then we may add in edges between pairs of vertices in $M$ (with weights given by $\tau$). See Figure \ref{Fig:gluing} for an example. We note that if any $L_i$ is empty, then the resulting gluing is disconnected.

    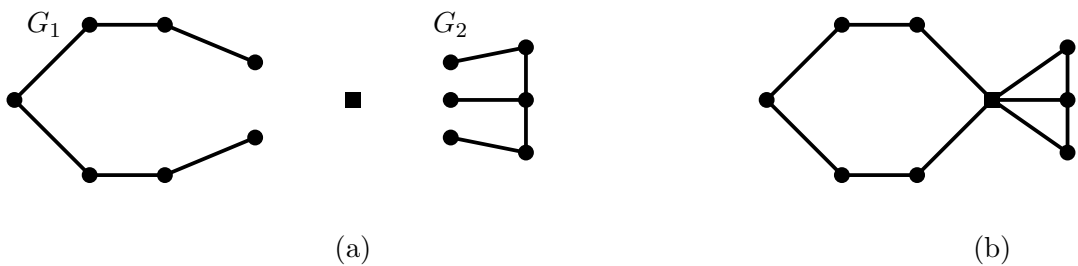
\begin{figure}[h!]
    \centering
    
    \begin{tikzpicture}

    \node at (-9,0) [circle,fill=black,inner sep=0pt,minimum size=6pt, color=black] {};
    \node at (-8,1) [circle,fill=black,inner sep=0pt,minimum size=6pt, color=black] {};
    \node at (-8,-1) [circle,fill=black,inner sep=0pt,minimum size=6pt, color=black] {};
    \node at (-5.8,0.5) [circle,fill=black,inner sep=0pt,minimum size=6pt, color=black] {};
    \node at (-5.8,-0.5) [circle,fill=black,inner sep=0pt,minimum size=6pt, color=black] {};
    
    \foreach \x in {-7} {
      \foreach \y in {-1, 1} {
        \node at (\x,\y) [circle,fill=black,inner sep=0pt,minimum size=6pt, color=black] {};
      }
    }

    \node at (-4.5,0) [rectangle,fill=black,inner sep=0pt,minimum size=6pt, color=black] {};
    

    \foreach \x in {-3.2} {
      \foreach \y in {-0.5, 0, 0.5} {
        \node at (\x,\y) [circle,fill=black,inner sep=0pt,minimum size=6pt, color=black] {};
      }
    }

    \foreach \x in {-2.2} {
      \foreach \y in {-0.7, 0, 0.7} {
        \node at (\x,\y) [circle,fill=black,inner sep=0pt,minimum size=6pt, color=black] {};
      }
    }

    \draw[line width=0.5mm] (-9,0) -- (-8,1) -- (-7,1) -- (-5.8,0.5) ;
    \draw[line width=0.5mm] (-9,0) -- (-8,-1) -- (-7,-1) -- (-5.8,-0.5) ;
    \draw[line width=0.5mm] (-3.2,0.5) -- (-2.2,0.7) -- (-2.2,0) -- (-2.2,-0.7) -- (-3.2,-0.5) ;
    \draw[line width=0.5mm] (-2.2,0) -- (-3.2,0) ;

    \node at (-8.6,1) {$G_1$};
    \node at (-3.2,1) {$G_2$};
    \node at (-4.5,-2) {\text{(a)}};

    \node at (1,0) [circle,fill=black,inner sep=0pt,minimum size=6pt, color=black] {};
    \node at (2,1) [circle,fill=black,inner sep=0pt,minimum size=6pt, color=black] {};
    \node at (2,-1) [circle,fill=black,inner sep=0pt,minimum size=6pt, color=black] {};
    \node at (4,0) [rectangle,fill=black,inner sep=0pt,minimum size=6pt, color=black] {};
    
    \foreach \x in {3} {
      \foreach \y in {-1, 1} {
        \node at (\x,\y) [circle,fill=black,inner sep=0pt,minimum size=6pt, color=black] {};
      }
    }

    \foreach \x in {5} {
      \foreach \y in {-0.7, 0, 0.7} {
        \node at (\x,\y) [circle,fill=black,inner sep=0pt,minimum size=6pt, color=black] {};
      }
    }

    \draw[line width=0.5mm] (1,0) -- (2,1) -- (3,1) -- (4,0) -- (3,-1) -- (2,-1) -- (1,0) ;
    \draw[line width=0.5mm] (4,0) -- (5,0.7) -- (5,0) -- (5,-0.7) -- (4,0) ;
    \draw[line width=0.5mm] (4,0) -- (5,0) ;

    \node at (4,-2) {\text{(b)}};

    \end{tikzpicture}

    \caption{An illustration of a gluing. (a) Two graphs $G_1$ and $G_2$, and the set $M$, which contains only a single element (pictured as a square). Here $f_1$ and $f_2$ map all the leaves of the two graphs to the unique element of $M$. (b) The gluing of $G_1$ and $G_2$ along $M$.}
    \label{Fig:gluing}
\end{figure}

Here we describe the specific conditions that the label functions must satisfy in order for our extension theorem to apply. These conditions guarantee that the labels associated with a graph $G$ in $\mathcal{F}$ contain sufficient information to recover the gluing data (i.e., the graphs $G_1,\dots,G_k$, the set $M$, the maps $f_1,\dots,f_k$, and the weights $\mu_1,\dots,\mu_k$) up to isomorphism.
Let $\mathcal{G}$ be a graph family with augmented $\mathcal{L}_{\mathcal{G}}$-labeling scheme $\Psi$. 
A graph family $\mathcal{F}$ with augmented $\mathcal{L}_{\mathcal{F}}$-labeling scheme $\widetilde{\Psi}$ 
has a \textit{$\mathcal{G}$-magic decomposition} if there exists $\mathcal{L}^* \subset \mathcal{L}_{\mathcal{F}}$ such that for every graph 
$G = (U,\alpha,\phi_G) \in \mathcal{F}$, there is a (possibly empty) set $M_G \subset U$ and nonempty connected graphs $\{G_i\}_{i=1}^k$ in $\mathcal{G}$ such that $G$ is a gluing of $\{G_i\}_{i=1}^k$ along $M_G$ and the augmented labeling scheme 
$\widetilde{\Psi}$ satisfies the following conditions:
\begin{enumerate}
\item for all $G \in \mathcal{F}$, we have $\widetilde{\psi}_G(M_G) \subset \mathcal{L}^*$ and $\widetilde{\psi}_G(U \setminus M_G) \cap \mathcal{L}^* = \emptyset$;
\item for all $G \in \mathcal{F}$, the restriction $\widetilde{\psi}_G|_{M_G}$ is injective;
\item there is a function $\eta : \mathcal{L}_{\mathcal{F}} \to \mathcal{L}_{\mathcal{G}}$ such that for all $G \in \mathcal{F}$, if  $u \in U'_i$, then $\psi_{G_i}(u) = \eta( \widetilde{\psi}_G(u))$;
\item there is a function $\eta' : \mathcal{L}_{\mathcal{F}} \times \mathcal{L}_{\mathcal{F}} \to \mathcal{L}_{\mathcal{G}}$ such that for all $G \in \mathcal{F}$, if $u \in L_i$ with $f_i(u) = m$ and $b(u) = u'$, then $\psi_{G_i}(u) = \eta'( \widetilde{\psi}_G(m), \widetilde{\psi}_G(u'))$; 
\item there is a function $\omega : \mathcal{L}_{\mathcal{F}} \to [0,1]$ such that for all $G \in \mathcal{F}$ and for all $u \in U'_i$, we have $\omega(\widetilde{\psi}_G(u)) = \mu_i$. 
\end{enumerate}
Condition (1) allows one to recognize $M_G$ from $\widetilde{\psi}_G$, and Condition (2) guarantees that any zero-cost joining of two graphs $G$ and $H$ in $\mathcal{F}$ will uniquely align the elements of $M_G$ and $M_H$. Condition (3) states that one can recover the augmented labeling of $u \in U'_i$ from its augmented labeling in $G$, and Condition (4) states that one can recover the augmented labeling of $u \in L_i$ from the labeling of its image $f_i(u)$ in $M_G$ and the labeling of its base $b(u)$. Finally, Condition (5) states that one can recover the weight $\mu_i$ associated to $G_i$ in the gluing. 

\begin{restatable}[]{theorem}{mgdecomposition}
\label{Thm:mgdecomposition}
Let $\mathcal{G}$ and $\mathcal{F}$ be families of fully supported graphs with isomorphism-invariant augmented labeling schemes $\Psi$ and $\widetilde{\Psi}$, respectively. Suppose that $\mathcal{F}$ has a $\mathcal{G}$-magic decomposition. If OGJ with cost $c_\Psi$ detects and identifies isomorphism for $\mathcal{G}$, then OGJ with cost $c_{\widetilde{\Psi}}$ detects and identifies isomorphism for $\mathcal{F}$.
\end{restatable}

We remark that in principle Theorem \ref{Thm:mgdecomposition} can be applied in a recursive manner to construct graph families with hierarchical decompositions for which OGJ can detect and identify isomorphisms.

As a special case of magic decompositions (with each $M_G = \emptyset$), we get a result for disjoint unions. For a precise definition of disjoint unions in this context, see Section~\ref{Subsect:decomposeweights}.
\begin{restatable}{corollary}{DisjointUnions}
\label{Cor:DisjointUnions}
    Let $\mathcal{G}$ be a family of fully supported graphs with an isomorphism-invariant augmented labeling scheme $\Psi$, and let $\mathcal{F}$ be the family of all (finite) disjoint unions of graphs in $\mathcal{G}$ with augmented labeling scheme $\widetilde{\Psi}$ satisfying the magic decomposition conditions (3) and (5) above. If OGJ with cost $c_{\Psi}$ detects and identifies isomorphism for $\mathcal{G}$, then OGJ with cost $c_{\widetilde{\Psi}}$ detects and identifies isomorphism for $\mathcal{F}$.
\end{restatable}

\begin{example} \label{Exa:forests}
    Let $\mathcal{F}$ be the family of weighted, labeled forests. The secondary label scheme $\Phi'$ is chosen such that the label of a vertex $u$ contains both the total weight of the tree containing $u$ and the multiweight of $u$. Let $\Psi$ be the associated augmented labeling scheme. Then $\mathcal{F}$ satisfies the hypotheses of Corollary \ref{Cor:DisjointUnions} with $\mathcal{G}$ as the family of trees. Then Theorems \ref{Thm:trees} and \ref{Thm:mgdecomposition} together yield that OGJ with cost $c_\Psi$ detects and identifies isomorphism for $\mathcal{F}$.
\end{example}

Thus, OGJ can detect and identify isomorphisms for the families of weighted trees and forests. Note that the families of trees and forests are well known to be detectable by the Weisfeiler-Leman test \cite{grohe2020recent} and the AHU algorithm \cite{aho1974design}. However, we are not aware of any results that guarantee the identification property (or convex exactness), even without weights.

As a simple example of a family of graphs with loops to which the extension principle can be applied, we consider the family of flower graphs.
\begin{example} \label{Exa:flowers}
    A flower is a graph having a single ``root" vertex $u_*$ and $k$ cycles ($k \ge 2$) attached to $u_*$, where the lengths of the cycles can be different. We note that any flower graph can be obtained as a gluing of paths of length $\ell \geq 3$ with $M_G = \{u_*\}$. See Figure \ref{Fig:flower}. Let $\mathcal{F}$ be the family of weighted, labeled flowers. As in the case of forests, we choose the secondary label scheme $\Phi'$ such that the label of a vertex $u$ contains both the total weight of the path containing $u$ and the multiweight of $u$. Let $\Psi$ be the associated augmented labeling scheme. It is straightfoward to check that $\mathcal{F}$ with $\Psi$ has a magic decomposition with the family of paths as the base family. Then 
    Theorems \ref{Thm:trees} and \ref{Thm:mgdecomposition} together yield that OGJ with cost $c_{\Psi}$ detects and identifies isomorphism for $\mathcal{F}$.

\begin{figure}[h!]
    \centering
    
    \begin{tikzpicture}

    \draw[line width=0.5mm] (0,0) -- (-1,1) -- (-2,0) -- (-1,-1) -- (0,0) -- (1,1) -- (2,1) -- (3,0) -- (2,-1) -- (1,-1) -- (2,-1) -- (1,-1) -- (0,0);

    \node at (0,0) [rectangle,fill=black,inner sep=0pt,minimum size=6pt, color=black] {};
    \node at (-2,0) [circle,fill=black,inner sep=0pt,minimum size=6pt, color=black] {};
    \node at (3,0) [circle,fill=black,inner sep=0pt,minimum size=6pt, color=black] {};
    
    \foreach \x in {-1, 1, 2} {
      \foreach \y in {-1, 1} {
        \node at (\x,\y) [circle,fill=black,inner sep=0pt,minimum size=6pt, color=black] {};
      }
    }
    
    \node at (0,-0.5) {$u_*$};
    \node at (0,-1.5) {\text{(a)}};

    \draw[line width=0.5mm] (7, 0.3) -- (6,1) -- (5,0) -- (6,-1) -- (7,-0.3);
    \draw[line width=0.5mm] (8.5,0.3) -- (9.5,1) -- (10.5,1) -- (11.5,0) -- (10.5,-1) -- (9.5,-1) -- (8.5,-0.3);

    \node at (5,0) [circle,fill=black,inner sep=0pt,minimum size=6pt, color=black] {};
    \node at (11.5,0) [circle,fill=black,inner sep=0pt,minimum size=6pt, color=black] {};
    \node at (7.75,0) [rectangle,fill=black,inner sep=0pt,minimum size=6pt, color=black] {};
    
    \foreach \x in {6, 9.5, 10.5} {
      \foreach \y in {-1, 1} {
        \node at (\x,\y) [circle,fill=black,inner sep=0pt,minimum size=6pt, color=black] {};
      }
    }

    \foreach \x in {7} {
      \foreach \y in {-0.3, 0.3} {
        \node at (\x,\y) [circle,fill=black,inner sep=0pt,minimum size=6pt, color=black] {};
      }
    }

    \foreach \x in {8.5} {
      \foreach \y in {-0.3, 0.3} {
        \node at (\x,\y) [circle,fill=black,inner sep=0pt,minimum size=6pt, color=black] {};
      }
    }
    
    \node at (7.75,-1.5) {\text{(b)}};
    \node at (7.75,-0.5) {$u_*$};
    \end{tikzpicture}

    \caption{(a) A flower graph with two cycles (of lengths four and six) attached to the base vertex $u_*$; (b) two path graphs that are glued along the set $M = \{u_*\}$ to create the flower in (a).}
    \label{Fig:flower}
\end{figure}
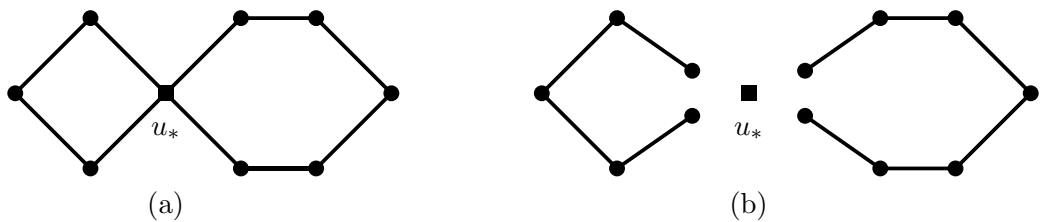
\end{example}
\section{Related Work} \label{Sect:relatedworks}

Optimal transport (OT) provides a mathematical framework for comparing probability distributions. 
Driven in part by advances in computing, ideas and techniques from OT have found use in a variety of fields, including
computer vision \cite{liu2020semantic, heitz2021sliced}, oceanography \cite{hyun2022ocean}, and cell biology \cite{tameling2021colocalization, cang2023screening, bunne2024optimal}. 
Recently, a number of authors have explored OT-based methods for comparing graphs \cite{chen2022weisfeiler, dong2020copt, maretic2019got, memoli2011gromov, vayer2020fused, yi2024alignment}. 
One branch of work in this direction makes use of the Gromov-Wasserstein (GW) distance and extensions for the purpose of graph comparison and alignment \cite{memoli2011gromov, peyre2016gromov, vayer2020fused}. These distances treat the graphs as metric spaces and aim to couple distributions over graph vertices so as to minimize the metric distortion. 
The paper \cite{rioux2024limit} establishes limit laws for the empirical GW distances across various settings and applies them to a problem involving hypothesis testing for isomorphisms in the context of random graph models. 

Another branch of related work uses OT in the context of Markov chains to compare and align graphs based on their dynamics \cite{Brugere2023DistancesFM, chen2022weisfeiler, o2022optimal, yi2024alignment}. 
As described above in Section \ref{Sect:MCs}, a graph $G$ with positive edge weights is naturally associated with a random walk on its vertices.  
For directed graphs the resulting Markov chain may not be reversible, in contrast to the situation studied here.
In \cite{Brugere2023DistancesFM, chen2022weisfeiler} the authors develop graph distance measures related to the classical Weisfeiler-Lehman (WL) isomorphism test \cite{grohe2020recent}.
These distance measures are based on consideration of the OT problem for the $k$-step distributions of the Markov chains on the given graphs;
connections with the present work are discussed in Section \ref{Sect:RelationWLtest} below.
In \cite{o2022optimal, yi2024alignment} the authors propose and develop a method
called NetOTC for comparing and aligning graphs by computing an optimal transition coupling of their random walks using a vertex-based cost function.  Like OGJ, NetOTC is based on graph couplings derived from random walks.  However,
NetOTC considers general directed graphs, i.e., no symmetry constraints 
are placed on the weight functions of the given graphs or their
transition couplings.  Absent symmetry constraints, the 
NetOTC problem is not a linear program, and general connections 
between NetOTC and graph isomorphism are not known at the present time. 

The graph isomorphism (GI) problem is to determine whether two given 
graphs are isomorphic, which in this work corresponds to the {\it detection} problem.  It is not known whether the GI problem is NP-complete or solvable in polynomial time; the best existing results establish 
that the GI problem can be solved in quasipolynomial time \cite{babai2016graph}. 
While much of the recent work on the GI problem is algebraic, involving the automorphism group of the graphs in question \cite{babai2016graph, grohe2020recent}, 
there is a substantial body of work that is more combinatorial in nature, with the AHU algorithm \cite{aho1974design} and the WL test (also known as the color refinement algorithm) being classical examples.  
The WL test is known to detect isomorphisms for certain families of graphs, including trees and forests \cite{immerman1990describing}, while the AHU algorithm applies only to trees and forests \cite{aho1974design, lindeberg2024isomorphism}. For detailed surveys on combinatorial approaches to the GI problem, 
see \cite{grohe2020recent, grohe2020graph}. 

The GI problem can be formulated as a combinatorial optimization problem 
involving the adjacency matrices $A$ and $B$ of the given graphs. 
The goal of the optimization problem is to minimize a cost, often the Frobenius norm, between $PA$ and $BP$ over all permutation matrices $P$. 
A common approach to the GI problem is to formulate and study convex 
relaxations in which the permutation matrices $P$ are replaced with doubly stochastic matrices or 
Markov kernels (as in \cite{patterson2021hausdorff}). 
Recent work \cite{aflalo2015convex, dym2018exact, fiori2015spectral}
provides theoretical guarantees establishing 
sufficient conditions under which the relaxed isomorphism problem can detect and identify graph isomorphism for certain graph families. 
As these results are similar in spirit to the identification results presented here, we describe them in further detail below. 

In \cite{aflalo2015convex} the authors show that the relaxed GI problem 
provides a method for detection and identification of isomorphism for the class of `friendly graphs.'  
A graph is friendly if its adjacency matrix has distinct eigenvalues and none of its eigenvectors 
is orthogonal to the all-ones vector. Note that friendly graphs are 
necessarily asymmetric (meaning that they have trivial automorphism groups). 
Subsequent work has extended these results to some non-friendly yet still asymmetric graphs \cite{fiori2015spectral}. 
Other work based on convex relaxation
establishes identification results for families that may contain
symmetric graphs \cite{dym2018exact}. 
Given two isomorphic graphs, the relaxed problem is said to be {\it convex exact} 
if the set of optimal solutions is equal to the convex hull of the permutation matrices corresponding to the common automorphism group of the graphs. 
Given a fixed set of vertices $U$ and a fixed group $\Gamma$ of permutations of $U$, \cite{dym2018exact} considers the set $\mathcal{W}(U, \Gamma)$ of adjacency matrices of weighted undirected graphs with vertex set $U$ and 
automorphism group $\Gamma$.
They establish that if all elements of $\Gamma$ are involutions and the action of $\Gamma$ on $U$ has an orbit of size $|\Gamma|$, then convex exactness holds for almost every adjacency matrix in $\mathcal{W}(U, \Gamma)$ (with respect to a modified version of Lebesgue measure on $\mathcal{W}(U, \Gamma)$).

In more recent work \cite{klus2025continuous}, the authors study the set of solutions to the relaxed problem considered in \cite{aflalo2015convex, fiori2015spectral, dym2018exact} and show in particular that if the adjacency matrices have repeated eigenvalues, then the dimension of the solution space increases, thereby complicating graph isomorphism testing.
In \cite{takapoui2016linear}, rather than using the Frobenius norm for the cost function as in \cite{aflalo2015convex, fiori2015spectral, dym2018exact}, the authors consider the problem of minimizing a linear cost over the set of doubly stochastic matrices $P$ satisfying $PA = BP$. It is shown that the resulting program can detect and uniquely identify isomorphism for friendly graphs.

\subsection{Relation with Weisfeiler-Lehman and NetOTC} 
\label{Sect:RelationWLtest}
In this section we discuss connections between OGJ, 
the Weisfeiler-Lehman (WL) test, the WL distance, 
and NetOTC. The proofs of the results in this section can be found in Section \ref{Sect:graphisomorphisms}.

The classical WL test (see \cite{grohe2020recent}) 
is an algorithm that takes two graphs as input and returns one of two possible outputs: either the graphs are not isomorphic, or the algorithm cannot determine whether the graphs are isomorphic. The more recent WL distance \cite{Brugere2023DistancesFM, chen2022weisfeiler} is a (generally non-metric) distance function
$d_{\WL}(G,H)$ on directed 
graphs that is defined via optimal non-stationary Markovian couplings 
of the random walks on $G$ and $H$ using a vertex-based cost.
The NetOTC method of \cite{yi2024alignment} gives rise to a (generally non-metric) 
distance function $d_{\OTC}(G,H)$ on directed graphs via optimal stationary Markovian 
couplings of the random walks on $G$ and $H$ using a vertex-based cost.
It is shown in \cite{Brugere2023DistancesFM} that
the WL distance is upper bounded by the NetOTC distance. Since reversible Markovian couplings are necessarily stationary, it is straightforward 
to see that the NetOTC distance is upper bounded by the OGJ cost. 

\begin{restatable}[]{proposition}{WLNetOTCOGJ}
\label{Prop:WLNetOTCOGJ}
	Let $G = (U,\alpha,\phi_G)$ and $H = (V,\beta,\phi_H)$ be undirected graphs. For any cost function $c : U \times V \to \mathbb{R}_{\ge 0}$, we have
	\begin{align*}
		d_{\WL}(G,H) \le d_{\OTC}(G,H) \le \rho_{c}(G,H).
	\end{align*}
\end{restatable}

As the name suggests, the WL distance is closely related to the WL test. 
Let $G = (U,\alpha,\phi)$ be an unweighted labeled graph, and let $\phi'(u) = (\deg(u),|U|)$ define a secondary labeling.
Fix $\delta \in (1/2,1)$, and let $\hat{G}_{\delta}$ be the weighted graph corresponding to the $\delta$-lazy random walk on $G$ with augmented label function $\psi = (\phi,\phi')$. 
It is shown in \cite[Proposition 3.3]{chen2022weisfeiler} that the WL test distinguishes unweighted labeled graphs $G$ and $H$ (i.e., determines correctly that they are not isomorphic) if and only if the WL distance between $\hat{G}_{\delta}$ and $\hat{H}_{\delta}$ with the label-based cost function is positive. 
In this sense, the WL distance has the same discriminating power as the WL test.
Combining this fact with Proposition \ref{Prop:WLNetOTCOGJ}, we obtain a similar result for OGJ. 

\begin{restatable}[]{corollary}{WLOGJ}
\label{Cor:WLOGJ}
Let $G$ and $H$ be unweighted labeled graphs. If the WL test determines that $G$ and $H$ are not isomorphic, then $\rho_{\Psi}(\hat{G}_{\delta},\hat{H}_{\delta})>0$. 
\end{restatable}

In this sense, OGJ has at least as much discriminating power as the WL 
test when restricted to the family of unweighted labeled graphs. 
The WL test can distinguish a random unweighted, unlabeled, 
and undirected graph on $n$ vertices from every other such graph on $n$ 
vertices with probability tending to one as $n$ tends to infinity \cite{babai1980random}.  Thus the same conclusion holds for OGJ. 
\section{Properties of weight joinings and graph joinings} \label{Sect:mainweight}

In this section, we describe the basic theoretical properties of weight joinings and graph joinings, beginning with properties of the set of weight joinings and some consequences of the marginal coupling and transition coupling conditions. In Section \ref{subsect:DisjointDecomposition}, we show that any weight joining of disconnected graphs can be decomposed into connected components corresponding to weight joinings of pairs of the connected components of the marginal graphs. Then in Section \ref{subsect:ExtremePoints} we characterize the extreme points of weight joinings in terms of their support sets.

\subsection{The set of weight joinings} \label{subsect:SetWeightJoinings}
Here we describe some basic properties of the set of weight joinings $\mathcal{J}(\alpha,\beta)$ for a given pair of weight functions $\alpha$ and $\beta$. Recall that a convex polyhedron is a subset of a Euclidean space that can be written as an intersection of a finite number of closed half-spaces. In the following result, suppose $\alpha$ and $\beta$ are weight functions on $U$ and $V$, respectively, where $|U| = m$ and $|V|=n$. We fix enumerations $U = \{u_1,\dots,u_m\}$ and $V = \{v_1,\dots,v_n\}$ and view each weight joining $\gamma$ as an element of $\R^{mn \times mn}$. 
\basicproperties*
\begin{proof}
    Let $\gamma^*$ be the product weight joining of $\alpha$ and $\beta$, i.e., $\gamma^* = \alpha \otimes \beta$. It is straightforward to verify that $\gamma^*$ is a weight function that satisfies both the marginal coupling and transition coupling conditions. Hence, $\mathcal{J}(\alpha, \beta)$ is nonempty. Because $\mathcal{J}(\alpha, \beta)$ is characterized by finitely many linear constraints over $\mathbb{R}^{mn \times mn}$, it is a convex polyhedron and is closed in $\mathbb{R}^{mn \times mn}$. Finally, for all $\gamma \in \mathcal{J}(\alpha, \beta)$, the non-negativity and normalization conditions for weight functions imply that $\gamma$ is a probability vector with index set $(U \times V)^2$, and therefore $\mathcal{J}(\alpha,\beta)$ is bounded. 
\end{proof}

\subsection{The marginal coupling and transition coupling conditions} \label{subsect:MarginalJoiningConditions}
To simplify notation, we use $G = (U, \alpha)$ to denote graphs when we do not need to refer to the label functions. The following proposition states that for connected graphs, the marginal coupling condition in the definition of weight joining is redundant. 
\begin{restatable}[]{proposition}{couplingredundant}
\label{Prop:couplingredundant}
    Suppose $G = (U, \alpha)$ and $H = (V, \beta)$ are two connected graphs. Then any weight function $\gamma$ on $U \times V$ that satisfies the transition coupling condition also satisfies the marginal coupling condition. 
\end{restatable}
\begin{proof}
    Let $\gamma$ be a weight function that satisfies the transition coupling condition. Let $(\widetilde{X},\widetilde{Y})$ be its associated reversible Markov chain. Let $X$ and $Y$ be the reversible Markov chains associated with $G$ and $H$, respectively. When $G$ and $H$ are connected, the reversible Markov chains $X$ and $Y$ are irreducible. As $\gamma$ satisfies the transition coupling condition, \cite[Proposition 4]{o2022optimal} gives that $(\widetilde{X}, \widetilde{Y})$ is a coupling of $X$ and $Y$. By considering the first marginals of these chains, we then see that $r_{\gamma}$ is a coupling of $p$ and $q$, and hence the marginal coupling condition is satisfied.  
\end{proof}

If two graphs are fully supported, then the following restatement of the transition coupling condition is useful. 
\begin{remark} \label{Rem:essential}
    Suppose $G = (U,\alpha)$ and $H = (V,\beta)$ are fully supported. Then $p(u) >0$ for all $u \in U$ and $q(v) > 0$ for all $v \in V$, and therefore the transition coupling conditions are equivalent to the following equations: for any $u,u'\in U$ and $v,v' \in V$,
    \begin{align*} 
        \displaystyle \sum_{\tilde{v} \in V} \gamma((u,v),(u',\tilde{v})) & = \dfrac{\alpha(u,u')}{p(u)} \, r_\gamma(u,v), \\
        \displaystyle \sum_{\tilde{u} \in U} \gamma((u,v),(\tilde{u},v')) & = \dfrac{\beta(v,v')}{q(v)} \, r_\gamma(u,v).
    \end{align*}
\end{remark}

The marginal coupling condition and the transition coupling condition together imply an {\it edge coupling condition}, described in the following proposition. 
\begin{restatable}[]{proposition}{realcoupling}
\label{Prop:realcoupling}
Let $\alpha$ and $\beta$ be weight functions defined on $U$ and $V$, respectively. Suppose $\gamma \in \mathcal{J}(\alpha,\beta)$. Then we have
    \begin{align}
        &\sum_{v,v' \in V} \gamma((u,{v}),(u',v')) = \alpha(u,u'), \quad \text{for all }u,u' \in U, \label{eqn:EdgeCoupling1} \\
        &\sum_{{u},{u'} \in U} \gamma(({u},v),({u'},v')) = \beta(v,v'), \quad \text{for all }v,v' \in V. \label{eqn:EdgeCoupling2}
    \end{align}
    Consequently, if $\gamma((u,v),(u',v')) > 0$, then $\alpha(u,u') > 0$ and $\beta(v,v') > 0$. We call this the {\it edge preservation property}.
\end{restatable}
\begin{proof}
First suppose that $p(u) = 0$. Then by definition of $p$ we must have $\alpha(u,u') = 0$ for all $u' \in U$, and by the marginal coupling condition, we have $r_{\gamma}(u,v) = 0$ for all $v \in V$. Thus both sides of (\ref{eqn:EdgeCoupling1}) are zero. Now suppose $p(u) > 0$. Then the transition coupling condition can be rewritten as in Remark \ref{Rem:essential}. Hence summing over $v \in V$ and using the maringal coupling condition gives
\begin{align*}
    \sum_{{v},{v'} \in V} \gamma((u,{v}),(u',{v'})) = \dfrac{\alpha(u,u')}{p(u)}\sum_{{v} \in V}r_{\gamma}(u,{v}) = \alpha(u,u').
\end{align*}
We have established Equation (\ref{eqn:EdgeCoupling1}). Equation (\ref{eqn:EdgeCoupling2}) is obtained analogously. 
\end{proof}

\subsection{Decomposition of weight joinings over disjoint unions} \label{Subsect:decomposeweights} \label{subsect:DisjointDecomposition}
We now consider disconnected graphs and study the structure of their weight joinings. Any disconnected graph can be expressed as a disjoint union of its connected components. Let us introduce some additional notation to describe disjoint unions of graphs. We assume that all graphs in this section are fully supported. We let $G=(U,\alpha)$ denote the graph $G$ with vertex set $U$ and weight function $\alpha$ when vertex labels are not taken into account. Suppose a graph $G = (U,\alpha)$ has $k$ connected components, and denote the vertex sets of these connected components by $U_i$, $i=1,\ldots,k$. For each component, we define a weighted undirected graph $G_i = (U_i,\alpha_i)$ as follows. Let $S_{G_i} = \sum_{u \in U_i} p(u)$. For all $u,u' \in U_i$, define ${\alpha}_i(u,u') = (S_{G_i})^{-1} \alpha (u,u')$. Since $U_i$ is a connected component of $G$, it is straightforward to verify that $\alpha_i$ is a weight function on $U_i$. We therefore write $G_i = (U_i,\alpha_i)$, and note that $G_i$ is a connected graph. Furthermore, observe that $\nu_G = (S_{{G}_1}, \ldots, S_{{G}_k} )$ is a probability distribution on the set $\{1,\dots,k\}$. Now, if $G$ has $k$ connected components, then we write $G = \bigsqcup_{i=1}^k G_i$. 

Suppose we have two disconnected graphs $G = \bigsqcup_{i=1}^k G_i$ and $H = \bigsqcup_{j=1}^l H_j$, where $G_i = (U_i, {\alpha}_i)$ and $H_j = (V_j, {\beta}_j)$ are the connected component graphs as described above. For $G_i = (U_i,{\alpha}_i)$, we use $p_i$ to denote its marginal function, defined as $p_i(u)=\sum_{u' \in U_i}{\alpha}_i(u,u')$, for $u \in U_i$. 
Similarly, for $H_j = (V_j,{\beta}_j)$, we let $q_j$ denote its marginal function. 

The following proposition provides a decomposition of any weight joining of a pair of disconnected graphs. In particular, it states that for two disconnected graphs $G = (U,\alpha)$ and $H = (V, \beta)$, we can express a weight joining $\gamma$ of $\alpha$ and $\beta$ as a linear combination of weight joinings of pairs of connected components. Moreover, the coefficients of this linear combination form a coupling $\pi$ of the weight distributions of the connected components of $G$ and $H$, denoted by $\nu_G$ and $\nu_H$ (as above). 
 
\begin{restatable}[]{proposition}{decomposeweights}
\label{Prop:decomposeweights}
Let $G = (U,\alpha)$ and $H = (V,\beta)$ be two fully supported graphs that can be disconnected. Suppose $G = \bigsqcup_{i=1}^k G_i$ and $H = \bigsqcup_{j=1}^l H_j$, with notation as above. 
Then for any weight joining $\gamma \in \mathcal{J}(\alpha, \beta)$, there is a coupling $\pi$ of $\nu_G$ and $\nu_H$ such that
    \begin{align*}
        \gamma = \sum_{i=1}^k \sum_{j=1}^l \pi_{ij} \cdot \gamma_{ij}, 
    \end{align*}
where each $\gamma_{ij}$ is a weight function on $U \times V$ such that $\spp(\gamma_{ij}) \subset (U_i \times V_j )^2$ and $\gamma_{ij}|_{(U_i \times V_j)^2}$ is a weight joining of ${\alpha}_i$ and ${\beta}_j$. 
\end{restatable}
\begin{proof}
    Let $\gamma$ be a weight joining of $\alpha$ and $\beta$. For $i = 1, \ldots, k$ and $j = 1, \ldots, l$, let 
    \begin{align*}
        \pi_{ij} = \sum_{(u,v),(u',v') \in U_i \times V_j} \gamma((u,v),(u',v')).
    \end{align*}
    If $\pi_{ij} = 0$, then let 
    \begin{align*}
        \gamma_{ij}((u,v),(u',v')) = \begin{cases}
            \alpha_i(u,u') \beta_j(v,v') &\text{if }(u,v), \ (u',v') \in U_i \times V_j, \\
            0 &\text{otherwise.}
        \end{cases}
    \end{align*}     
    Hence $\gamma_{ij}$ is a weight function and $\gamma_{ij}|_{(U_i \times V_j)^2} = \alpha_i \otimes \beta_j$, which is a weight joining of ${\alpha}_i$ and ${\beta}_j$. 
    Now suppose $\pi_{ij} > 0$. 
    For all $u,u' \in U$ and $v,v' \in V$, let
    \begin{align*}
        \gamma_{ij}((u,v),(u',v')) = \begin{cases}
            (\pi_{ij})^{-1} \gamma((u,v),(u',v')) &\text{if }(u,v), \ (u',v') \in U_i \times V_j, \\
            0 &\text{otherwise.}
        \end{cases}
    \end{align*}
    The non-negativity $\gamma_{ij}$ is clear, and the normalization property follows immediately from the definition of the scaling factor $(\pi_{ij})^{-1}$. The symmetry of $\gamma_{ij}$ is inherited from $\gamma$. Hence $\gamma_{ij}$ is a weight function. It is also clear from the definition that $\spp(\gamma_{ij}) \subset (U_i \times V_j)^2$. Let us now verify that $\gamma_{ij}|_{(U_i \times V_j)^2}$ is a weight joining of ${\alpha}_i$ and ${\beta}_j$.  Since $G_i$ and $H_j$ are connected, it remains only to show that $\gamma_{ij}$ satisfies the transition coupling conditions (by Proposition \ref{Prop:couplingredundant}). To this end, we first observe that for all $u,u' \in U_i$, we have $p(u)>0$ and $S_{G_i} >0$ since $G$ is fully supported, and then
    \begin{align} \label{Eqn:Puppy}
        \dfrac{\alpha(u,u')}{p(u)} = \dfrac{\alpha(u,u')/S_{G_i}}{p(u)/S_{G_i}} = \dfrac{\alpha_i(u,u')}{p_i(u)}.  
    \end{align}
 Let $r_{ij}$ be the marginal function of $\gamma_{ij}$. For any $(u,v) \in U_i \times V_j$, the definition of $\gamma_{ij}$ and the fact that $U_i$ and $V_j$ are connected components yields that 
    \begin{equation} \label{Eqn:Finch}
        r_{ij}(u,v) =  (\pi_{ij})^{-1} r_{\gamma}(u,v).
    \end{equation} By fact that $V_j$ is a connected component, the transition coupling condition for $\gamma$, and Equations (\ref{Eqn:Puppy}) and (\ref{Eqn:Finch}), for $u,u' \in U_i$ and $v \in V_j$, we obtain
    \begin{align*}
    		\sum_{v' \in V_j} \gamma_{ij}((u,v),(u',v')) &= (\pi_{ij})^{-1} \sum_{v' \in V} \gamma((u,v),(u',v')) \\
    		&=  (\pi_{ij})^{-1} \dfrac{\alpha(u,u')}{p(u)}r_{\gamma}(u,v) \\
    		&= \dfrac{{\alpha}_i(u,u')}{{p}_i(u)} r_{ij}(u,v).
    \end{align*}
    The remaining transition coupling condition relating $\gamma_{ij}$ and $\beta_j$ is established by an analogous argument.
	Hence $\gamma_{ij}|_{(U_i \times V_j)^2}$ is a weight joining of ${\alpha}_i$ and ${\beta}_j$. 
    By the definition of $\gamma_{ij}$ for all $i$ and $j$, we observe that
     \begin{align*}
        \gamma = \sum_{i=1}^k \sum_{j=1}^l \pi_{ij}\gamma_{ij}. 
    \end{align*}
    
    To prove that $\pi = (\pi_{ij})$ is a coupling of $\nu_G$ and $\nu_H$, note that for each $i = 1, \ldots, k$, we have
    \begin{align*}
        \sum_{j=1}^l \pi_{ij} = \sum_{j=1}^l \sum_{\substack{(u,v) \in U_i \times V_j \\ (u',v') \in U_i \times V_j}} \gamma((u,v),(u',v')) = \sum_{u,u' \in U_i} \alpha(u,u') = S_{G_i},
    \end{align*}
    where the second equality is a consequence of Proposition~\ref{Prop:realcoupling}.
    Thus $\nu_G$ is the first marginal of $\pi$. Similarly, $\nu_H$ is the second marginal of $\pi$, which completes the proof.
\end{proof}

We also note that if one has weight joinings for each pair of connected component graphs, then one can construct a weight joining of $G$ and $H$. 
\begin{remark} \label{Rem:conversedisjoint}
Under the hypotheses of Proposition \ref{Prop:decomposeweights}, let ${\gamma}_{ij} \in \mathcal{J}({\alpha_i},{\beta_j})$ for all $i = 1,\ldots, k$ and $j = 1,\ldots,l$. 
For any coupling $\pi$ of $\nu_G$ and $\nu_H$,  define the function $\gamma:(U\times V)^2\to \mathbb{R}$ by setting
$\gamma|_{(U_i \times V_j)^2} = \pi_{ij} \gamma_{ij}$
for all $i = 1,\ldots, k$, $j = 1,\ldots,l$, and $\gamma = 0$ elsewhere. Then $\gamma$ is a weight joining of $\alpha$ and $\beta$. We omit the straightforward proof of this fact. 
\end{remark}


\subsection{Extreme points of weight joinings} \label{subsect:ExtremePoints}
Recall that $\gamma$ is an extreme point of the convex set $\mathcal{J}(\alpha,\beta)$ if the following condition holds: whenever $\gamma = t \gamma_1 + (1-t)\gamma_2$ for some $\gamma_1, \gamma_2 \in \mathcal{J}(\alpha,\beta)$ and some $t \in (0,1)$, it necessarily follows that $\gamma = \gamma_1 = \gamma_2$. In the next proposition below we establish an equivalence between the extreme points of the set of weight joinings and those weight joinings with minimal support.

\begin{definition} \label{Def:sppwj} 
Let $\alpha$ and $\beta$ be weight functions on $U$ and $V$, respectively. We say a set $F \subset (U \times V)^2$ {\it supports a weight joining} if there exists a weight joining $\gamma \in \mathcal{J}(\alpha,\beta)$ such that $\spp(\gamma) = F$. Given a set $F$ that supports a weight joining of $\alpha$ and $\beta$, we say $F$ is {\it minimal} if no set $F'$ supporting a weight joining of $\alpha$ and $\beta$ is a proper subset of $F$. 
\end{definition}

The following proposition is essentially a restatement of standard results in linear optimization applied to our setting. Nonetheless, we provide a proof for completeness.
\begin{restatable}[]{proposition}{minimalextreme}
\label{Prop:minimalextreme}
    A weight joining $\gamma$ is an extreme point of $\mathcal{J}(\alpha,\beta)$ if and only if $\spp(\gamma)$ is minimal. Moreover, if $\gamma$ is extremal and $\spp(\gamma) = \spp(\gamma')$ for some $\gamma' \in \mathcal{J}(\alpha,\beta)$, then $\gamma = \gamma'$. 
\end{restatable}

Before proving Proposition \ref{Prop:minimalextreme}, we first prove the following lemma. 
\begin{lemma}\label{Lem:alpha}
    Let $\alpha,\beta$ be weight functions. Let $\gamma,\gamma'\in \mathcal{J}(\alpha,\beta)$ with $\spp(\gamma') \subset \spp(\gamma)$. Then there exists a sufficiently small $t >0$ such that $\tilde{\gamma} = (1+t) \gamma - t \gamma'\in \mathcal{J}(\alpha,\beta)$.    
\end{lemma}
\begin{proof}
    Let
    \begin{align*}
        t = \min_{((u,v),(u',v')) \in \spp(\gamma')} \dfrac{\gamma((u,v),(u',v'))}{\gamma'((u,v),(u',v'))}. 
    \end{align*}
    Since $\spp(\gamma')\subset \spp(\gamma)$, we have that $t$ is well-defined and positive. If $((u,v),(u',v')) \in \spp(\gamma')$, then by the minimality of $t$, we have
    $$\gamma((u,v),(u',v')) - t\gamma'((u,v),(u',v')) \ge 0. $$
    If $((u,v),(u',v')) \notin \spp(\gamma')$, then 
    $$\gamma((u,v),(u',v')) - t\gamma'((u,v),(u',v')) = \gamma((u,v),(u',v')) \ge 0. $$
    Thus $\tilde{\gamma} := (1+t) \gamma - t \gamma' = (\gamma - t\gamma') + t\gamma \ge 0$. Since the rest of the weight joining constraints (in Definition \ref{Def:weightjoinings}) are linear and $\tilde{\gamma}$ is a linear combination of the weight joinings $\gamma$ and $\gamma'$, it is straightforward to verify that $\tilde{\gamma}$ is a weight joining. 
\end{proof}

\begin{proof}[Proof of Proposition \ref{Prop:minimalextreme}]
Let $\gamma$ be an extreme point of $\mathcal{J}(\alpha,\beta)$. Suppose for contradiction that $\spp(\gamma)$ is not minimal. Then there exists ${\gamma'}$ such that $\spp(\gamma')\subsetneq \spp(\gamma)$. Note that $\gamma' \neq \gamma$. 
Then by Lemma \ref{Lem:alpha}, for a sufficiently small $t>0$, we have that $\tilde{\gamma}=(1+t)\gamma-t \gamma'$ belongs to $\mathcal{J}(\alpha,\beta)$. Consequently, $\gamma$ can be expressed as $\gamma = \frac{1}{1+t}\tilde{\gamma} + \frac{t }{1+t}\gamma',$ which contradicts the assumption that $\gamma$ is an extreme point of $\mathcal{J}(\alpha,\beta)$ (since $\gamma' \neq \gamma)$. Note that by this same argument, if $\spp(\gamma) = \spp(\gamma')$, then we must have $\gamma = \gamma'$, which establishes the second conclusion of the proposition.
    
Now suppose $\gamma \in \mathcal{J}(\alpha,\beta)$ is a weight joining such that $\spp(\gamma)$ is minimal. 
Since $\mathcal{J}(\alpha,\beta)$ is a compact, convex polytope, it has finitely many extreme points. By the Krein-Milman theorem, 
there exist extreme points $\gamma^{(1)},\dots,\gamma^{(k)}$ of $\mathcal{J}(\alpha,\beta)$ and $t_1,\dots,t_k >0$ with $\sum_{i=1}^k t_i = 1$ such that
\begin{equation*}
\gamma = \sum_{i=1}^k t_i \gamma^{(i)}.
\end{equation*}
By the non-negativity of each $\gamma^{(j)}$, the support of each $\gamma^{(j)}$ must be contained in that of $\gamma$. Specifically, for any $((u,v),(u',v'))\in \spp(\gamma^{(j)})$, it follows that $\gamma((u,v),(u',v'))>0$, and hence $$\spp(\gamma^{(j)})\subseteq \spp(\gamma) \text{ for all } j=1,\cdots,k. $$
Then by the minimality of $\spp(\gamma)$, we must have $\spp(\gamma^{(j)}) = \spp(\gamma)$ for all $j$. 

Now suppose for contradiction that $\gamma^{(j)} \neq \gamma$ for some $j$. By Lemma~\ref{Lem:alpha}, there exists a sufficiently small $\theta >0$ such that $\gamma^*=(1+\theta) \gamma^{(j)} - \theta \gamma  \in \mathcal{J}(\alpha,\beta)$. 
Then $\gamma^{(j)}$ can be written as a nontrivial convex combination $\gamma^{(j)} = \frac{1}{1+\theta}\gamma^*+\frac{\theta}{1+\theta}\gamma$, with $\gamma \neq \gamma^{(j)}$, which contradicts the assumption that $\gamma^{(j)}$ is an extreme point. 
Hence, $\gamma = \gamma^{(i)}$ for some $i$, and we conclude that $\gamma$ is an extreme point.
\end{proof}

As a consequence of the previous proposition, our next result states that an extremal graph joining of connected graphs must also be connected. Note that we only consider the part of the graph joining where each vertex has positive marginal degree.
\begin{restatable}[]{proposition}{preserveconnected}
\label{Prop:preserveconnected}
    If $G = (U, \alpha)$ and $H = (V,\beta)$ are connected graphs and $\gamma \in \mathcal{J}(\alpha,\beta)$ is extremal, then the graph $K = (\spp(r_{\gamma}),\gamma)$ is also connected.
\end{restatable}
\begin{proof}
    Suppose for contradiction that the graph $K = (\spp(r_{\gamma}),\gamma)$ is not connected. 
    Then there exists two nonempty subsets $S_1,S_2 \subset U \times V$ such that 
    $S_1 \cap S_2 = \emptyset$, $S_1 \cup S_2 = U \times V$,  $\gamma((u,v),(u',v')) = 0$ for all $(u,v) \in S_1, (u',v') \in S_2$, and there exist pairs $(u,v) \in S_1$ and $(u',v') \in S_2$ such that $r_{\gamma}(u,v)>0$ and $r_{\gamma}(u',v')>0$. 
    Define the normalized restriction of $\gamma$ to $S_1$ by 
    \begin{align*}
        \gamma_1((u,v),(u',v')) = 
        \begin{cases} 
            (C^{-1})\gamma((u,v),(u',v')) &\text{if } (u,v),(u',v') \in S_1, \\
            0 &\text{otherwise,}
        \end{cases}
    \end{align*}
    where the normalization constant is given by
    \begin{align*}
        C = \sum_{(x,y),(x',y') \in S_1}\gamma((x,y),(x',y')) > 0.
    \end{align*}
    We now verify that $\gamma_1 \in \mathcal{J}(\alpha,\beta)$. For $(u,v) \in S_1$,
    \begin{align*}
        r_{\gamma_1} (u,v) = \frac{\sum_{(x',y')\in S_1}\gamma((u,v),(x',y'))}{C} 
        = \frac{r_{\gamma}(u,v)}{C},
    \end{align*}
    since $\gamma((u,v),(u',v')) = 0$ whenever $(u',v') \in S_2$. For $(u,v) \in S_2$, we have $\gamma_1((u,v),(u',v')) = 0$ for all $(u',v')$, and thus $r_{\gamma_1}(u,v) = 0$.
    We now verify the transition coupling condition. First suppose $(u,v) \in S_1$ and $u' \in U$. If $\alpha(u,u') = 0$, then the transition coupling condition is trivially satisfied. So suppose that $\alpha(u,u')>0$. Then the transition coupling condition for $\gamma$ gives that there exists $v' \in V$ such that $\gamma((u,v),(u',v')) >0$, and by the definition of $S_1$, any such $v'$ must satisfy $(u',v') \in S_1$. Then by the definition of $\gamma_1$, the previous sentence, the transition coupling condition for $\gamma$, and the previous display, we see that
    \begin{align*}
        \sum_{v' \in V} \gamma_1((u,v),(u',v')) 
        &= \dfrac{\sum_{v' \in V: (u',v') \in S_1} \gamma((u,v),(u',v'))}{C}  \\
        &= \dfrac{\sum_{v' \in V} \gamma((u,v),(u',v'))}{C}  \\
        &= \dfrac{\alpha(u,u')}{p(u)} \frac{r_{\gamma}(u,v)}{C} \\
        &= \dfrac{\alpha(u,u')}{p(u)} r_{\gamma_1}(u,v),
    \end{align*}
    which verifies the transition coupling condition for $\gamma_1$ for $(u,v)$ and $u'$.
    If $(u,v) \in S_2$, then both sides of the transition coupling condition equal zero, so the equality trivially holds.
    The other transition coupling condition can be verified analogously. 
    By Proposition \ref{Prop:couplingredundant}, since $G$ and $H$ are connected, the marginal coupling condition is implied by the transition coupling condition. 
    Thus, $\gamma_1 \in \mathcal{J}(\alpha, \beta)$.
    
    By assumption, $\gamma$ is an extreme point of $\mathcal{J}(\alpha, \beta)$, and by Proposition \ref{Prop:minimalextreme}, its support is minimal.
    However, by construction, $\spp(\gamma_1) \subsetneq \spp(\gamma)$, contradicting the minimality of $\spp(\gamma)$.
    Therefore, we conclude that $K$ must be connected.
\end{proof}

\begin{remark}
Note that when $\gamma$ is not extremal, $K = (\spp(r_{\gamma}),\gamma)$ may be disconnected.
\end{remark}

\section{Properties of the optimal graph joining problem} \label{Sect:propertiesOGJ}

In this section, we first establish basic properties of the OGJ problem, including the fact it is a linear program. We then define the OGJ distance on graphs with a common vertex set and verify that it satisfies the properties of a metric when the vertex cost $c$ does. Finally, we establish a stability result for OGJ to conclude this section. 

\subsection{OGJ as a linear program}
As mentioned in Section \ref{Sect:generaltheory}, the OGJ problem is indeed a linear program and always has solutions. The following proposition is an immediate consequence of Proposition \ref{Prop:basicproperties}. 

\linearOGJ*
\begin{proof}
    Since $\mathcal{J}(\alpha, \beta)$ is a convex polyhedron in $\mathbb{R}^{mn \times mn}$ and the objective function is linear, the OGJ problem is a linear program. Because $\mathcal{J}(\alpha, \beta)$ is nonempty and compact, the OGJ problem always has solutions. 
\end{proof}

The optimal solution set $\mathcal{J}^*(\alpha, \beta)$ is also a convex set by the convexity of $\mathcal{J}(\alpha, \beta)$ and the linearity of the objective function. The following remark about the extreme points of $\mathcal{J}(\alpha, \beta)$ and $\mathcal{J}^*(\alpha, \beta)$ is a direct consequence of Theorem 2.7 in \cite{bertsimas1997introduction}.
\begin{remark}\label{remark:extremeoftwo_sets}
If $\gamma$ is an extreme point of $\mathcal{J}(\alpha,\beta)$ and $\gamma\in\mathcal{J}^*(\alpha,\beta)$, then $\gamma$ is also an extreme point of $\mathcal{J}^*(\alpha,\beta)$. Conversely, if $\gamma$ is an extreme point of $\mathcal{J}^*(\alpha,\beta)$, then $\gamma$ is an extreme point of $\mathcal{J}(\alpha,\beta).$
\end{remark}

\subsection{The OGJ distance}
Our development is compatible with the metric theory for optimal transport. That is, one can define a Wasserstein-style distance on undirected graphs with a common vertex set using the OGJ problem. We call this distance the $\kappa$-$\OGJ$ distance. 

\begin{definition} \label{Def:Wassersteindis}
Let $c: U \times U \to \mathbb{R}$ be a metric on $U$. For $\kappa \ge 1$, the $\kappa$-$\OGJ$ distance between two graphs $G = (U,\alpha)$ and $H = (U, \beta)$ is defined by
$$\rho_{\kappa}(G,H) = \min_{\gamma \in \mathcal{J}(\alpha,\beta)} \left( \sum_{(u,u') \in U \times U} c(u,u')^{\kappa} \, r_{\gamma}(u,u') \right)^{1/\kappa}.$$
\end{definition}

In the special case when $\kappa = 1$, we note that the $1$-OGJ distance is equal to the OGJ transport cost. 
Next we prove that the $\kappa$-$\OGJ$ distance is indeed a metric on the family of graphs with the same vertex set $U$ when $c$ is a metric on $U$. 
\begin{restatable}[]{proposition}{distanceOGJ}
\label{Prop:distanceOGJ}
    Let $\mathcal{G}$ be the family of graphs on vertex set $U$. If the cost function $c$ is a metric on $U$, then the $\kappa$-$\OGJ$ distance $\rho_{\kappa}( \cdot, \cdot) : \mathcal{G} \times \mathcal{G} \to \R$ is a metric on $\mathcal{G}$.
\end{restatable}

In the proof of the metric property of the OGJ distance, which is given below, we follow the standard outline from optimal transport theory by first establishing a Gluing Lemma and then using it to establish the triangle inequality. In the rest of this subsection, let $\mathcal{G}$ be the family of graphs on vertex set $U$. Consider three graphs $G_1 = (U, \alpha_1)$, $G_2 = (U, \alpha_2)$, and $G_3 = (U, \alpha_3)$ in $\mathcal{G}$. Let $p_1$, $p_2$, and $p_3$ denote their marginal functions, respectively. We use $\mathcal{J}(\alpha_1, \alpha_2, \alpha_3)$ to denote the set of three-way weight joinings of $\alpha_1, \alpha_2, \alpha_3$, which are defined as weight functions on $(U \times U \times U)^2$ satisfying the three-way marginal coupling and transition coupling conditions (as in Definition \ref{Def:weightjoinings}). For example, for a weight function $\gamma : (U \times U \times U)^2 \to \mathbb{R}$, the marginal coupling and transition coupling conditions with respect to $\alpha_1$ are
\begin{align*}
    \sum_{y,z \in U} r_{\gamma}(x,y,z) &= p_1(x),  \text{ and }\\
    p_1(x) \sum_{y', z' \in U} \gamma((x,y,z),(x',y',z')) &= \alpha_1(x,x') \, r_{\gamma}(x,y,z),
\end{align*}
where $r_{\gamma}(x,y,z) = \sum_{x',y',z' \in U}\gamma((x,y,z),(x',y',z'))$.
The other constraints for $\alpha_2$ and $\alpha_3$ are analogous. For two arbitrary weight joinings $\gamma_{12} \in \mathcal{J}(\alpha_1, \alpha_2)$ and $\gamma_{23} \in \mathcal{J}(\alpha_2, \alpha_3)$, let $r_{12}$ and $r_{23}$ denote their marginal functions, respectively. 
\begin{lemma}[Gluing Lemma] \label{Lem:gluinglemma}
    For weight joinings $\gamma_{12} \in \mathcal{J}(\alpha_1, \alpha_2)$ and $\gamma_{23} \in \mathcal{J}(\alpha_2, \alpha_3)$, there exists a three-way weight joining $\gamma \in \mathcal{J}(\alpha_1, \alpha_2, \alpha_3)$ such that $\gamma_{12}$ and $\gamma_{23}$ are its marginals, i.e, 
        \begin{align*} 
            \sum_{z, z' \in U} \gamma((x,y,z),(x',y',z')) = \gamma_{12}((x,y),(x',y')), \\ 
            \sum_{x, x' \in U} \gamma((x,y,z),(x',y',z')) = \gamma_{23}((y,z),(y',z')). 
        \end{align*}
\end{lemma}
\begin{proof}
    Define a function $\gamma : (U \times U \times U)^2 \to \mathbb{R}$ by setting
    \begin{align*}
        \gamma((x,y,z),(x',y',z')) = \left\{ \begin{array}{lll}
            \dfrac{\gamma_{12}((x,y),(x',y')) \, \gamma_{23}((y,z),(y',z'))}{\alpha_2(y,y')} &\text{ if } \alpha_2(y,y') \neq 0 \\
            0 &\text{ otherwise.}
        \end{array} \right.
    \end{align*}
     Then it is straightforward to show that $\gamma \in \mathcal{J}(\alpha_1,\alpha_2,\alpha_3)$. It remains to prove that $\gamma$ admits $\gamma_{12}$ and $\gamma_{23}$ as marginals. Indeed, if $\alpha_2(y,y') > 0$, then because $\gamma_{23} \in \mathcal{J}(\alpha_2,\alpha_3)$, Proposition \ref{Prop:realcoupling} yields that
    \begin{align*}
        \sum_{z, z' \in U} \gamma((x,y,z),(x',y',z')) &= \gamma_{12}((x,y),(x',y')) \sum_{z, z' \in U} \dfrac{\gamma_{23}((y,z),(y',z'))}{\alpha_2(y,y')} \\
        &= \gamma_{12}((x,y),(x',y')).
    \end{align*}
    If $\alpha_2(y,y') = 0$, then both sides of the coupling condition are zeros.
    Thus $\gamma_{12}$ is a marginal of $\gamma$. The fact that $\gamma_{23}$ is also a marginal of $\gamma$ is proved similarly.
\end{proof}

\begin{proof}[Proof of Proposition \ref{Prop:distanceOGJ}]
    One may easily show that $\rho_{\kappa}(\cdot,\cdot)$ is a nonnegative function and is symmetric on $\mathcal{G} \times \mathcal{G}$. Next, let us prove that $\rho_{\kappa}(G_1,G_2) = 0$ if and only if $\alpha_1 = \alpha_2$. Suppose $\alpha_1$ is identical to $\alpha_2$ on $U \times U$. Consider the following function $\gamma^* : (U \times U)^2 \to \mathbb{R}$
    \begin{align*}
        \gamma^* ((x,y),(x',y')) = 
        \left\{ \begin{array}{lll}
             \alpha_1(x,x') &\text{if} \; x=y,\text{ } x'=y'  \\
             0 &\text{otherwise}.
        \end{array}\right.
    \end{align*}
    Then for all $x,y \in U$, we have
    \begin{align*}
        r_{\gamma^*}(x,y) = \begin{cases}
            \sum_{x'\in U}\alpha(x,x') &\text{if } x=y,\\
            0 &\text{otherwise.}
        \end{cases}
    \end{align*}
    It is straightforward to check that $\gamma^* \in \mathcal{J}(\alpha_1,\alpha_2)$. 
    Since $c$ is a metric, we have $c(x,y) = 0$ whenever $x=y$. Thus $\langle c, r_{\gamma^*} \rangle = 0$, which implies $\rho_{\kappa}(G_1,G_2) = 0$. 
    
    Conversely, suppose $\rho_{\kappa}(G_1,G_2) = 0$. Let $\gamma$ be an optimal weight joining. Consider $u,u' \in U$ such that $\alpha_1(u,u') > 0$. By Proposition \ref{Prop:realcoupling}, there are some $v,v' \in U$ such that $\gamma((u,v),(u',v')) > 0$. Hence, $r_{\gamma}(u,v)>0$. Since $\rho_{\kappa}(G_1,G_2) = 0$, we must have $c(u,v)=0$. Hence $u = v$, since $c$ is a metric. By a similar argument for $(u',v')$, we obtain that $u' = v'$. Thus $v, v'$ are the unique vertices of $U$ such that $\gamma((u,v),(u',v')) > 0$. Then by Proposition \ref{Prop:realcoupling}, we have that
    $$\gamma((u,v),(u',v')) = \alpha_1(u,u').$$
    Likewise, we get 
    $$\gamma((u,v),(u',v')) = \alpha_2(v,v').$$
    Since $u=v$ and $u'=v'$, we obtain
    $$\alpha_1(u,u') = \alpha_2(v,v') = \alpha_2(u,u').$$
    Now consider $u,u' \in U$ such that $\alpha_1(u,u')=0$. If $\alpha_2(u,u') > 0$, then applying the same argument, we obtain $\alpha_1(u,u')>0$, a contradiction. In conclusion, $\alpha_1(u,u') = \alpha_2(u,u')$ for all $u,u' \in U$. 
    
    Finally, let us show that $\rho_{\kappa}(\cdot,\cdot)$ satisfies the triangle inequality, i.e.,  
    \begin{align}\label{triangleineq}
        \rho_{\kappa}(G_1,G_3) \le \rho_{\kappa}(G_1,G_2) +  \rho_{\kappa}(G_2,G_3). 
    \end{align}
    Indeed, let $\gamma_{12} \in \mathcal{J}(\alpha_1,\alpha_2)$ and $\gamma_{23} \in \mathcal{J}(\alpha_2,\alpha_3)$ be  optimal solutions of the OGJ problems on the right-hand side of \eqref{triangleineq}. By Lemma \ref{Lem:gluinglemma}, there exists $\gamma \in \mathcal{J}(\alpha_1,\alpha_2,\alpha_3)$ such that it admits $\gamma_{12}$ and $\gamma_{23}$ as marginals. Now, define $\gamma_{13} : (U \times U)^2 \to \mathbb{R}$ by letting 
    \begin{align*}
        \gamma_{13}((x,z),(x',z')) = \sum_{y,y' \in U} \gamma ((x,y,z),(x',y',z')). 
    \end{align*}
	It is straightforward to show that $\gamma_{13} \in \mathcal{J}(\alpha_1,\alpha_3)$.    
    The corresponding marginal function $r_{13}$ of $\gamma_{13}$ is given by
    \begin{align*}
    r_{13}(x,z) = \sum_{x',z' \in U} \sum_{y,y' \in U} \gamma ((x,y,z),(x',y',z')) = \sum_{y \in U} r_{\gamma}(x,y,z),
    \end{align*}
    for all $x,z\in U$. Hence $r_{\gamma}$ has $r_{13}$ as a marginal. Using this fact, the triangle inequality for $c$, Minkowski's inequality, the fact $\gamma$ has $\gamma_{12}$ and $\gamma_{23}$ as marginals, and the optimality of $\gamma_{12}$ and $\gamma_{23}$, we see that
    \begin{align*}
        \rho_{\kappa}^{\kappa}(G_1,G_3) &\le \sum_{x,z \in U} c(x,z)^{\kappa} \, r_{13}(x,z) \\
        &= \sum_{x,y,z \in U} c(x,z)^{\kappa} \, r_{\gamma}(x,y,z) \\
        &\le \sum_{x,y,z \in U} (c(x,y)+c(y,z))^{\kappa} \, r_{\gamma}(x,y,z) \\ 
        &\le \Big[ \Big( \sum_{x,y,z \in U} c(x,y)^{\kappa} \, r_{\gamma}(x,y,z) \Big)^{1/ \kappa} 
        + \Big( \sum_{x,y,z \in U} c(y,z)^{\kappa} \, r_{\gamma}(x,y,z) \Big)^{1/ \kappa}\Big]^{\kappa}  \\
        &= \Big[ \Big( \sum_{x,y \in U} c(x,y)^{\kappa} \, r_{12}(x,y) \Big)^{1/ \kappa} 
        + \Big( \sum_{y,z \in U} c(y,z)^{\kappa} \, r_{23}(y,z) \Big)^{1/ \kappa}\Big]^{\kappa}  \\
        &= \Big[\rho_{\kappa}(G_1,G_2) + \rho_{\kappa}(G_2,G_3)\Big]^{\kappa}.
    \end{align*}
Raising both sides of this inequality to the power $1/\kappa$ gives the result. 
\end{proof}

\subsection{Stability of OGJ} 
We now explore the stability of the OGJ problem. Before stating this property, we recall that for two complete metric spaces $\mathcal{X}$ and $\mathcal{Y}$, a set $\mathcal{V} \subset \mathcal{Y}$ is a neighborhood of a set $\mathcal{W} \subset \mathcal{Y}$ if $\mathcal{W} \subset \inte(\mathcal{V})$. Neighborhoods in $\mathcal{X}$ are defined analogously. A multifunction $F: \mathcal{X} \to 2^{\mathcal{Y}}$ is upper semicontinuous at a point $x_0 \in \mathcal{X}$ if for every neighborhood $\mathcal{V}_{\mathcal{Y}}$ of $F(x_0)$, there is some neighborhood $\mathcal{V}_{\mathcal{X}}$ of $x_0$ such that $F(\mathcal{V}_{\mathcal{X}}) \subset \mathcal{V}_{\mathcal{Y}}$. The function $F$ is upper semicontinuous if it is upper semicontinuous at every point. 

The following proposition ensures that the optimal cost function of OGJ is continuous and the optimal solution function of OGJ is upper semicontinuous. The result is based on Theorem 13 in \cite{o2022optimal}.
Consider two vertex sets $U = \{u_1,\dots,u_m\}$ and $V = \{v_1,\dots,v_n\}$. Let $\mathcal{W}(U)$ and $\mathcal{W}(V)$ be the set of all weight functions on $U$ and $V$, respectively. We endow $\mathcal{W}(U)$  with the subspace topology inherited by viewing $\mathcal{W}(U)$ as a subset of $\R^{m \times m}$, and similarly we endow $\mathcal{W}(V)$ with the subspace topology inherited from $\R^{n \times n}$. Let $\mathcal{Z} = \mathcal{W}(U) \times \mathcal{W}(V)$, and let 
$$\mathcal{S} = \{\gamma \in \mathcal{W}(U \times V) : \gamma \in \mathcal{J}(\alpha,\beta), \text{ for some } (\alpha,\beta) \in \mathcal{Z} \}.$$
We endow $\mathcal{S}$ with the subspace topology inherited from $\R^{mn \times mn}$. 

\begin{restatable}[]{proposition}{stabilityOGJ}
\label{Prop:stabilityOGJ}
    Let $c$ be a fixed cost function defined on $U \times V$. Then the following statements are true:
    \begin{enumerate}
        \item The optimal cost function $\rho_{c} : \mathcal{Z} \to \R$, defined by 
        $$\rho_c(\alpha, \beta) = \min\limits_{\gamma \in \mathcal{J}(\alpha, \beta)} \langle c, r_{\gamma} \rangle,$$ 
        is continuous.
        \item The optimal solution (multi)function $J^* : \mathcal{Z} \to 2^{\mathcal{S}}$, defined by 
        $$J^*(\alpha, \beta) = \argminunder_{\gamma \in \mathcal{J}(\alpha, \beta)} \langle c, r_{\gamma} \rangle,$$ 
        is upper semicontinuous.
    \end{enumerate}
\end{restatable}

Our proof of this proposition relies on a general stability result for optimization problems (Proposition 4.4 in \cite{bonnans2013perturbation}). We first consider this general framework and recall some preliminaries. Consider the following optimization problem with the set of possible solutions $\mathcal{Y}$ and the set of parameters $\mathcal{X}$:  
\begin{equation}
\begin{aligned}\label{parameterizedproblem}
        &\text{minimize } f(y,x) \\
    &\text{subject to } y \in F(x),
\end{aligned}
\end{equation}
where $f(\cdot,x)$ is the objective function and $F: \mathcal{X} \to 2^{\mathcal{Y}}$ is the constraint function. Problem \eqref{parameterizedproblem} is a parameterized optimization problem, where each point $x \in \mathcal{X}$ serves as a parameter. In what follows we suppose that $\mathcal{X}$ and $\mathcal{Y}$ are two complete metric spaces. The multifunction $F$ is closed at a point $x \in \mathcal{X}$ if the following condition holds: for all sequences $\{x_k\}_{k=1}^{\infty}$ in $\mathcal{X}$ and $\{y_k\}_{k=1}^{\infty}$ in $\mathcal{Y}$ and for all points $y \in \mathcal{Y}$, if $\{x_k\}_{k=1}^{\infty}$ converges to $x$, $\{y_k\}_{k=1}^{\infty}$ converges to $y$, and $y_k \in F(x_k)$ for all $k$, then $y \in F(x)$. If $F$ is closed at every point $x \in \mathcal{X}$, then it is closed.


\begin{proof}[Proof of Proposition \ref{Prop:stabilityOGJ}]
    Let us first reformulate the OGJ problem as an instance of Problem \eqref{parameterizedproblem} and then show that the hypotheses of Proposition 4.4 in \cite{bonnans2013perturbation} are satisfied. Let $\mathcal{X} = \mathcal{Z}$, and let $\mathcal{Y} = \mathcal{S}$.
    Then we claim that $\mathcal{X}$ and $\mathcal{Y}$ are compact spaces in $\mathbb{R}^{m \times m} \times \mathbb{R}^{n  \times n}$ and $\mathbb{R}^{mn \times mn}$, respectively. Indeed, since $\mathcal{X}$ is a product of two compact sets, it is compact in $\mathbb{R}^{m \times m} \times \mathbb{R}^{n  \times n}$. The boundedness of $\mathcal{Y}$ follows directly from the normalization property of weight functions, so we only need to check that it is closed. Let $\{\gamma_k \}$ be a sequence in $\mathcal{Y}$ that converges to $\gamma$. For each $k \in \mathbb{N}$, we have $\gamma_k \in \mathcal{J}(\alpha_k,\beta_k)$, for some $(\alpha_k,\beta_k) \in \mathcal{X}$. As $\alpha_k$ and $\beta_k$ are the images of $\gamma_k$ under the coordinate projections, which are continuous, we obtain that the sequences $\{\alpha_k\}$ and $\{\beta_k\}$ converge. Moreover, since $\mathcal{X}$ is closed (with the product topology), $\{(\alpha_k,\beta_k)\}$ converges to some $(\alpha,\beta) \in \mathcal{X}$. By passing to the limit in the linear constraints in the definition of weight joinings (Definition \ref{Def:weightjoinings}), we see that $\gamma \in \mathcal{J}(\alpha,\beta)$. Therefore, $\mathcal{Y}$ is closed in $\mathbb{R}^{mn \times mn}$. 
    The objective function $f(\cdot,\cdot)$ is chosen as the map $(\gamma, (\alpha, \beta)) \mapsto \langle c, r_{\gamma} \rangle$. Additionally, we let the multifunction $F : \mathcal{X} \to 2^{\mathcal{Y}}$ be the constraint function for the problem; that is, we let $F(\alpha,\beta) = \mathcal{J}(\alpha,\beta)$, for all $(\alpha,\beta) \in \mathcal{X}$. 

    Now we show that the OGJ problem satisfies the hypotheses (i)-(iv) of Proposition 4.4 in \cite{bonnans2013perturbation}. Indeed, let $x_0 = (\alpha_0, \beta_0) \in \mathcal{X}$. First, observe that
    \begin{align}\label{linearbounded}
        f(\cdot,\cdot) \le |f(\cdot,\cdot)| \le \|c\|_{\infty}.
    \end{align}
    Due to the continuity of the inner product, $f$ is continuous. Next, suppose there are two sequences $\{(\alpha_k,\beta_k)\} \subset \mathcal{X}$ and $\{\gamma_k\} \subset \mathcal{Y}$ 
    and a point $\gamma \in \mathcal{Y}$ such that the sequence $\{(\alpha_k, \beta_k)\}$ converges to $(\alpha, \beta) \in \mathcal{X}$, the sequence $\{\gamma_k\}$ converges to $\gamma$, and for each $k$, we have $\gamma_k \in F(\alpha_k, \beta_k)$.
    By an argument similar to the one in the previous paragraph (showing that $\mathcal{Y}$ is closed),
    we see that $\gamma \in F(\alpha,\beta)$. Therefore $F$ is closed, and (ii) holds. To show (iii), from the observation \eqref{linearbounded}, we can choose $M = \|c\|_{\infty}$. Then the level set $\{y \in F(x) : f(y,x) \le M \}$ is always nonempty. In addition, the set $\mathcal{Y}$ is compact, and the level set above is a subset of $\mathcal{Y}$. Lastly, let $\mathcal{V}_{\mathcal{Y}}$ be a neighborhood of $\argmin_{y \in F(x_0)} f(y,x_0)$. Choose 
    $$\mathcal{V}_{\mathcal{X}} = \{ (\alpha, \beta) \in \mathcal{X} : \exists \gamma \in \mathcal{V}_{\mathcal{Y}}, \,  \gamma \in \mathcal{J}(\alpha, \beta) \}.$$
    Then, by construction, $x_0 \in \mathcal{V}_{\mathcal{X}}$, and $F(x) \cap \mathcal{V}_{\mathcal{Y}} \neq \emptyset$ for all $x \in \mathcal{V}_{\mathcal{X}}$. We have shown that all conditions of Proposition 4.4 in \cite{bonnans2013perturbation} are satisfied. Hence, $\rho_{c}$ is continuous, and $J^*$ is upper semicontinuous at every point $(\alpha, \beta)$.
\end{proof}

\section{Deterministic and bijective weight joinings} \label{Sect:deterministic}
As in standard optimal transport, a weight joining may split mass from one vertex in $G$ and transport it to multiple vertices in $H$. In this section, we first characterize those weight joinings that do \textit{not} split mass in this manner. We then relate such deterministic weight joinings to graph isomorphisms. Finally, we establish results concerning the relationship between bijective weight joinings and extreme points of weight joinings. 

\subsection{Relations with graph isomorphisms}
We first introduce weight joinings that do not split mass in one direction. 
\begin{definition} \label{Def:deterministicwj}
    Let $\alpha$ and $\beta$ be weight functions on $U$ and $V$, respectively. A weight joining $\gamma \in \mathcal{J}(\alpha,\beta)$ with marginal function $r_\gamma$ is {\em deterministic} from $U$ to $V$ if for each $u \in U$ there is a unique $v \in V$ such that $r_\gamma(u,v)> 0$.
\end{definition}
In other words, a weight joining $\gamma$ is deterministic if the mass of each vertex in the source graph is sent to only one vertex in the target graph. However, a vertex in the target graph may receive mass from multiple source vertices. 
Bijective weight joinings (as defined in Section \ref{Sect:theoresults}) are deterministic in both directions.

If $\gamma$ is a deterministic weight joining from $G$ to $H$, then we let $f_{\gamma}: U \to V$ denote the map induced by $\gamma$, which is the map such that $f_{\gamma}(u)$ is the unique $v \in V$ satisfying $r_{\gamma}(u,v) > 0$. Note that $f_{\gamma}$ is well-defined by the property of deterministic weight joinings. We refer to $f_{\gamma}$ as the map induced by the deterministic weight joining $\gamma$. 

Any map induced by a deterministic weight joining satisfies a characteristic property, which was first introduced in \cite{yi2024alignment} for weighted directed graphs. Here we provide the analogous definition in the setting of weighted undirected graphs.

\begin{definition} \label{Def:factormaps}
    Given two weighted undirected graphs $G = (U,\alpha)$ and $H = (V,\beta)$, a map $f : U \to V$ is a {\it factor map} if it is surjective and for all $v, v' \in V$ and $u \in f^{-1}(v)$, we have
    \begin{align*}
        q(v)\sum_{u' \in f^{-1}(v')} \alpha(u,u') = p(u) \, \beta(v,v'). 
    \end{align*}
\end{definition}
If $f : U \to V$ is a factor map, then we define a function $\gamma_{f} : (U \times V)^2 \to \mathbb{R}$ by letting
	\begin{align*} 
        \gamma_{f}((u,v),(u',v')) = \left\{\begin{array}{ll}
            \alpha(u,u') & \text{ if } v = f(u), \text{and } v' = f(u') \\
            0            & \text{ otherwise}.
        \end{array} \right.
	\end{align*}
    The following proposition makes precise connections between deterministic weight joinings and factor maps.
\begin{restatable}[]{proposition}{factordeterministic}
\label{Prop:factordeterministic}
    Let $G = (U,\alpha)$ and $H = (V,\beta)$ be two fully supported graphs. 
    \begin{enumerate}
    		\item If $f : U \to V$ is a factor map, then the function $\gamma_{f}$ is a deterministic weight joining from $U$ to $V$. 
        \item If $\gamma$ is a deterministic weight joining of $\alpha$ and $\beta$, then $f_{\gamma} : U \to V$ is a factor map. 
    \end{enumerate}
\end{restatable}
As this proposition is a special case of Theorem 16 in \cite{yi2024alignment}, we omit its proof. 

It turns out that graph isomorphisms are closely related to factor maps. Specifically, if a factor map is injective, then it is a bijection that preserves the edge weights. However, additional considerations are required to ensure label preservation (and obtain an isomorphism of labeled graphs).
Recall that a labeled graph is given by $G = (U, \alpha, \phi_G)$, where $\phi_G$ is a vertex label function. Let $\mathcal{G}$ be a family of labeled graphs. Then $\Phi = \{\phi_G : G \in \mathcal{G} \}$ is the primary labeling scheme of $\mathcal{G}$. The primary label-based cost function is defined by 
$$c_{\Phi}(u,v) = \mathbb{I}(\phi_G(u) \neq \phi_H(v)).$$
We also recall that $\rho_{\Phi}(G,H)$ and $\mathcal{J}^*_{\Phi}(\alpha,\beta)$ are the OGJ transport cost and the set of optimal weight joinings of $G$ and $H$ with the cost function $c_{\Phi}$, respectively. 
\bijectiveiso*
\begin{proof}
    Let $f = f_{\gamma}$ be the map induced by $\gamma$. 
    First, since $\gamma$ is bijective, for each $v \in V$, there exists a unique $u \in U$ such that $r_{\gamma}(u,v) > 0$, and  therefore the induced map $f$ is injective. Also, since $\gamma$ is bijective, it is deterministic from $U$ to $V$. Hence $f$ is a factor map by Proposition \ref{Prop:factordeterministic}. Then $f$ is surjective, making it a bijection, and for all $u, u' \in U$, we have
    $$q(f(u)) \, \alpha (u,u') = p(u) \, \beta(f(u),f(u')).$$
   Furthermore, since $\gamma$ is bijective and $r_{\gamma}$ is a coupling of $p$ and $q$, we obtain that $p(u) = r_{\gamma}(u,f(u)) = q(f(u))$ for all $u \in U$. 
   Since $G$ is fully supported, $p(u)>0$ for all $u \in U$. Thus, by the previous display, we have $\alpha(u,u') = \beta(f(u),f(u'))$ for all $u,u' \in U$, i.e., $f$ preserves weights. To show that $f$ preserves the primary labels (given by $\Phi$), first observe that for every $u \in U$, since $G$ is fully supported, we have $p(u) > 0$. Hence, by the marginal coupling condition, we have $r_{\gamma}(u,f(u)) > 0$. Because $\rho_{\Phi}(G,H) = 0$, we must have $c_{\Phi}(u,f(u)) = 0$, and thus $\phi_G(u) = \phi_H(f(u))$ (by definition of the cost function $c_{\Phi}$). In conclusion, $f$ is a graph isomorphism between $G$ and $H$. 
\end{proof}

The following proposition provides the converse: if $f$ is a graph isomorphism, then there is a unique bijective weight joining such that its induced map coincides with $f$, and it has zero OGJ transport cost.  Combining this result and Proposition \ref{Prop:bijectiveiso}, these zero-cost bijective weight joinings are in one-to-one correspondence with graph isomorphisms. 
\begin{restatable}[]{proposition}{isomorphismbijective}
\label{Prop:isomorphismbijective}
Let $G = (U,\alpha,\phi_G)$ and $H = (V,\beta,\phi_H)$ be fully supported graphs. If $f : U \to V$ is a graph isomorphism between $G$ and $H$, then there is a unique bijective weight joining $\gamma \in \mathcal{J}(\alpha,\beta)$ such that $f = f_{\gamma}$, and $\gamma$ has zero cost under $c_{\Phi}$. 
\end{restatable}
\begin{proof}    
   Since $f$ is a graph isomorphism, $f$ is a factor map. Let $\gamma = \gamma_f$. Hence, $\gamma$ is a (deterministic) weight joining from $U$ to $V$ by Proposition \ref{Prop:factordeterministic}. Next, let us show that $\gamma$ is bijective. First, by the definition of $\gamma_f$, we observe that the marginal function of $\gamma$ must take the following form:
    \begin{align*}
        r_{\gamma}(u,v) = \left\{\begin{array}{lll}
        p(u) &\text{ if } v = f(u) \\
        0 &\text{ otherwise.}
        \end{array} \right.
    \end{align*}
Using this expression for $r_{\gamma}$, the fact that $f$ is a bijection,  and the fact that $G$ and $H$ are fully supported, we obtain that $\gamma$ is bijective and $f_{\gamma} = f$. Since $f$ preserves labels with $\Phi$, we obtain $\langle c_{\Phi}, r_{\gamma} \rangle = 0$. Therefore, the OGJ transport cost of $\gamma$ is zero. Finally, note that any bijective weight joining $\gamma'$ corresponding to the same $f$ must have the following form:
    $$\gamma'((u,v),(u',v')) = 
        \begin{cases}
            \alpha(u,u') &\text{ if } v = f(u), v' = f(u'),\\
             0 &\text{ otherwise.}
        \end{cases}$$
Hence, $\gamma$ is the unique bijective weight joining with these properties. 
\end{proof}

These results suggest that the OGJ problem is closely related to the graph isomorphism problem. See Section \ref{Sect:graphisomorphisms} for further statements in this direction. 
\subsection{Bijective weight joinings and extreme points}
In this subsection, we explore the connection between bijective weight joinings and extreme points. Later we use these results to establish that OGJ identifies isomorphism for some graph families.

The following proposition states that any bijective weight joining $\gamma \in \mathcal{J}^*(\alpha,\beta)$ must be an extreme point of $\mathcal{J}^*(\alpha,\beta)$. Since a bijective weight joining corresponds to a graph isomorphism, this insight enables us to identify graph isomorphisms by examining the extreme points of the feasible set $\mathcal{J}^*(\alpha,\beta)$. Note that there is no reference to the labeling scheme in the next two propositions, so these results hold for a general vertex cost function.  
\begin{restatable}[]{proposition}{bijective_extreme}
\label{Prop:bijective_extreme}
    Let $G = (U,\alpha,\phi_G)$ and $H = (V,\beta,\phi_H)$ be two fully supported graphs. If $\gamma\in \mathcal{J}^*(\alpha,\beta)$ is a bijective weight joining, then $\gamma$ must be an extreme point of $\mathcal{J}^*(\alpha,\beta).$ 
\end{restatable}
\begin{proof}
    Let $\gamma \in \mathcal{J}^*(\alpha,\beta)$ be bijective, and suppose that there exists $\gamma', \tilde{\gamma} \in \mathcal{J}^*(\alpha,\beta)$ and $t \in (0,1)$ such that $$\gamma = t\gamma'+(1-t)\tilde{\gamma}.$$
    Let $f_{\gamma}:U\to V$ denote the bijective map induced by $\gamma$. Also, let $r_{\gamma}, r_{\gamma'}, r_{\tilde{\gamma}}$ denote the respective marginal functions, and note that $r_{\gamma} =  t r_{\gamma'} + (1-t) r_{\tilde{\gamma}}$. Since $\gamma$ is bijective with induced map $f_{\gamma}$, we have $r_{\gamma}(u,v) = 0$ for all $v \neq f_{\gamma}(u)$.  Thus we see that $r_{\gamma'}(u,v)=0$ for all $v\neq f_{\gamma}(u)$. Moreover, since $G$ is fully supported, we must have $$\sum_{v \in V} r_{\gamma'}(u,v) = p(u)>0,$$ and hence $r_{\gamma'}(u,f_{\gamma}(u)) = p(u) > 0$. Therefore $\gamma'$ is also a bijective weight joining with the same induced map $f_{\gamma}$. 
    As $\gamma$ is the unique bijective weight joining with induced map $f_{\gamma}$ (by Proposition \ref{Prop:isomorphismbijective}), we must have $\gamma' = \gamma$.
    An analogous argument shows that $\tilde{\gamma} = \gamma$. 
    Thus we conclude that  $\gamma$ is an extreme point of $\mathcal{J}^*(\alpha,\beta)$.
    \end{proof}

It is natural to ask about the converse direction. We provide a characterization of when an extreme point of the optimal weight joinings is bijective. 

\begin{restatable}[]{proposition}{extremebijective}
\label{Prop:extremebijective}
    Let $\alpha$ and $\beta$ be two weight functions. Suppose that for any $\gamma \in \mathcal{J}^*(\alpha,\beta)$, we have $\langle c, r_{\gamma} \rangle = 0$.
    Then the following two statements are equivalent:
    \begin{itemize}
        \item[(i)] For every weight joining $\gamma$ in $\mathcal{J}^*(\alpha,\beta)$, there exists a bijective weight joining $\gamma'$ in $\mathcal{J}^*(\alpha, \beta)$ such that $\spp(\gamma') \subset \spp(\gamma)$.
        \item[(ii)] Every extreme point of $\mathcal{J}^*(\alpha,\beta)$ is bijective.
    \end{itemize}
\end{restatable}
\begin{proof}
Suppose statement (i) holds, and let $\gamma$ be an extreme point of $\mathcal{J}^*(\alpha,\beta)$. By assumption, there exists a bijective weight joining $\gamma' \in \mathcal{J}^*(\alpha, \beta)$ such that $\spp(\gamma') \subseteq \spp(\gamma)$. Since $\gamma$ is an extreme point of $\mathcal{J}^*(\alpha,\beta)$, by Remark \ref{remark:extremeoftwo_sets}, $\gamma$ is an extreme point of $\mathcal{J}(\alpha,\beta)$. By Proposition \ref{Prop:minimalextreme}, $\spp(\gamma)$ is minimal. Thus, we conclude that $\spp(\gamma) = \spp(\gamma')$. 
Let $f$ be the bijective map associated with $\gamma'$. Then
$$\gamma^{\prime}\left((u, v),\left(u^{\prime}, v^{\prime}\right)\right)= \begin{cases}\alpha\left(u, u^{\prime}\right) & \text { if } v=f_{\gamma^{\prime}}(u), \text { and } v^{\prime}=f_{\gamma^{\prime}}\left(u^{\prime}\right) \\ 0 & \text { otherwise. }\end{cases}$$
Thus,
$$\spp(\gamma)=\spp\left(\gamma^{\prime}\right)=\Big\{\big( (u, f_{\gamma^{\prime}}(u)),(u^{\prime}, f_{\gamma^{\prime}}(u^{\prime})) \big): \alpha\left(u, u^{\prime}\right)>0\Big\}.$$
Using this expression for the support of $\gamma$, one may readily check that $\gamma$ is bijective, which establishes statement (ii).

Now suppose statement (ii) holds, i.e., every extreme point of $\mathcal{J}^*(\alpha,\beta)$ is bijective.  Since $\mathcal{J}^*(\alpha,\beta)$ is a convex polytope, any $\gamma\in\mathcal{J}^*(\alpha,\beta)$ can be expressed as a convex combination of bijective weight joinings. Therefore there exists at least one bijective weight joining $\gamma' \in \mathcal{J}^*(\alpha,\beta)$ such that $\spp(\gamma') \subset \spp(\gamma)$.
\end{proof}
\section{Optimal graph joining and graph isomorphism} \label{Sect:graphisomorphisms}

This section is dedicated to the proofs of results connecting OGJ to graph isomorphism detection and identification.
\subsection{Labeling schemes}
We begin by proving the elementary lemma about equivalence between secondary labeling schemes and augmented labeling schemes. Recall that for a graph family $\mathcal{G}$, we denote its primary, secondary, and augmented labeling schemes by $\Phi$, $\Phi'$, and $\Psi$, respectively. Also, recall that for a graph $G = (U,\alpha,\phi_G)$ with augmented labeling $\psi_G$, we let $\hat{G}$ be the same graph as $G$ except that $\psi_G$ plays the role of primary labeling, i.e., $\hat{G} = (U,\alpha,\psi_G)$.
\secondaug*
\begin{proof} 
The implication (2) implies (1) is immediate from the definitions.
    If the secondary labeling scheme is isomorphism-invariant, then so is the augmented labeling scheme, since the primary labeling scheme is isomorphism-invariant by definition. Hence (1) implies (2).
    
     Now suppose (2) holds.  If $G \cong H$, then the isomorphism invariance of the augmented labeling scheme gives that $\hat{G} \cong \hat{H}$. Conversely, we always have that $\hat{G} \cong \hat{H}$ implies $G \cong H$. 
     Thus (2) implies (3).
     
     Finally, suppose (3) holds. If $f$ is an isomorphism from $G$ to $H$, then by (3), $f$ is an isomorphism from $\hat{G}$ to $\hat{H}$, which gives that the augmented labeling scheme is invariant under $f$. Hence (2) holds, which completes the proof.
\end{proof}

\subsection{The necessary condition for isomorphic graphs}
The proposition below states that if the labeling scheme $\Psi$ is isomorphism-invariant and two graphs are isomorphic, then the OGJ transport cost (with vertex cost function $c_{\Psi}$) between them is zero. 
\necessary*
\begin{proof}
    Suppose $G \cong H$, and let $f : U \to V$ be a graph isomorphism from $G$ to $H$. Let $\gamma$ be the bijective weight joining corresponding to $f$. Note that the marginal function $r_{\gamma}$ is supported on the graph of $f$, and as $\Psi$ is isomorphism-invariant, we have that $c_{\Psi}(u,f(u)) = 0$ for all $u \in U$. Hence $\langle c_{\Psi}, r_{\gamma} \rangle = 0$.
    Then, since $\rho_{\Psi}$ is non-negative and $\rho_{\Psi}(G,H) \leq \langle c_{\Psi}, r_{\gamma'} \rangle$ for all $\gamma' \in \mathcal{J}(\alpha,\beta)$, we have $$0 \leq \rho_{\Psi}(G,H) \leq \langle c_{\Psi}, r_{\gamma} \rangle = 0,$$ 
    which completes the proof.
\end{proof}

\subsection{Results with informative labeling schemes} 
In this subsection, our goal is to prove our two main results about labeling schemes that are sufficiently informative to distinguish two graphs, which we recall below. 
\injectivesuff*
\weakinjectivesuff*
Observe that if $\Psi$ is injective, then it is immediate from the definitions that $\Psi$ is locally injective and has magic symbols. Hence, Proposition \ref{Prop:injectivesuff} is a special case of Theorem \ref{Thm:weakinjectivesuff}, and it suffices to prove the latter. Before proving Theorem \ref{Thm:weakinjectivesuff}, we first establish some preliminary results.

In what follows, we use $(\mathcal{G},\Psi)$ to denote a graph family $\mathcal{G}$ with an augmented labeling scheme $\Psi$. The following lemma states that if the pair $(u,v)$ has positive weight for a zero-cost weight joining, then $u$ and $v$ must share the same label. 
\begin{restatable}[]{lemma}{samelabel}
\label{Lem:samelabel}
    Let $G$ and $H$ be two graphs in $(\mathcal{G},\Psi)$. Suppose $\rho_{\Psi}(G,H) = 0$, and let $\gamma$ be an optimal weight joining. Then for every $(u,v) \in \spp(r_{\gamma})$, we have $\psi_G(u) = \psi_H(v)$. 
\end{restatable}
\begin{proof}
    Let $(u,v) \in \spp(r_{\gamma})$. Since $r_{\gamma}(u,v) > 0$ and $\langle c_{\Psi}, r_{\gamma} \rangle = 0$, we must have $c_{\Psi}(u,v)=0$, and hence  $\psi_G(u) = \psi_H(v)$.
\end{proof}

Next, we introduce some ingredients required for our proof. Consider a graph $G = (U, \alpha, \phi_G)$. 
Recall that a finite sequence $u_0 u_1 \ldots u_l$ is a (finite) path in $G$ if each $u_k$ is in $U$ and we have $\alpha(u_k,u_{k+1}) > 0$ for all $k = 0, \ldots, l-1$. Given two graphs $G = (U, \alpha, \phi_G)$ and $H = (V, \beta, \phi_H)$ in $(\mathcal{G},\Psi)$, we define the projection map $\pr_G : U \times V \to U$ by setting $\pr_G(u,v)=u$ for all $(u,v) \in U \times V$. The projection map $\pr_H : U \times V \to V$ is defined analogously. 
Finally, recall that for a vertex $u$ in a graph $G = (U,\alpha,\phi_G)$, we let $N_G(u) = \{u' \in U : \alpha(u,u')  > 0 \}$ denote the neighborhood of $u$ in $G$. 
\begin{definition} \label{Def:localinjectivegraphdiamond}
    Consider two graphs $G$ and $H$ in $(\mathcal{G},\Psi)$, and let $K = (U \times V, \gamma, \phi_G \otimes \phi_H)$ be a graph joining of $G$ and $H$. We say that $\pr_G$ (resp.~$\pr_H$) is {\it locally injective} on $K$ if for any $(u,v) \in \spp(r_\gamma)$, the restriction of $\pr_G$  (resp.~$\pr_H$) to $N_K(u,v)$ is an injection. 
\end{definition}

The next lemma provides a connection between local injectivity of augmented label functions and projection maps. 
\begin{restatable}[]{lemma}{nographdiamond}
\label{Lem:nographdiamond}
    Consider two graphs $G$ and $H$ in $(\mathcal{G},\Psi)$. Suppose $\rho_{\Psi}(G,H) = 0$ and the labeling scheme $\Psi$ is locally injective. Then for any optimal graph joining $K$, the maps $\pr_G$ and $\pr_H$ are locally injective on $K$.
\end{restatable}
\begin{proof}
Let $\gamma$ be the optimal weight joining associated with $K$, with marginal function $r_{\gamma}$. Suppose for contradiction that $\pr_G$ is not locally injective. Then there exists some $(u_0,v_0)$ in $\spp(r_{\gamma})$ with $(u_1,v_1), (u_1,v'_1) \in N_K(u_0,v_0)$ and $v_1 \neq v'_1$. 
Since $\rho_{\Psi}(G,H) = 0$ and we have both $r_{\gamma}(u_1,v_1) > 0$ and $r_{\gamma}(u_1,v'_1) > 0$, Lemma \ref{Lem:samelabel} gives that $\psi_H(v_1) = \psi_G(u_1) = \psi_H(v'_1)$. In addition, by the transition coupling condition, we have that both $v_1$  and $v'_1$ are in $N_H(v_0)$. Combining these two observations, we obtain a contradiction to the local injectivity assumption of $\psi_H$. Hence, $\pr_G$ is locally injective. Interchanging the roles of $G$ and $H$ gives the conclusion for $\pr_H$. 
\end{proof}

\begin{restatable}[]{proposition}{sufficient}
\label{Prop:sufficient} 
    Let $\mathcal{G}$ be a family of connected graphs with augmented labeling scheme $\Psi$.
    Suppose $\Psi$ is locally injective and has magic symbols. 
    For any two graphs $G$ and $H$ in $\mathcal{G}$, if $\rho_{\Psi}(G,H)=0$, then $G \cong H$, and furthermore there is a unique optimal weight joining, which is bijective.
\end{restatable}

\begin{proof} 
Let $\gamma$ be an optimal weight joining and $K$ be its corresponding graph joining. 
Let $v_* \in V_j$ be a magic vertex in $V$, i.e., there is a label $w \in \mathcal{L}$ such that $v_* \in V$ is the unique vertex in $V$ such that $\psi_H(v_*) = w$. 
Since $H$ is fully supported, we have $q(v_*)>0$, and then since $r_{\gamma}$ is a coupling of $p$ and $q$, there is some $u_* \in U$ satisfying $r_{\gamma}(u_*,v_*)>0$. 
Hence, by Lemma \ref{Lem:samelabel}, we have $\psi_G(u_*) = \psi_H(v_*) = w$.

Now let us show that $\gamma$ is bijective. Let $u \in U$ be arbitrary, and suppose $(u,v), (u,v') \in \spp(r_{\gamma})$. 
Our aim is to show $v=v'$. Indeed, because $G$ is connected, there exists a path $u_0 \dots u_k$ in $G$ that starts at $u_0 = u$ and ends at $u_k = u_*$. Since $r_{\gamma}(u,v) > 0$, by the transition coupling condition, there exists a path $v_0 \ldots v_k$ in $H$ such that $v_0 = v$, and $\Gamma = (u_0,v_0) (u_1,v_1) \ldots (u_k,v_k)$ is a path in $K$. 
Similarly, there is a path $v'_0 \ldots v'_k$ in $H$ satisfying $v'_0 = v'$, and $\Gamma' = (u_0,v'_0) (u_1,v'_1) \ldots (u_k,v'_k)$ is a path in $K$. Since $u_k = u_*$, we have $r_{\gamma}(u_*,v_k) > 0$ and $r_{\gamma}(u_*,v'_k) > 0$. By Lemma \ref{Lem:samelabel}, we obtain $\psi_H(v_k) = \psi_G(u_*) = \psi_H(v'_k)$ 
Since $v_*$ is the unique vertex in $V$ whose image is $w = \psi_G(u_*)$, we must have $v_k =  v_* = v'_k$. Thus, $\Gamma$ and $\Gamma'$ are two paths that end at the same vertex $(u_*,v_*)$ in $K$. 

Next, note that $(u_{k-1},v_{k-1}) (u_*,v_*)$ and $(u_{k-1},v'_{k-1}) (u_*,v_*)$ are two paths in $K$, and 
$$\pr_G(u_{k-1},v_{k-1}) = \pr_G(u_{k-1},v'_{k-1}) = u_{k-1}.$$
By assumption, $\Psi$ is locally injective, and hence, by Lemma \ref{Lem:nographdiamond}, the map $\pr_G$ is also locally injective. Thus, we obtain $v_{k-1} =  v'_{k-1}$. Arguing inductively for decreasing $i \in \{0,\dots,k-1\}$, we get that $v_i = v'_i$ for all $i = 0,\dots,k-1$. 
Since $v_0 = v$ and $v'_0 = v'$, this gives $v=v'$. Since $u \in U$ was arbitrary, we have shown that for each $u \in U$, there is a unique $v \in V$ such that $r_{\gamma}(u,v) >0$. Interchanging the roles of $G$ and $H$, we also obtain that for all $v \in V$, there is a unique $u \in U$ such that $r_{\gamma}(u,v) >0$. Thus, we conclude that $\gamma$ is a bijective weight joining. 
By Proposition \ref{Prop:bijectiveiso}, there is a graph isomorphism $f$ from $G$ to $H$. 

Since $\gamma \in \mathcal{J}^*_{\Psi}(\alpha,\beta)$ was arbitrary, we have shown that every optimal weight joining is bijective. By Proposition \ref{Prop:bijective_extreme}, bijective weight joinings are always extreme points. Thus, every optimal weight joining is an extreme point. Since $\mathcal{J}^*_{\Psi}(\alpha,\beta)$ is convex and consists solely of optimal weight joinings, it follows that all elements of $\mathcal{J}^*_{\Psi}(\alpha,\beta)$ are extreme points. Since $\mathcal{J}^*_{\Psi}(\alpha,\beta)$ is nonempty, it must be a singleton. Therefore, we conclude that there is a \textit{unique} optimal weight joining, and it is bijective.
\end{proof}

\begin{proof}[Proof of Theorem \ref{Thm:weakinjectivesuff}]
	By Proposition \ref{Prop:necessary}, we obtain one direction of detection:
	$$G \cong H \Longrightarrow \rho_{\Psi}(G,H) = 0.$$
	The converse direction is given by Proposition \ref{Prop:sufficient}. Also, if $G$ and $H$ are isomorphic graphs in $\mathcal{G}$, then Proposition \ref{Prop:necessary} gives $\rho_{\Psi}(G,H) = 0$, and then by Proposition \ref{Prop:sufficient}, the optimal solution set $\mathcal{J}^*_{\Psi}(\alpha,\beta)$ contains exactly one element, which is a zero-cost bijective weight joining. Thus, we obtain the identification result. 
\end{proof}

\subsection{Results with process-level labeling schemes}
In this subsection we prove the results regarding process-level labeling schemes. Recall from Section \ref{Sect:ProcessLevel} that $\Psi^*$ denotes the process-level labeling scheme.
\equivcost*

Before proving this theorem, we first establish the following proposition. 
\begin{restatable}[]{proposition}{zeroexpectedcost}
\label{Prop:zeroexpectedcost}
    Let $G = (U, \alpha, \phi_G)$ and $H = (V, \beta, \phi_H)$ be fully supported weighted graphs with augmented labeling scheme $\Psi$.  
    For $\gamma \in \mathcal{J}(\alpha,\beta)$, we have $\langle c_{\Psi},r_{\gamma} \rangle = 0$ if and only if $\langle c_{\Psi^*},r_{\gamma} \rangle = 0$. 
\end{restatable}

\begin{proof}
Suppose $\langle c_{\Psi},r_{\gamma} \rangle = 0$. Let $\gamma$ be the optimal solution of this OGJ problem, and $(\widetilde{X},\widetilde{Y})$ be its associated reversible Markov chain as in Section \ref{Sect:MCs}. Since $\langle c_{\Psi}, r_{\gamma} \rangle = 0$ and $r_{\gamma}$ is the distribution of $(\widetilde{X}_0,\widetilde{Y}_0)$, we get that $c_{\Psi}(\widetilde{X}_0,\widetilde{Y}_0) = 0$ a.s. Then, since $(\widetilde{X},\widetilde{Y})$ is stationary, for all $k \geq 0$, we see that $c_{\Psi}(\widetilde{X}_k,\widetilde{Y}_k) = 0$ a.s. Thus, for all $k \geq 0$, we have $\psi_G(\widetilde{X}_k)= \psi_H(\widetilde{Y}_k)$ a.s., which means $\psi_G(\widetilde{X})= \psi_H(\widetilde{Y})$ a.s. Now let $(u,v) \in U \times V$ such that $r_{\gamma}(u,v) > 0$. Since $\psi_G(\widetilde{X})= \psi_H(\widetilde{Y})$ a.s., we get
\begin{align*}
    \sum_{(u',v') \in U \times V} \mathbb{P}(\psi_G(\widetilde{X}) \neq \psi_H(\widetilde{Y}) \mid \widetilde{X}_0 = u', \widetilde{Y}_0 = v' ) r_{\gamma}(u',v') = 0.
\end{align*}
As the right-hand side is zero and $r_{\gamma}(u,v) > 0$, we must have
\begin{align*}
    \mathbb{P}(\psi_G(\widetilde{X}) \neq \psi_H(\widetilde{Y}) \mid \widetilde{X}_0 = u, \widetilde{Y}_0 = v)=0.
\end{align*}
But this means $\psi_G(\widetilde{X}^u) = \psi_H(\widetilde{Y}^v)$ a.s. On the other hand, since $(\widetilde{X},\widetilde{Y})$ is a coupling of $X$ and $Y$, we have $\widetilde{X}^u \overset{d}{=} X^u$ and $\widetilde{Y}^v \overset{d}{=} Y^v$. Thus,
\begin{align*}
    \psi_G(X^u) \overset{d}{=} \psi_G(\widetilde{X}^u) \overset{d}{=} \psi_H(\widetilde{Y}^v) \overset{d}{=} \psi_H(Y^v),
\end{align*}
which means $\psi_G^*(u) = \psi_H^*(v)$. This gives that $c_{\Psi^*}(u,v)=0$. Hence, for any $(u,v) \in U \times V$ satisfying $r_{\gamma}(u,v) > 0$, we have $c_{\Psi^*}(u,v)=0$. Therefore, $\langle c_{\Psi^*},r_{\gamma} \rangle = 0$. 

Conversely, assume that $\langle c_{\Psi^*},r_{\gamma} \rangle = 0$. Let $\gamma$ be the optimal solution of this OGJ problem. 
Let $(u,v) \in U \times V$ such that $r_{\gamma}(u,v) > 0$. Since $\langle c_{\Psi^*}, r_{\gamma} \rangle = 0$, we see that $\psi_G(X^u)$ is equal in distribution to $\psi_H(Y^v)$, and hence $\psi_G(X_0^u)$ is equal in distribution to $\psi_H(Y_0^v)$. As $X_0^u = u$ and $Y_0^v = v$, 
we get $\psi_G(u) = \psi_H(v)$. Thus, we obtain $\langle c_{\Psi},r_{\gamma} \rangle = 0$.
\end{proof}

\begin{proof}[Proof of Theorem \ref{Thm:equivcost}]
    For any $\gamma \in J(\alpha,\beta)$, by Proposition \ref{Prop:zeroexpectedcost}, we have $\langle c_{\Psi},r_{\gamma} \rangle = 0$ if and only if $\langle c_{\Psi^*},r_{\gamma} \rangle = 0$. Thus, $\rho_{\Psi}(G,H)=0$ if and only if $\rho_{\Psi^*}(G,H)=0$, which proves the first sentence of the theorem, as well as fact that OGJ with cost $c_{\Psi}$ detects isomorphism if and only if OGJ with cost $c_{\Psi^*}$ detects isomorphism. Furthermore, when these costs are zero, Proposition \ref{Prop:zeroexpectedcost} gives that $\mathcal{J}^*_{\Psi}(\alpha,\beta) = \mathcal{J}^*_{\Psi^*}(\alpha,\beta)$. Hence the extreme points of $\mathcal{J}^*_{\Psi}(\alpha,\beta)$ coincide with the extreme points of $\mathcal{J}^*_{\Psi^*}(\alpha,\beta)$. Thus, the extreme points of $\mathcal{J}^*_{\Psi}(\alpha,\beta)$ are bijective weight joinings with $c_{\Psi}$-zero cost if and only if the extreme points of $\mathcal{J}^*_{\Psi^*}(\alpha,\beta)$ are bijective weight joinings with $c_{\Psi^*}$-zero cost. This proves the identification part of the theorem. 
\end{proof}

Next, we show that OGJ can detect and identify isomorphism for graphs with complete sets of landmarks. 
Recall that the definitions for graphs with landmarks are given in Definition \ref{landmarks}.
\landmarks*
\begin{proof}
Let $\Phi^*$ be the process-level labeling scheme. First we show that $\Phi^*$ is injective. To that end, let $G \in \mathcal{G}$, and let $u \neq u'$ be in $U$. Since $G$ has a complete set of landmarks, there is label $a \in \mathcal{L}^*$ such that $\dst_G(u, a) \neq \dst_G(u', a)$. 
Without loss of generality, assume that $\dst_G(u, a) < \dst_G(u', a)$ and let $m = \dst_G(u, a)$. If $\{X_k\}_{k=0}^{\infty}$ is the random walk on $G$, then $\mathbb{P}(\phi_G(X_m) = a \mid X_0 = u) >0$, whereas $\mathbb{P}(\phi_G(X_m) = a \mid X_0 = u') = 0$. Hence $\phi_G(X^u)$ and $\phi_G(X^{u'})$ are not equal in distribution, and thus $\Phi^*$ is injective.
Since $\Phi^*$ is injective, by Proposition \ref{Prop:injectivesuff}, OGJ with cost $c_{\Phi^*}$ detects and identifies isomorphism for $\mathcal{G}$, and then by Theorem \ref{Thm:equivcost}, we conclude that OGJ with cost $c_{\Phi}$ detects and identifies isomorphism for $\mathcal{G}$.
\end{proof}

\subsection{Trees}
In this section we provide proofs for our results concerning detection and identification of isomorphism for weighted, labeled trees. 
Recall the statement of the main theorem for trees.

\trees*
Before proving this theorem, we require some preliminary definitions and results.

\begin{definition} \label{Def:restrictionweight}
Let $\alpha$ be a weight function on $U$ with marginal function $p$. Suppose $U_0 \subset U$. First, we define
\begin{equation*}
Z_{\alpha}(U_0) = \sum_{u,u' \in U_0} \alpha(u,u'),
\end{equation*}
and for $u \in U$, we define
\begin{equation*}
    Z_{\alpha}(u,U_0) = \sum_{u' \in U_0} \alpha(u,u').
\end{equation*}
If $Z_{\alpha}(U_0)>0$, then we define the \textit{restriction} of $\alpha$ to $U_0$ to be $\alpha_0 : (U_0)^2 \to [0,1]$, where
\begin{equation*}
\alpha_0(u,u') = \frac{\alpha(u,u')}{Z_{\alpha}(U_0)}.
\end{equation*}
\end{definition}

\begin{remark} \label{Rem:resweight}
The restriction $\alpha_0$ defined in Definition \ref{Def:restrictionweight} is a weight function on $U_0$. 
\end{remark}

\begin{lemma} \label{Lem:restriction}
Let $G = (U,\alpha,\phi_G)$ and $H = (V,\beta,\phi_H)$ be fully supported graphs, and let $U_0 \subset U$ and $V_0 \subset V$ be nonempty. Let $\alpha_0$ and $\beta_0$ be the restrictions of $\alpha$ and $\beta$ to $U_0$ and $V_0$, respectively. Further suppose that $\gamma \in \mathcal{J}(\alpha,\beta)$ satisfies $\spp(r_{\gamma}) \subset (U_0 \times V_0) \cup ((U \setminus U_0) \times (V \setminus V_0))$. Then $Z_{\alpha}(U_0) = Z_{\beta}(V_0) = Z_{\gamma}(U_0 \times V_0)$ and the restriction of $\gamma$ to $U_0 \times V_0$ is a weight joining of $\alpha_0$ and $\beta_0$. 
\end{lemma}
\begin{proof}
Suppose the hypotheses of the lemma. Let us first show that 
$$Z_{\alpha}(U_0) = Z_{\beta}(V_0) = Z_{\gamma}(U_0 \times V_0).$$
To begin, we note that by the assumption on $\spp(r_{\gamma})$ and the edge coupling property \eqref{eqn:EdgeCoupling1}, we have
\begin{align*}
Z_{\gamma}(U_0 \times V_0) 
& = \sum_{u,u' \in U_0} \sum_{v,v' \in V_0}  \gamma((u,v),(u',v'))  \\
& = \sum_{u,u' \in U_0} \sum_{v,v' \in V}  \gamma((u,v),(u',v'))  \\
& = \sum_{u,u' \in U_0} \alpha(u,u') \\
& = Z_{\alpha}(U_0).
\end{align*}
Similarly, we have $Z_{\gamma}(U_0 \times V_0) = Z_{\beta}(V_0)$.

Now let $Z$ be the common value of $Z_{\gamma}(U_0,V_0)$, $Z_{\alpha}(U_0)$, and $Z_{\beta}(V_0)$. Note that $Z >0$, since $U_0$ and $V_0$ are nonempty and $G$ and $H$ are fully supported.
For notation, let $\gamma_0$ denote the restriction of $\gamma$ to $U_0\times V_0$ (which is well-defined, since $Z>0$), and let $r_0$ be its associated marginal function. 
Let $p_0$ and $q_0$ denote the marginal functions for the restrictions $\alpha_0$ and $\beta_0$, respectively. 
Recall from Definition \ref{Def:restrictionweight} that for $u \in U_0$ and $v \in V_0$, we have the notation
\begin{align*}
Z_{\gamma}((u,v),U_0 \times V_0) & = \sum_{(u',v') \in U_0 \times V_0} \gamma((u,v),(u',v')) \\
Z_{\alpha}(u,U_0) & = \sum_{u' \in U_0} \alpha(u,u') \\
Z_{\beta}(v,V_0) & = \sum_{v' \in V_0} \beta(v,v').
\end{align*}
It is immediate from these definitions that we have $r_0(u,v) = Z_{\gamma}((u,v),U_0 \times V_0)/Z$, $p_0(u) = Z_{\alpha}(u,U_0)/Z$, and $q_0(v) = Z_{\beta}(v,V_0)/Z$. Additionally, if $r_{\gamma}(u,v)>0$, then by the assumption on $\spp(r_{\gamma})$ and the transition coupling condition, we have
\begin{equation}
\begin{aligned}
\frac{Z_{\gamma}((u,v),U_0 \times V_0)}{r_{\gamma}(u,v)} & = \frac{1}{r_{\gamma}(u,v)} \sum_{(u',v') \in U_0 \times V_0} \gamma((u,v),(u',v')) \\
& = \sum_{u' \in U_0} \frac{1}{r_{\gamma}(u,v)}  \sum_{v' \in V_0} \gamma((u,v),(u',v')) \\
&  = \sum_{u' \in U_0} \frac{1}{r_{\gamma}(u,v)}  \sum_{v' \in V} \gamma((u,v),(u',v')) \\
& = \sum_{u' \in U_0} \frac{\alpha(u,u')}{p(u)} \\
& = \frac{Z_{\alpha}(u,U_0)}{p(u)},
\label{ratio_equation}
\end{aligned}
\end{equation}
Similarly, we get
\begin{equation*}
\frac{Z_{\gamma}((u,v),U_0 \times V_0)}{r_{\gamma}(u,v)} = \frac{Z_{\beta}(v,V_0)}{q(v)}.
\end{equation*}

By Remark \ref{Rem:resweight}, $\gamma_0$ is a weight function on $U_0 \times V_0$. Next let us show that $\gamma_0$ satisfies the marginal coupling condition and the transition coupling condition. For the marginal coupling condition, for any $u \in U_0$, by the condition on the support of $r_{\gamma}$ and the edge coupling property \eqref{eqn:EdgeCoupling1}, we have
\begin{align*}
\sum_{v \in V_0} r_0(u,v) & = \frac{1}{Z} \sum_{v \in V_0} \sum_{(u',v') \in U_0 \times V_0} \gamma((u,v),(u',v')) \\
& =  \frac{1}{Z} \sum_{u' \in U_0} \sum_{v,v' \in V_0} \gamma((u,v),(u',v')) \\
& = \frac{1}{Z} \sum_{u' \in U_0} \sum_{v,v' \in V} \gamma((u,v),(u',v')) \\
& = \frac{1}{Z} \sum_{u' \in U_0} \alpha(u,u') \\
& = p_0(u).
\end{align*}
This establishes the marginal coupling condition for $r_0$ with respect to $p_0$. The marginal coupling condition with respect to $q_0$ is verified analogously.

Let us now show that $\gamma_0$ satisfies the transition coupling condition for $\alpha_0$ and $\beta_0$. Let $(u,v) \in U_0 \times V_0$. If $r_0(u,v) = 0$, then the condition is trivially satisfied. Suppose $(u,v) \in U_0 \times V_0$ satisfies $r_0(u,v) >0$. Then there exists $(u',v') \in U_0 \times V_0$ such that $\gamma((u,v),(u',v'))>0$. Hence $r_{\gamma}(u,v)>0$, $\alpha(u,u')>0$ and $\beta(v,v') >0$, which then gives that $p_0(u) >0$ and $q_0(v)>0$. Let $u' \in U_0$. Then by \eqref{ratio_equation}, the transition coupling condition for $\gamma$, and the facts $p_0(u) = Z_{\alpha}(u,U_0)/Z$ and $\alpha_0(u,u') = \alpha(u,u')/Z$, we obtain that 
\begin{align*}
\frac{1}{r_0(u,v)} \sum_{v' \in V_0} \gamma_0((u,v),(u',v')) & = \frac{r_{\gamma}(u,v)}{r_0(u,v)} \cdot \frac{1}{r_{\gamma}(u,v)} \sum_{v' \in V_0} \frac{\gamma((u,v),(u',v'))}{Z} \\
& = \frac{ r_{\gamma}(u,v)}{Z_{\gamma}((u,v),U_0 \times V_0)/Z} \cdot \frac{1}{Z} \cdot \frac{1}{r_{\gamma}(u,v)} \sum_{v' \in V} \gamma((u,v),(u',v'))\\
& = \frac{p(u)}{Z_{\alpha}(u,U_0)} \cdot \frac{\alpha(u,u')}{p(u)} \\
& = \frac{\alpha(u,u')}{Z_{\alpha}(u,U_0)}\\
& = \frac{\alpha_0(u,u')}{p_0(u)}.
\end{align*}
This establishes the transition coupling condition with respect to $\alpha_0$.
The proof of the transition coupling condition with respect to $\beta_0$ is analogous.
\end{proof}

For two graphs $G$ and $H$, recall that $L(G)$ and $L(H)$ denote the sets of leaves of $G$ and $H$, respectively. 
\begin{lemma} \label{Lem:leaves}
Let $G = (U,\alpha,\phi_G)$ and $H = (V,\beta,\phi_H)$ be two graphs with secondary labeling given by the multiweight function and augmented labeling scheme denoted by $\Psi$. Suppose $\gamma \in \mathcal{J}(\alpha,\beta)$ has zero cost with respect to $c_{\Psi}$. Then 
\begin{equation} \label{Eqn:Asheville}
    \spp(r_{\gamma}) \subset (L(G) \times L(H)) \cup ((U \setminus L(G)) \times (V \setminus L(H))).
\end{equation}
\end{lemma}
\begin{proof}
    Since $\gamma$ has cost zero, Lemma \ref{Lem:samelabel} gives us that if $r_{\gamma}(u,v)>0$, then $\psi_G(u) = \psi_H(v)$. Since the secondary labeling is the multiweight function and leaves are recognizable from the multiweight function (as they are precisely the vertices incident to exactly one non-zero edge), if $\psi_G(u) = \psi_H(v)$, then $u \in L(G)$ if and only if $v \in L(H)$. Combining the previous two sentences, we get \eqref{Eqn:Asheville}. 
\end{proof}

The following lemma gives conditions under which a weight joining yields a bijective alignment of the leaves attached to a vertex in $G$ with the leaves attached to a vertex in $H$. The proof relies on Birkhoff's Theorem \cite{birkhoff1946tres}.
\begin{lemma} \label{Lemma:Atlantic}
Let $G = (U,\alpha,\phi_G)$ and $H = (V,\beta,\phi_H)$ be two graphs with secondary labeling given by the multiweight function and augmented labeling scheme denoted by $\Psi$. Suppose $\gamma \in \mathcal{J}(\alpha,\beta)$ has zero cost with respect to $c_{\Psi}$ and there exists $u_0$ in  $U$ and $v_0$ in $V$ satisfying $r_{\gamma}(u_0,v_0) = p(u_0) = q(v_0)$ and $N_G(u_0) \cap L(G) \neq \varnothing$. Then there is a bijection $\sigma : N_G(u_0) \cap L(G) \to N_H(v_0) \cap L(H)$ such that $\gamma((u_0,v_0),(u,\sigma(u))) >0$ for all $u \in N_G(u_0) \cap L(G)$.
\end{lemma}

\begin{proof}
Let $I_G(u_0) = N_G(u_0) \cap L(G)$ and $I_H(v_0) = N_H(v_0) \cap L(H)$. 
Let $w \in \mathcal{L}$ be a label such that $\psi_G^{-1}(w) \cap I_G(u_0) \neq \emptyset$. Let $A_w = \psi_G^{-1}(w) \cap I_G(u_0)$ and $B_w = \psi_H^{-1}(w) \cap I_H(v_0)$. 
Next let $\mu$ be the probability distribution on $A_w$ given by $\mu(u) = \alpha(u_0,u)/Z$, where $Z = \sum_{u' \in A_w} \alpha(u_0,u')$, and let $\nu$ be the probability distribution on $B_w$ given by $\nu(v) = \beta(v_0,v)/Z'$, where $Z' = \sum_{v' \in B_w} \beta(v_0,v')$. 
Now let $\pi$ be the distribution on $A_w \times B_w$ given by $\pi(u,v) = \gamma((u_0,v_0),(u,v))/Z''$, where $Z'' = \sum_{(u',v') \in A_w \times B_w} \gamma((u_0,v_0),(u',v'))$. 

First, we claim that $Z = Z' = Z''$. We begin by showing that when $\gamma((u_0,v_0),(u,v)) > 0$, we have $u \in A_w$ if and only if $v \in B_w$. Indeed, suppose $u \in A_w$. Since $\gamma((u_0,v_0),(u,v)) > 0$, we have $r_{\gamma}(u,v) > 0$. By Lemma \ref{Lem:samelabel}, $\psi_G(u) = \psi_H(v)$. Thus, $\psi_H(v) = \psi_G(u) = w$. Because $\gamma((u_0,v_0),(u,v)) > 0$, we have $v \in N_H(v_0)$ by the edge preservation property. By Lemma \ref{Lem:leaves}, we get $v \in L(H)$. Combining these three facts, we get $v \in B_w$. Reversing the role of $A_w$ and $B_w$ gives the converse direction. Now, we shall show $Z = Z''$. 
Indeed, by the fact that if $\gamma((u_0,v_0),(u,v))>0$, then $u\in A_w$ if and only if $v\in B_w$, the transition coupling condition, and the hypothesis that $r_{\gamma}(u_0,v_0) = p(u_0)$, we get
\begin{align*}
\sum_{(u,v) \in A_w \times B_w} \gamma((u_0,v_0),(u,v)) & = r_{\gamma}(u_0,v_0) \sum_{u \in A_w} \frac{1}{r_{\gamma}(u_0,v_0)} \sum_{v \in B_w}  \gamma((u_0,v_0),(u,v)) \\
& = r_{\gamma}(u_0,v_0) \sum_{u \in A_w} \frac{1}{p(u_0)} \alpha(u_0,u) \\
& =  \sum_{u \in A_w} \alpha(u_0,u).
\end{align*} 
This shows that $Z'' = Z$. 
Similarly, we get $Z'' = Z'$. 

Next note that $\pi$ is a coupling of $\mu$ and $\nu$ by the transition coupling condition (again using the hypothesis that $r_{\gamma}(u_0,v_0) = p(u_0) = q(v_0)$). By definition of $A_w$, we have $\psi_G(u) = w$ for all $u \in A_w$, and similarly, $\psi_H(v) = w$ for all $v \in B_w$. Since the multiweight function at any leaf $u \in I_G(u_0)$ consists of only one non-zero value ($\alpha(u,u_0)$), we conclude that there is a single value $a>0$ (determined only by $w$) such that for all $u \in A_w$ and $v \in B_w$, we have $\alpha(u,u_0) = a = \beta(v,v_0)$. Then by symmetry and the fact that $Z = Z'$, we get that $\mu$ and $\nu$ are identical uniform distributions, and $|A_w| = |B_w|$.   
Then by Birkhoff's Theorem \cite{birkhoff1946tres}, there is a bijection $\sigma_w : A_w \to B_w$ such that $\pi(u,\sigma_w(u))>0$ for all $u \in A_w$, and hence $\gamma((u_0,v_0),(u,\sigma_w(u))) >0$ for all $u \in A_w$. Next observe that $\{A_w : w \in \mathcal{L}, \, \psi_G^{-1}(w) \cap I_G(u_0) \neq \emptyset\}$ forms a partition of $I_G(u_0)$ and $\{B_w : w \in \mathcal{L}, \,  \psi_H^{-1}(w) \cap I_H(v_0) \neq \emptyset\}$ forms a partition of $I_H(v_0)$, and furthermore $\psi_G^{-1}(w) \cap I_G(u_0) \neq \emptyset$ if and only if $\psi_H^{-1}(w) \cap I_H(v_0) \neq \emptyset$ (since $\gamma$ has zero cost).  Thus, there is a bijection $\sigma : I_G(u_0) \to I_H(v_0)$, defined by $\sigma(u) = \sigma_w(u)$ for $u \in A_w$, and we note that $\gamma((u_0,v_0),(u,\sigma(u))) >0$ for all $u \in I_G(u_0)$, as desired.
\end{proof}

The proof of the following proposition contains the main argument in our proof of Theorem \ref{Thm:trees}.

\begin{restatable}[]{proposition}{treesufficient}
\label{Prop:treesufficient}
    Under the hypotheses of Theorem \ref{Thm:trees}, let $G = (U, \alpha, \phi_G)$ and $H = (V, \beta, \phi_H)$ be two trees in $\mathcal{G}$ such that $\rho_{\Psi}(G,H) = 0$. Then every extreme point of $\mathcal{J}^*_{\Psi}(\alpha, \beta)$ is bijective.
\end{restatable}

The outline of the proof of this proposition is as follows. We argue inductively on the size of the tree. In the induction step, we remove the leaves and consider the restrictions of the weight joining to the resulting subtrees. We then apply the induction hypothesis to get a graph isomorphism on these subtrees, and finally we extend this isomorphism to include the leaves.

\begin{proof}
We proceed by induction on the maximum number of nodes in the tree. For notation, we let $\max(|U|,|V|) = n$.

For the base case, suppose $n =2$. As we consider only trees with at least two nodes, we have $|U| = |V| = 2$. Let $U = \{u_0,u_1\}$, and note that we must have $\alpha(u_0,u_1) = 1/2$. Similarly, let $V = \{v_0,v_1\}$,  and note that $\beta(v_0,v_1)=1/2$. Now let $\gamma$ be an extreme optimal weight joining, so in particular we have $\langle c_{\Psi},r_{\gamma} \rangle = 0$. In this simple case the direct product graph $G \otimes H$ consists of two disjoint edges. Since $\gamma$ is extremal, Proposition \ref{Prop:preserveconnected} gives that the graph joining associated with $\gamma$ must be connected, and hence $\gamma$ must be supported on exactly one of the two edges in $G \otimes H$. Thus there exists a permutation $\sigma : \{0,1\} \to \{0,1\}$ such that $r_{\gamma}(u_i,v) >0$ if and only if $v = v_{\sigma(i)}$. 
Since $\langle c_{\Psi},r_{\gamma} \rangle = 0$, by Lemma \ref{Lem:samelabel}, we obtain $\psi_G(u_i) = \psi_H(v_{\sigma(i)})$. Let $f : U \to V$ be defined by $f(u_i) = v_{\sigma(i)}$. Observe that $f$ is a graph isomorphism from $G$ to $H$. 
Let $\gamma_f$ be the bijective weight joining corresponding to $f$. By the definition of $\sigma(i)$, we see that $\spp(\gamma_f) \subset \spp(\gamma)$. Since $\gamma$ is extreme, we conclude that $\gamma_f = \gamma$ by Proposition \ref{Prop:minimalextreme}, and therefore $\gamma$ is bijective.

Now suppose for induction that there exists $n \geq 2$ such that for all trees $G$ and $H$ with at most $n$ vertices, if $\rho_{\Psi}(G,H)=0$, then every extreme point of $\mathcal{J}^*_{\Psi}(\alpha,\beta)$ is bijective. To establish this statement for trees with at most $n+1$ vertices, let $G$ and $H$ be trees such that $\max(|U|,|V|)= n+1$, and suppose $\rho_{\Psi}(G,H)=0$. We assume without loss of generality (by reversing the roles of $G$ and $H$ if necessary) that $|U| = n+1$. Let $\gamma$ be an extreme optimal weight joining. It remains to show that $\gamma$ is bijective.

Let us now delete the leaves and consider the restrictions of the weight joining to the interiors of the trees. 
Let $U_0 = (U \setminus L(G))$ and $V_0 = (V \setminus L(H))$. Since $G$ is a tree with $n+1 \geq 3$ vertices, we have that $U_0$ is nonempty. By Lemma \ref{Lem:leaves}, we get
\begin{align}\label{Charlotte}
    \spp(r_{\gamma}) \subset (L(G) \times L(H)) \cup (U_0 \times V_0).
\end{align}
Due to \eqref{Charlotte}, $V_0$ is nonempty. 
Now we distinguish two cases: either $|U_0| = 1$ (in which case $G$ is a star graph) or $|U_0| \geq 2$ (in which case the restriction graph $G_0  = (U_0,\alpha_0,\phi_{G_0})$ is a non-trivial tree with at most $n-1$ vertices).  
 
\vspace{2mm}
\noindent
\textit{Case 1.} Suppose $|U_0| = 1$. Let $U_0 = \{u_0\}$, and note that $U = \{u_0\} \cup L(G)$. We claim that $|V_0| = 1$.  Suppose for contradiction that there exists $v \neq v'$ in $V_0$. Since $H$ is a tree, there exists a path $\Gamma = v_0 \dots v_k$ in $H$ such that $v_0 = v$, $v_k = v'$, and $v_i \in V_0$ for all $i$ (i.e., the path does not have to traverse any leaves). Since $v \neq v'$, we have $\beta(v_0, v_1)>0$. By Lemma \ref{Lem:leaves} and the fact that $U_0 = \{u_0\}$, if $r_{\gamma}(u,v_i) >0$, then $u = u_0$. Since $\beta(v_0,v_1)>0$, we must have $\gamma((u_0,v_0),(u_0,v_1)) >0$. Then by Proposition~\ref{Prop:realcoupling}, we must have $\alpha(u_0,u_0) >0$, which is a contradiction, since $G$ is a tree (so in particular it has no self-loops). Hence $|V_0| = 1$. Let $V_0 = \{v_0\}$. 

Since $U_0 = \{u_0\}$ and $V_0 = \{v_0\}$, the marginal coupling condition and Lemma \ref{Lem:leaves}  yield that $r_{\gamma}(u_0,v_0) = p(u_0) = q(v_0)$. 
We also have $N_G(u_0)\cap L(G) \neq \emptyset$, as $G$ is connected and $U\setminus \{u_0\} = L(G)$.
Then by Lemma \ref{Lemma:Atlantic}, there exists a bijection $\sigma : L(G) \to L(H)$ such that $\gamma((u_0,v_0),(u,\sigma(u))>0$ for all $u \in L(G)$. Now define the map $f : U \to V$ by setting $f(u_0) = v_0$ and $f(u) = \sigma(u)$ for all $u \in L(G)$. We claim that $f$ is a graph isomorphism from $G$ to $H$. 
First, note that $f$ is a bijection from $U$ to $V$ (since $\sigma$ is a bijection). Next, since $G$ is a star graph (a tree with exactly one non-leaf vertex), every edge with non-zero weight connects $u_0$ to some $u \in L(G)$. Suppose $\alpha(u_0,u) >0$. Then $\gamma((u_0,v_0),(u,f(u))) >0$ (by definition of $f$),  and therefore $\psi_G(u) = \psi_H(f(u))$, which gives that $\alpha(u_0,u) = \beta(v_0,f(u))$. 
Hence $\beta(f(u_0),f(u)) = \alpha(u_0,u)$. The same argument with the roles of $G$ and $H$ reversed gives that if $\beta(v_0,f(u)) >0$, then $\alpha(u_0,u) = \beta(f(u_0),f(u))$. 
Finally, note that for any $u \in U$, we have $r_{\gamma}(u,f(u)) >0$, and therefore, by Lemma \ref{Lem:samelabel}, $\psi_G(u) = \psi_H(f(u))$, which gives that $\phi_G(u) = \phi_H(f(u))$. We now conclude that $f$ is a graph isomorphism  from $G$ to $H$. 
Furthermore, the associated bijective weight joining $\gamma_f$ satisfies $\spp(\gamma_f) \subset \spp(\gamma)$. Since $\gamma$ is extremal, by Proposition~\ref{Prop:minimalextreme}, $\gamma$ has minimal support, and thus $\spp(\gamma_f) = \spp(\gamma)$. Hence $\gamma = \gamma_f$, and therefore $\gamma$ is bijective.

\vspace{2mm}
\noindent
\textit{Case 2.} Suppose $|U_0| \geq 2$. We first claim that $|V_0| \geq 2$. If not, then $|V_0| =1$ (we've already established that it is nonempty), and then by repeating the argument in Case 1 with the roles of $G$ and $H$ reversed, we get that $|U_0|=1$, which contradicts our assumption that $|U_0| \geq 2$. 

Now let $\alpha_0$ and $\beta_0$ be the restrictions of $\alpha$ and $\beta$ to $U_0$ and $V_0$, respectively. Note that the restriction graphs $G_0 = (U_0,\alpha_0,\phi_{G_0})$ and $H_0 = (V_0,\beta_0,\phi_{H_0})$ are non-trivial trees, and they each have at most $n-1$ vertices (since $L(G)$ and $L(H)$ are non-empty). By Lemma \ref{Lem:leaves}, the assumption on $\spp(r_{\gamma})$ of Lemma \ref{Lem:restriction} is satisfied. Hence, by Lemma \ref{Lem:restriction}, we get that the restriction of $\gamma$ to $U_0 \times V_0$, which we denote by $\gamma_0$, is a weight joining of $\alpha_0$ and $\beta_0$. Let $r_0$ denote the marginal function of $\gamma_0$. Also, by Lemma \ref{Lem:restriction}, we obtain $Z_{\alpha}(U_0) = Z_{\beta}(V_0) = Z_{\gamma}(U_0 \times V_0)$. 

Next let us show that for all $(u,v) \in \spp(r_{0})$, the multiweight of $u$ in $G_0$ is equal to the multiweight of $v$ in $H_0$.
Let $(u,v) \in \spp(r_{0})$. 
By the definition of restriction, we get $r_{\gamma}(u,v) >0$. 
By Proposition~\ref{Prop:zeroexpectedcost}, $\langle c_{\Psi}, r_{\gamma} \rangle = 0$ further yields $\langle c_{\Psi^*}, r_{\gamma} \rangle = 0$, which implies $\psi^*_G(u) = \psi^*_H(v)$.
Let $X$ and $Y$ denote the reversible Markov chains associated with $G$ and $H$, respectively.
Suppose that $u$ is adjacent to at least one leaf, and let $a$ be the weight of the edge connecting $u$ to such a leaf. Then the multiweight of $X_1^u$ equals $\{\!\{a\}\!\}$ with positive probability. 
Since $\psi^*_G(u)=\psi^*_H(v)$, the event that the multiweight of $Y_1^v$ equals $\{\!\{a\}\!\}$ has the same probability as the event that the multiweight of $X_1^u$ equals $\{\!\{a\}\!\}$. Thus, $v$ is also adjacent to a leaf, and the number of leaves attached to $v$ by edges of weight $a$ equals the number of leaves attached to $u$ by edges of weight $a$. As this is true for all edge weights $a$, we see that the multisets of edge weights of leaves adjacent to $u$ and to $v$ are in bijection.
Removing all leaves and renormalizing via $Z_{\alpha}(U_0) = Z_{\beta}(V_0)$ gives that  the multiweight of $u$ in $G_0$ is equal to the multiweight of $v$ in $H_0$, as we claimed.

Next let us use the induction hypothesis to obtain a graph isomorphism on the interiors of the trees. By the previous paragraph, we observe that the $\langle c_{\Psi}, r_{0} \rangle = 0$, and hence $\rho_{\Psi}(G_0,H_0) = 0$. 
Then  by the induction hypothesis and Proposition~\ref{Prop:extremebijective}, there is a bijective weight joining $\gamma_1 \in \mathcal{J}(\alpha_0,\beta_0)$ such that $\spp(\gamma_1) \subset \spp(\gamma_0)$.  Let $f_0 : U_0 \to V_0$ be the graph isomorphism induced by $\gamma_1$ from $G_0$ to $H_0$. 

Now let us extend $\gamma_1$ to a weight joining of $\alpha$ and $\beta$. 
Recall that $L(G) = U \setminus U_0$ and $L(H) = V \setminus V_0$. For notation, let $r_1$ denote the marginal function of $\gamma_1$. 
Observe that for all $u \in U_0$, we have $r_1(u,f_0(u)) = p_0(u) = q_0(f_0(u))$, since $\gamma_1$ is bijective. 
Let $Z$ be the common value of $Z_{\alpha}(U_0)$, $Z_{\beta}(V_0)$, and $Z_{\gamma}(U_0 \times V_0)$.
By the definition of restrictions, for all $u \in U_0$, we also get 
\begin{equation*}
\frac{Z_{\alpha}(u,U_0)}{Z} = p_0(u) = q_0(f_0(u)) = \frac{Z_{\beta}(f_0(u),V_0)}{Z}, 
\end{equation*}
which gives that $Z_{\alpha}(u,U_0) = Z_{\beta}(f_0(u),V_0)$. Since $\psi_G(u) = \psi_H(f_0(u))$, we also have that $p(u) = q(f_0(u))$, and together these facts yield that $Z_{\alpha}(u,L(G)) = Z_{\beta}(f_0(v),L(H))$ for all $u \in U_0$.

Now we define $\hat{\gamma} : (U \times V)^2 \to [0,1]$ as follows. For $((u,v),(u',v')) \in (U_0 \times V_0)^2$ we set
\begin{equation*}
\hat{\gamma}((u,v),(u',v')) =   Z \cdot \gamma_1((u,v),(u',v')).
\end{equation*}
For $u \in U_0$ and $(u',v') \in L(G) \times L(H)$, we let 
\begin{equation*}
\hat{\gamma}((u,f_0(u)),(u',v')) = Z_{\alpha}(u,L(G)) \cdot \frac{\gamma((u,f_0(u)),(u',v'))}{Z_{\gamma}((u,f_0(u)),L(G) \times L(H))},
\end{equation*}
and for symmetry, we also set
\begin{equation*}
\hat{\gamma}((u',v'), (u,f_0(u))) = Z_{\alpha}(u,L(G)) \cdot \frac{\gamma((u,f_0(u)),(u',v'))}{Z_{\gamma}((u,f_0(u)),L(G) \times L(H))},
\end{equation*}
with the convention that $\frac{0}{0} = 0$.
Finally, all other values of $\hat{\gamma}$ are defined to be zero. We now denote $\hat{r} (u,v) = \sum_{(u',v')\in U\times V}\hat{\gamma}((u,v),(u',v'))$ for all $(u,v)\in U\times V$.

\noindent
{\it Claim 1.}
For all pairs $(u,f_0(u))$ with $u \in U_0$, we have $\hat{r}(u_0,f_0(u)) = p(u_0) = q(f_0(u))$, and  $\hat{\gamma}$ is an optimal weight joining of $\alpha$ and $\beta$ with zero OGJ cost under $c_{\Psi}$.
\begin{proof}[Proof of Claim 1]
By definition, we have that $\hat{\gamma}$ is non-negative and symmetric. For the normalization condition, we note that
\begin{align*}
\sum_{(u,v),(u',v') \in (U \times V)} \hat{\gamma}((u,v),(u',v')) & = Z \sum_{(u,v),(u',v') \in (U_0 \times V_0)} \gamma_1((u,v),(u',v')) \\
& \quad + 2 \sum_{u \in U_0}  \sum_{(u',v') \in L(G) \times L(H)} \frac{Z_{\alpha}(u,L(G)) \gamma((u,f_0(u)),(u',v'))}{Z_{\gamma}((u,f_0(u)),L(G) \times L(H))} \\
& = Z + 2 \sum_{u \in U_0} Z_{\alpha}(u,L(G)) \\
& = \sum_{u \in U_0} Z_{\alpha}(u,U_0) + 2 \cdot Z_{\alpha}(u,L(G)) \\
& = \sum_{u,u' \in U} \alpha(u,u') \\
& = 1.
\end{align*}

Note that $\spp(\hat{r}) \subset \{(u,f_0(u)) : u \in U_0\} \cup L(G) \times L(H)$. Also, for $u \in U_0$, we have
\begin{align*}
\hat{r}(u,f_0(u)) & = Z r_1(u,f_0(u)) + \sum_{(u',v') \in L(G) \times L(H)} \frac{ Z_{\alpha}(u,L(G)) \gamma((u,f_0(u)),(u',v'))}{Z_{\gamma}((u,f_0(u)),L(G)\times L(H))} \\
& = Z_{\alpha}(u,U_0) + Z_{\alpha}(u,L(G)) \\
& = p(u),
\end{align*}
and similarly $\hat{r}(u,f_0(u)) = q(f_0(u))$. 

Since $G$ and $H$ are connected, Proposition~\ref{Prop:couplingredundant} gives that the transition coupling condition implies the marginal coupling condition. Let us now verify the transition coupling condition. Let $(u,v) \in U \times V$. If $\hat{r}(u,v) = 0$, then the transition coupling conditions at $(u,v)$ are trivially satisfied, so we suppose that $\hat{r}(u,v) >0$. First suppose $u \in U_0$, and thus $v = f_0(u) \in V_0$. Then for $u' \in U_0$, by the transition coupling condition for $\gamma_1$ and the facts $r_1(u,f_0(u)) = p_0(u)$ and $\hat{r}(u,f_0(u)) = p(u)$, we have
\begin{align*}
\frac{1}{\hat{r}(u,f_0(u))} \sum_{v' \in V} \hat{\gamma}((u,f_0(u)),(u',v')) & = \frac{1}{\hat{r}(u,f_0(u))} \sum_{v' \in V_0} \hat{\gamma}((u,f_0(u)),(u',v')) \\
& =  \frac{1}{\hat{r}(u,f_0(u))} \sum_{v' \in V_0} Z \gamma_1((u,f_0(u)),(u',v')) \\
& = \frac{Z r_1(u,f_0(u))}{\hat{r}(u,f_0(u)) r_1(u,f_0(u))}  \sum_{v' \in V_0} \gamma_1((u,f_0(u)),(u',v')) \\
& = \frac{Z r_1(u,f_0(u))}{\hat{r}(u,f_0(u))} \cdot \frac{\alpha_0(u,u')}{p_0(u)} \\
& =  \frac{\alpha(u,u')}{p(u)}.
\end{align*}

Next observe that the $\gamma$-transition probability from $(u,f_0(u))$ to $L(G) \times L(H)$ is equal to the $\alpha$-transition probability from $u$ to $L(G)$ (since $\gamma$ satisfies the transition coupling condition and $\spp(r_{\gamma}) \subset (U_0 \times V_0) \cup (L(G) \times L(H))$), and hence $Z_{\alpha}(u,L(G)) = Z_{\gamma}((u,f_0(u)),L(G) \times L(H))$. Using this fact, the transition coupling condition $\gamma$, and the fact that $\hat{r}(u,f_0(u)) = p(u)$, we see that for $u' \in L(G)$, we have
\begin{align*}
   \frac{1}{\hat{r}(u,f_0(u))}&   \sum_{v' \in V} \hat{\gamma}((u,f_0(u)),(u',v')) \\
= \text{ } & \text{ } \frac{1}{\hat{r}(u,f_0(u))} \sum_{v' \in L(H)} \hat{\gamma}((u,f_0(u)),(u',v')) \\
= \text{ } & \text{ } \frac{Z_{\alpha}(u,L(G)) r_{\gamma}(u,f_0(u))}{\hat{r}(u,f_0(u)) Z_{\gamma}((u,f_0(u)),L(G) \times L(H)) } \cdot \frac{1}{r_{\gamma}(u,f_0(u))} \sum_{v' \in L(H)} \gamma((u,f_0(u)),(u',v'))  \\
= \text{ } & \text{ } \frac{Z_{\alpha}(u,L(G)) r_{\gamma}(u,f_0(u))}{p(u) Z_{\gamma}((u,f_0(u)),L(G) \times L(H))} \cdot \frac{\alpha(u,u')}{p(u)} \\
= \text{ } & \text{ } \frac{\alpha(u,u')}{p(u)}.
\end{align*}
We have now shown that $\hat{\gamma}$ satisfies the transition coupling condition at $(u,v) \in U_0 \times V_0$ with respect to $\alpha$. The analogous argument shows that $\hat{\gamma}$ also satisfies the transition coupling condition at $(u,v) \in U_0 \times V_0$ with respect to $\beta$.

Next suppose that $u \in L(G)$ with $\hat{r}(u,v) >0$, and hence $v \in L(H)$. Then there is a unique $u_0 \in U_0$ such that $\alpha(u_0,u) >0$ and a unique $v_0 \in V_0$ such that $\beta(v_0,v) >0$. By the edge preservation property for $\gamma$, the unique $\gamma$-edge connected to $(u,v)$ is $((u,v),(u_0,v_0))$. Then by the definition of $\hat{\gamma}$, we have that $((u,v),(u_0,v_0))$ is the unique $\hat{\gamma}$-edge connected to $(u,v)$. Hence the $\hat{\gamma}$ transition distribution at $(u,v)$ is the trivial coupling of the (trivial) $\alpha$- and $\beta$-transition distributions at $u$ and $v$, respectively, which shows that $\hat{\gamma}$ satisfies the transition coupling condition at $(u,v)$. 

Finally, by the definition of $\hat{\gamma}$, we see that $\spp(\hat{r}) \subset  \spp(r_{\gamma})$. 
Since $\gamma$ has zero OGJ cost under $c_{\Psi}$, we have $\psi_G(u) = \psi_H(v)$ for all $(u,v) \in \spp(r_{\gamma})$, and therefore $\psi_G(u) = \psi_H(v)$ for all $(u,v) \in \spp(\hat{r})$. Hence  $\hat{\gamma}$ has zero OGJ cost under $c_{\Psi}$. 
This concludes the proof of Claim 1. 
\end{proof}

We return to the proof of the proposition.
Fix a pair $(u_0,v_0) \in \spp(\hat{r})$ such that $N_G(u_0) \cap L(G) \neq \emptyset$. By Lemma \ref{Lem:leaves}, we have that $N_H(v_0) \cap L(H) \neq \emptyset$. Claim 1 gives that the hypotheses of Lemma \ref{Lemma:Atlantic} are satisfied for $\hat{\gamma}$. Then by Lemma \ref{Lemma:Atlantic}, there is a bijection $\sigma_{u_0,v_0} : N_G(u_0) \cap L(G) \to N_H(v_0) \cap L(H)$ such that $\hat{\gamma}((u_0,v_0),(u,\sigma_{u_0,v_0}(u)))>0$ for all $u \in N_G(u_0) \cap L(G)$. By definition of $\hat{\gamma}$, we get that $\gamma((u_0,v_0),(u,\sigma_{u_0,v_0}(u))) >0$ for all $u \in N_G(u_0) \cap L(G)$.

Now we extend the map $f_0 : U_0 \to V_0$ to all of $U$ to get an isomorphism. We define a map $f : U \to V$ as follows. For $u \in U_0$, let $f(u) = f_0(u)$. For a leaf $u' \in L(G)$, there exists a unique vertex $u \in U_0$ such that $u' \in N_G(u) \cap L(G)$, and we let $f(u') = \sigma_{u,f_0(u)}(u')$. 
We claim that $f$ is a graph isomorphism from $G$ to $H$. First, note that $f : U \to V$ is bijective, since $f_0 : U_0 \to V_0$ is bijective and each map $\sigma_{u,f_0(u)}$ is bijective. 
Also, we note that $f$ preserves all edge weights and labels within $U_0$, since $f_0$ is a graph isomorphism.
To see that $f$ preserves edges weights and labels at the leaves of $G$, let $u \in U_0$ and $u' \in N_G(u) \cap L(G)$. Then $\gamma((u,f(u)),(u',f(u'))) >0$ (by our choice of $\sigma_{u,f(u)}$), and hence $\psi_G(u') = \psi_H(f(u'))$. Since the primary labels are included in $\psi_G$ and $\psi_H$, we get that $\phi_G(u') = \phi_H(f(u'))$. 
Since the multiweight function is included in $\psi_G$ and $\psi_H$ and $u'$ and $f(u')$ are leaves, we get $\alpha(u,u') = \beta(f(u),f(u'))$. Thus we have that $f$ is a graph isomorphism. Furthermore, by construction of $f$ we have that $\gamma((u,f(u)),(u',f(u'))) >0$ whenever $\alpha(u,u')>0$. 
Let $\gamma'$ be the bijective weight joining associated with $f$, and note that by the previous sentence, we have $\spp(\gamma') \subset \spp(\gamma)$. Since $\gamma$ is extremal, we get that $\gamma = \gamma'$ (by Proposition \ref{Prop:minimalextreme}), and hence $\gamma$ is bijective.
This concludes the proof of Proposition \ref{Prop:treesufficient}.
\end{proof}

\begin{proof}[Proof of Theorem \ref{Thm:trees}]
	Since the multiweight labeling scheme is isomorphism-invariant, Proposition \ref{Prop:necessary} gives that
	$$G \cong H \Longrightarrow \rho_{\Psi}(G,H) = 0.$$
	To show the reverse implication, suppose $\rho_{\Psi}(G,H) = 0$ and let $\gamma$ be an extreme point of $\mathcal{J}^*_{\Psi}(\alpha, \beta)$. By Proposition \ref{Prop:treesufficient}, $\gamma$ is bijective. Thus, by Proposition \ref{Prop:bijectiveiso}, the induced map $f_{\gamma}$ is a graph isomorphism between two graphs. Hence we obtain the detection part. Now, let $G = (U,\alpha,\phi_G)$ and $H = (V,\beta,\phi_H)$ be two isomorphic trees. By Proposition \ref{Prop:necessary}, $\rho_{\Psi}(G,H)=0$. By Proposition \ref{Prop:treesufficient} and Proposition \ref{Prop:bijective_extreme}, the extreme points of $\mathcal{J}_{\Psi}^*(\alpha,\beta)$ coincide with zero-cost bijective weight joinings of $\alpha$ and $\beta$, which gives the identification part. 
\end{proof}

\subsection{Magic decompositions}

Now we move on to our extension principle via magic decompositions. Recall that the definitions and notation necessary for this result are given in Section \ref{Sect:theoresults}. First we recall the statement of the main result.

\mgdecomposition*
Before proving the main theorem, we first establish the following result. 
\begin{restatable}[]{proposition}{mgsufficient}
\label{Prop:mgsufficient}
    Under the hypotheses of Theorem \ref{Thm:mgdecomposition}, suppose that for every $G = (U, \alpha, \phi_G)$ and $H = (V, \beta, \phi_H)$ in $\mathcal{G}$, if $\rho_{\Psi}(G,H) = 0$, then every extreme point of $\mathcal{J}^*_{\Psi}(\alpha, \beta)$ is bijective. Then for every $G = (U, \alpha, \phi_G)$ and $H = (V, \beta, \phi_H)$ in $\mathcal{F}$, if $\rho_{\widetilde{\Psi}}(G,H) = 0$, then every extreme point of $\mathcal{J}^*_{\widetilde{\Psi}}(\alpha, \beta)$ is bijective.
\end{restatable}

\begin{proof}
Let $G = (U, \alpha, \phi_G)$ and $H = (V, \beta, \phi_H)$ be two graphs in $\mathcal{F}$ such that $\rho_{\widetilde{\Psi}}(G,H)=0$. 
Let $\gamma \in \mathcal{J}^*_{\widetilde{\Psi}}(\alpha, \beta)$. 
Since $\mathcal{F}$ has a $\mathcal{G}$-magic decomposition, we may decompose $G$ and $H$ as gluings of $\{ G_i \}_{i=1}^k$ and $\{ H_j \}_{j=1}^{\ell}$ along $M_G$ and $M_H$, respectively, where $G_i \in \mathcal{G}$ for $i \in \{1,\dots,k\}$ and $H_j \in \mathcal{G}$ for $j \in \{1,\dots,\ell\}$.
Here and throughout this proof, we use the notation for gluings established in Section \ref{Sect:mgdecom}. In particular, we let $G_i = (U_i,\alpha_i,\phi_{G_i})$ and $H_j = (V_j,\beta_j,\phi_{H_j})$. We organize the remainder of the proof using a series of claims.

\vspace{1mm}
\noindent
\textit{Claim 1.} Then $\spp(r_{\gamma}) \subset (M_G \times M_H) \cup \bigl((U \setminus M_G) \times (V \setminus M_H) \bigr)$, and $\gamma$ is bijective from $M_G$ to $M_H$, \textit{i.e.},
\begin{itemize}
    \item $\forall u \in M_G$, $\exists! v \in V$ such that $r_{\gamma}(u,v) > 0$, and furthermore $v \in M_H$;
    \item $\forall v \in M_H$, $\exists! u \in U$ such that $r_{\gamma}(u,v) > 0$, and furthermore $u \in M_G$.
\end{itemize}
\begin{proof}
Suppose $(u,v) \in \spp(r_{\gamma})$. Then by Lemma \ref{Lem:samelabel}, we get $\widetilde{\psi}_G(u) = \widetilde{\psi}_H(v)$. By property (1) in the definition of magic decompositions, we have $\widetilde{\psi}_G(u) \in \mathcal{L}^*$ if and only if $u \in M_G$, and similarly $\widetilde{\psi}_H(v)\in \mathcal{L}^*$ if and only if $v\in M_H$. Thus, since $\widetilde{\psi}_G(u) = \widetilde{\psi}_H(v)$, we have $u \in M_G$ if and only if $v \in M_H$, which establishes the first statement of the claim.

Now suppose $u \in M_G$. Since $G$ is fully supported, we have $p(u) >0$. Hence, as $r_{\gamma}$ is a coupling of $p$ and $q$, there exists $v \in V$ such that $r_{\gamma}(u,v)>0$. 
Suppose that there exist vertices $v,v' \in V$ such that $r_{\gamma}(u,v) >0$ and $r_{\gamma}(u,v')>0$. Then $\widetilde{\psi}_H(v) = \widetilde{\psi}_G(u) = \widetilde{\psi}_H(v')$. Since $u \in M_G$, the first part of the claim gives that $v,v' \in M_H$. 
By property (2) in the definition of magic decomposition, $\widetilde{\psi}_H$ is injective on $M_H$, and therefore we must have $v = v'$. We conclude that for all $u \in M_G$, there exists a unique $v \in V$ such that  $r_{\gamma}(u,v) > 0$, and furthermore $v \in M_H$. An analogous argument for $v \in M_H$ finishes the proof of the claim.
\end{proof}

Note that by the first part of Claim 1, if $U = M_G$, then we must have $V = M_H$. Next observe that if $U = M_G$ and $V = M_H$, then by the second part of Claim 1, every $\gamma \in \mathcal{J}^*_{\widetilde{\Psi}}(\alpha,\beta)$ is bijective, and hence every extreme point of $\mathcal{J}^*_{\widetilde{\Psi}}(\alpha, \beta)$ is bijective. 

Now we assume that $U \setminus M_G$ is nonempty (and hence $V \setminus M_H$ is also nonempty). Fix $i \in \{1,\dots,k\}$ and $j \in \{1,\dots,\ell\}$. Let $L_i = f_i^{-1}(M_G)$ and $K_j = g_j^{-1}(M_H)$ (which are the subsets of leaves of $G_i$ and $H_j$ that are glued to $M_G$ and $M_H$, respectively). Let $U_i' = U_i \setminus L_i$ and $V'_j = V_j \setminus K_j$. We define a function $\hat{\gamma}_{ij} : (U_i \times V_j)^2 \to \mathbb{R}$ as follows. First, if $(u,v), (u',v') \in U'_i \times V'_j$, then we let $\hat{\gamma}_{ij}((u,v),(u',v')) = \gamma((u,v),(u',v'))$. Next if $(u,v) \in U'_i \times V'_j$ and $(u',v') \in L_i \times K_j$ satisfy $u = b(u')$, $f_i(u') = u''$, $v =  b(v')$ and $g_j(v') = v''$, then we let $\hat{\gamma}_{ij}((u,v),(u',v')) = \gamma((u,v),(u'',v''))$ and $\hat{\gamma}_{ij}((u',v'),(u,v)) = \gamma((u'',v''),(u,v))$.
All other values of $\hat{\gamma}_{ij}$ are defined to be zero. Next let
\begin{align*}
    \pi_{ij} = \sum_{(u,v),(u',v') \in U_i \times V_j} \hat{\gamma}_{ij}((u,v),(u',v')).
\end{align*}
Notice that $\pi_{ij} >0$ if and only if there exists $(u,v) \in U'_i \times V'_j$ such that $r_{\gamma}(u,v) >0$. 
For all $i,j$ such that $\pi_{ij}>0$, we let $\gamma_{ij} = (\pi_{ij})^{-1} \hat{\gamma}_{ij}$, and we note that $\gamma_{ij}$ is a weight function by construction. Let $r_{ij}$ denote the marginal function associated with $\gamma_{ij}$. 

\vspace{2mm}
\noindent
\textit{Claim 2.} If $\pi_{ij} > 0$, then $\gamma_{ij}$ is a weight joining of $\alpha_i$ and $\beta_j$ with $\langle c_{\Psi}, r_{{ij}} \rangle = 0$.
\begin{proof}
Suppose $\pi_{ij}>0$. 
As $\gamma_{ij}$ is a weight function and the graphs $G_i$ and $H_j$ are connected (by the definition of gluings), if we show that $\gamma_{ij}$ satisfies the transition coupling condition, then $\gamma_{ij}$ also satisfies the marginal coupling condition (by Proposition \ref{Prop:couplingredundant}). Let us now show that the transition coupling condition is satisfied. 

First, note that $\spp(r_{ij}) \subset (L_i \times K_j) \cup (U'_i \times V'_j)$ by definitions of $\hat{\gamma}_{ij}$ and $\gamma_{ij}$.
Let $(u,v) \in U_i \times V_j$. The transition coupling conditions are trivially satisfied at $(u,v)$ if $r_{ij}(u,v) = 0$, so we suppose $r_{ij}(u,v) >0$. Next, suppose that $(u,v) \in L_i \times K_j$. As $u$ is a leaf in $U_i$ and $v$ is a leaf in $V_j$, the transition distribution at $(u,v)$ is supported on a single edge, which gives the trivial transition coupling of the single-edge transition distributions at $u$ and $v$. Hence the transition coupling condition is satisfied at $(u,v)$. Now suppose that $(u,v) \in U'_i \times V'_j$. Since $\gamma_{ij}((u,v),\cdot)$ is equal to a scalar multiple of $\gamma((u,v),\cdot)$ and $\gamma$ satisfies the transition coupling condition at $(u,v)$, we conclude that $\gamma_{ij}$ satisfies the transition coupling condition at $(u,v)$. Thus we conclude that $\gamma_{ij}$ is a weight joining of $\alpha_i$ and $\beta_j$. 

Now let $(u,v) \in U_i \times V_j$ with $r_{ij}(u,v) >0$. First suppose $(u,v) \in L_i \times K_j$. By definition of $\gamma_{ij}$, we have that $r_{\gamma}(f_i(u),g_j(v))>0$. 
Since $\gamma$ has zero OGJ cost, by Lemma \ref{Lem:samelabel}, we get $\widetilde{\psi}_G(f_i(u)) = \widetilde{\psi}_H(g_j(v))$. Hence, by property (4) of magic decompositions and the fact $(u,v) \in L_i \times K_j$, we obtain $\psi_{G_i}(u) = \psi_{H_j}(v)$. 
Now suppose $(u,v) \in U'_i \times V'_j$. Arguing similarly, we see that $r_{\gamma}(u,v)>0$, and $\widetilde{\psi}_G(u) = \widetilde{\psi}_H(v)$. Then by property (3) of magic decompositions, we obtain that $\psi_{G_i}(u) = \psi_{H_j}(v)$.
We conclude that $\langle c_{\Psi}, r_{ij} \rangle = 0$. 
\end{proof}

Consider $i$ and $j$ such that $\pi_{ij}>0$. By Claim 2, we have that $\rho_{\Psi}(G_i,H_j) = 0$ and  $\gamma_{ij} \in \mathcal{J}^*_{\Psi}(\alpha_i, \beta_j)$. Then by our hypothesis about OGJ on the family $\mathcal{G}$, every extreme point of $\mathcal{J}^*_{\Psi}(\alpha_i,\beta_j)$ is bijective. Hence, by Proposition \ref{Prop:extremebijective}, there exists a bijective weight joining $\gamma_{ij}^*$ such that $\spp(\gamma_{ij}^*) \subset \spp(\gamma_{ij})$. Let $f_{ij} : U_i \to V_j$ be the induced map associated with $\gamma_{ij}^*$, and note that by Proposition \ref{Prop:bijectiveiso}, $f_{ij}$ is a graph isomorphism from $G_i$ to $H_j$. Additionally, since $\spp(\gamma_{ij}^*) \subset \spp(\gamma_{ij})$, if $\alpha_i(u,u') >0$, then $\gamma_{ij}((u, f_{ij}(u)),((u', f_{ij}(u')))>0$.

%

\vspace{2mm}
\noindent
\textit{Claim 3.} There is a bijection $\sigma: \{1,\ldots,k\} \to \{1,\ldots,\ell\}$ such that $\pi_{i\sigma(i)} > 0$ for every $i \in \{1,\dots,k\}$.
\begin{proof}
First, note that for each $i \in \{1,\ldots,k\}$, there is some $j \in \{1,\ldots,\ell\}$ such that $\pi_{ij} > 0$. Indeed, let $i \in \{1, \ldots, k\}$. Pick $(u,u') \in U'_i \times U_i$ such that $\alpha_i(u,u') > 0$. Since $\gamma$ is a weight joining of $\alpha$ and $\beta$ and is bijective from $M_G$ to $M_H$ (by Claim 1), by Proposition \ref{Prop:realcoupling} there is some $(v,v') \in (V \setminus M_H) \times V$ such that $\gamma((u,v),(u',v')) > 0$. Hence, there is some $j \in \{1,\ldots,l\}$ such that $v \in V'_j$. Therefore $(u,v) \in U'_i \times V'_j$ and $r_{\gamma}(u,v)>0$, which gives that $\pi_{ij}>0$.

Now, we shall use Hall's marriage theorem to construct the bijection $\sigma$. We construct a bipartite graph as follows. Let $\mathcal{U} = \{U'_1,\ldots,U'_k\}$, and let $\mathcal{V} = \{V'_1,\ldots,V'_{\ell}\}$. We let $\mathcal{U} \cup \mathcal{V}$ be the vertex set of our bipartite graph. We draw an edge from $U'_i$ to $V'_j$ whenever $\pi_{ij} >0$. 
By the previous paragraph, the edge set is nonempty. Let $\mathcal{S} = \{U'_i \}_{i \in I}$ be an arbitrary subset of $\mathcal{U}$, and let $N(\mathcal{S}) = \{V'_j\}_{j \in J}$ denote its neighborhood in the bipartite graph. Let $S_{G_i} = \sum_{u \in U'_i} p(u)$ and $S_{H_j} = \sum_{v \in V'_j} q(v)$. Then by the marginal coupling condition, the bijective property of $\gamma$ from $M_G$ to $M_H$, and the definition of $N(\mathcal{S})$, we obtain  
\begin{align*}
    \sum_{i \in I} S_{G_i} &= \sum_{i \in I} \sum_{u \in U'_i} \sum_{v \in V} r_{\gamma}(u,v) = \sum_{i \in I} \sum_{u \in U'_i} \sum_{j \in J} \sum_{v \in V'_j} r_{\gamma}(u,v)  \\
    &\le \sum_{j \in J} \sum_{v \in V'_j} \sum_{u \in U} r_{\gamma}(u,v) = \sum_{j \in J} S_{H_j}. 
\end{align*}
Suppose for contradiction that  $|J| < |I|$. Note that if $\pi_{ij} > 0$, then property (5) of magic decompositions 
gives that $\mu_i = \nu_j$, which then yields $S_{G_i} = S_{H_j}$. 
Since $|J| < |I|$, there is an injective map $\tau : J \to I$ such that $\tau(J) \subsetneq I$ and for each $j \in J$, we have $\pi_{\tau(j) j} >0$. Hence $\mu_{\tau(j)} = \nu_j$ for all $j \in J$. Using these facts, we see that
\begin{equation*}
 \sum_{i \in I} S_{G_i} > \sum_{i \in \tau(J)} S_{G_i} = \sum_{j \in J} S_{G_{\tau(j)}} = \sum_{j \in J} S_{H_j},
\end{equation*}
which contradicts the previous display. Thus, $|I| \le |J|$. By Hall's Marriage Theorem, there exists an injection $\sigma: \{1,\ldots,k\} \to \{1,\ldots,\ell\}$ satisfying $\pi_{i \sigma(i)} > 0$ for all $i \in \{1,\dots,k\}$.  
By repeating the above argument with the roles of $G_i$ and $H_j$ reversed, there is also an injective map $\sigma' : \{1,\dots,\ell\} \to \{1,\dots,k\}$, which shows that $k = \ell$.
Thus, the injective map $\sigma : \{1,\dots,k\} \to \{1,\dots,\ell\}$ is in fact bijective, which completes the proof of the claim. 
\end{proof}

\vspace{1mm}
\noindent
\textit{Claim 4.} There is a graph isomorphism $f: U \to V$ from $G$ to $H$ such that for all $u,u' \in U$ with $\alpha(u,u') > 0$, we have $\gamma((u,f(u)),(u',f(u'))) > 0$. 
\begin{proof}
Indeed, by Claim 3, for each $i \in \{1,\dots,k\}$, we have that $\pi_{i\sigma(i)} > 0$. Recall that if $\pi_{ij}>0$, then we have defined a graph isomorphism $f_{ij}$ from $G_i$ to $H_j$. Hence, for each $i \in \{1,\dots,k\}$, we have a graph isomorphism from $G_i$ to $H_{\sigma(i)}$ given by  $f_{i\sigma(i)}$. Now we define $f : U \to V$ as follows. For all $i \in \{1,\dots,k\}$, we let $f|_{U'_i} = f_{i\sigma(i)}$. Since $\gamma$ is bijective from $M_G$ to $M_H$ (by Claim 1), we may define a bijection $g : M_G \to M_H$ by the condition that $r_{\gamma}(u,g(u)) >0$. Then we let $f|_{M_G} = g$. This defines a bijection $f: U \to V$. Note that for all $u \in U$, we have $r_{\gamma}(u,f(u))>0$, which gives that $\widetilde{\psi}_G(u) = \widetilde{\psi}_H(f(u))$ by Lemma \ref{Lem:samelabel}. Therefore $f$ preserves the primary vertex labels (\textit{i.e.}, $\phi_G(u) = \phi_H(f(u))$ for all $u \in U$). 

Now consider an edge in $G$, \textit{i.e.}, let $u,u' \in U$ with $\alpha(u,u') > 0$. After accounting for the symmetry of $\alpha$, we note that there are three types of edges in $G$: (i) $u,u' \in U'_i$ for some $i \in \{1,\dots,k\}$, (ii) $(u,u') \in U'_i \times M_G$ for some $i \in \{1,\dots,k\}$, and (iii) $(u,u') \in M_G \times M_G$. First consider case (i). Then $\alpha_i(u,u') = \beta_{\sigma(i)}(f_{i\sigma(i)}(u),f_{i\sigma(i)}(u'))$, since $f_{i\sigma(i)}$ is a graph isomorphism from $G_i$ to $H_{\sigma(i)}$. Then by definition of $f$ and property (5) of magic decomposition, we see that $\alpha(u,u') = \beta(f(u),f(u'))$. Now consider case (ii), and let $j = \sigma(i)$. Since $u' \in M_G$, there is some $u'' \in L_i$ such that the gluing map $f_i$ satisfies $f_i(u'') = u'$. Let $v'' = f_{ij}(u'')$. Then by definition of $\gamma_{ij}$, we have $\gamma((u,f(u)),(u',g_j(v''))) = \pi_{ij} \gamma_{ij}((u,f(u)),(u'',v''))$. Since $f(u) = f_{ij}(u)$, $v'' = f_{ij}(u'')$, and $f_{ij}$ satisfies $\gamma_{ij}((u,f_{ij}(u)),(u'',f_{ij}(u''))) > 0$, we obtain $\gamma((u,f(u)),(u',g_j(v'')))>0$. Since $u' \in M_G$ and $\gamma$ is bijective from $M_G$ to $M_H$, we conclude that $g_j(v'') = g(u') = f(u')$. Furthermore, by property (5) of magic decomposition and the fact $f_{ij}$ is a graph isomorphism, we obtain that $\alpha(u,u') = \beta(f(u),f(u'))$. Finally, consider case (iii). By Proposition \ref{Prop:realcoupling}, there exists $(v,v') \in V \times V$ such that $\gamma((u,v),(u',v'))>0$. Since $\gamma$ is bijective from $M_G$ to $M_H$, we conclude that $v = f(u)$ and $v' = f(u')$. Again using Proposition \ref{Prop:realcoupling}, we must have that $\alpha(u,u') = \gamma((u,v),(u',v')) = \beta(v,v') = \beta(f(u),f(u'))$. 
This concludes the proof of the claim.
\end{proof}

Let us now finish the proof of the proposition.
Let $\gamma'$ be the bijective weight joining corresponding to $f$ (given by Claim 4), and note that $\spp(\gamma') \subset \spp(\gamma)$. Then by Proposition \ref{Prop:extremebijective}, every extreme point of $\mathcal{J}^*_{\Psi}(\alpha,\beta)$ is bijective. This concludes the proof of Proposition \ref{Prop:mgsufficient}.
\end{proof}

\begin{proof}[Proof of Theorem \ref{Thm:mgdecomposition}]
	Let $G, H$ be two graphs in $\mathcal{F}$. The direction 
	$$G \cong H \Longrightarrow \rho_{\Psi}(G,H) = 0$$
	follows from Proposition \ref{Prop:necessary}. Let $\gamma$ be an extreme point of $\mathcal{J}^*(\alpha, \beta)$. By the hypotheses and Proposition \ref{Prop:mgsufficient}, $\gamma$ is bijective. Thus, by Proposition \ref{Prop:bijectiveiso}, the induced map $f_{\gamma}$ is a graph isomorphism between $G$ and $H$. This establishes the detection part. Now suppose OGJ with cost $c_{\Psi}$ identifies isomorphism for $\mathcal{G}$. Thus, we obtain for every $G = (U, \alpha, \phi_G)$ and $H = (V, \beta, \phi_H)$ in $\mathcal{G}$, if $\rho_{\Psi}(G,H) = 0$, then every extreme point of $\mathcal{J}^*_{\Psi}(\alpha, \beta)$ is bijective.
    Let $G = (U,\alpha,\phi_G)$ and $H = (V,\beta,\phi_H)$ be two isomorphic graphs in $\mathcal{F}$. By Proposition \ref{Prop:necessary}, $\rho_{\widetilde{\Psi}}(G,H)=0$. Then by Proposition \ref{Prop:mgsufficient} and Proposition \ref{Prop:bijective_extreme}, the extreme points of $\mathcal{J}_{\widetilde{\Psi}}^*(\alpha,\beta)$ coincide with zero-cost bijective weight joinings of $\alpha$ and $\beta$, which gives the identification part.  
\end{proof}

\DisjointUnions*
\begin{proof}
    Since $\mathcal{F}$ consists of finite disjoint unions of graphs in $\mathcal{G}$, we can choose $\mathcal{L}^* = \emptyset$. Then for each $G \in \mathcal{F}$, we can choose $M_G = \emptyset$. Because the augmented labeling scheme $\widetilde{\Psi}$ satisfies the magic decomposition conditions (3) and (5) and $M_G = \emptyset$, $\mathcal{F}$ has a $\mathcal{G}$-magic decomposition. By Theorem \ref{Thm:mgdecomposition}, we obtain the conclusion. 
\end{proof}

\subsection{Relation to Weisfeiler-Leman test}  
Lastly, we establish the relationship between OGJ and the Weisfeiler-Leman test. 
Recall that $d_{\OTC}$ and $d_{\WL}$ are the NetOTC distance and the WL distance introduced in Section \ref{Sect:RelationWLtest}, respectively. Recall the statement of the following proposition. 
\WLNetOTCOGJ*
\begin{proof}
	Let $X$ and $Y$ be the associated reversible Markov chains on $G$ and $H$, respectively. Let $(\widetilde{X},\widetilde{Y})$ be any reversible Markov chain associated with a graph joining of $G$ and $H$. Then $(\widetilde{X},\widetilde{Y})$ is also a transition coupling (see Definition 3 in \cite{o2022optimal}) of $X$ and $Y$. Thus, we obtain $d_{\OTC}(G,H) \le \rho_c(G,H)$. On the other hand, by Proposition 6 in \cite{Brugere2023DistancesFM}, we have $d_{\WL}(G,H) \le d_{\OTC}(G,H)$. Combining two inequalities, we obtain the result. 
\end{proof}
Next we show that if the WL test can distinguish two unweighted labeled graphs $G$ and $H$, then OGJ can also distinguish them with the labeling scheme $\Psi$. Let $G = (U,\alpha,\phi)$ be an unweighted labeled graph. Recall that $\phi' : U \to \mathbb{Z}^2$ is defined by $\phi'(u) = (\deg(u),|U|)$. For $\delta \in (1/2,1)$, $\hat{G}_{\delta}$ is the weighted graph corresponding to the $\delta$-lazy random walk on $G$ with augmented label function $\psi = (\phi,\phi')$. 
\WLOGJ*
\begin{proof}
Firstly, by Proposition \ref{Prop:WLNetOTCOGJ}, we have
\begin{align} \label{Ineq}
	d_{\WL}(\hat{G}_{\delta},\hat{H}_{\delta}) \le d_{\OTC}(\hat{G}_{\delta},\hat{H}_{\delta}) \le \rho_{\Psi}(\hat{G}_{\delta},\hat{H}_{\delta}). 
\end{align}
Next, Proposition 3.3 in \cite{chen2022weisfeiler} gives us that the WL test determines $G$ and $H$ are not isomorphic if and only if $d_{\WL}(\hat{G}_{\delta},\hat{H}_{\delta}) > 0$. By \eqref{Ineq}, if $d_{\WL}(\hat{G}_{\delta},\hat{H}_{\delta})>0$ then $\rho_{\Psi}(\hat{G}_{\delta},\hat{H}_{\delta})>0$. Putting these two statements together, we conclude that if the WL test determines that $G$ and $H$ are not isomorphic, then $\rho_{\Psi}(\hat{G}_{\delta},\hat{H}_{\delta})>0$.
\end{proof}


\bibliographystyle{plain}      
\bibliography{refs}

\footnotesize
\vskip 0.5cm

{\sc Phuong N. Ho\`ang, Department of Mathematics and Statistics, UNC Charlotte}

{\it Email address}: \href{mailto:phoang3@charlotte.edu}{phoang3@charlotte.edu} 
\vskip 0.2cm

{\sc Kevin McGoff, Department of Mathematics and Statistics, UNC Charlotte}

{\it Email address}: \href{mailto:kmcgoff1@charlotte.edu}{kmcgoff1@charlotte.edu} 
\vskip 0.2cm

{\sc Andrew B. Nobel, Department of Statistics and Operations Research, UNC Chapel Hill}

{\it Email address}: \href{mailto:nobel@email.unc.edu}{nobel@email.unc.edu} 
\vskip 0.2cm

{\sc Yang Xiang, Department of Statistics and Operations Research, UNC Chapel Hill}

{\it Email address}: \href{mailto:xya@unc.edu}{xya@unc.edu} 
\vskip 0.2cm

{\sc Bongsoo Yi, Google}

{\it Email address}: \href{mailto:bongsooyi0426@gmail.com}{bongsooyi0426@gmail.com} 


\end{document}